\documentclass[a4paper, reqno]{amsart}
\usepackage{amssymb,amscd,amsfonts,amsbsy,nicefrac,mathtools}
\usepackage[utf8]{inputenc}
\usepackage{amsmath,amsthm,amssymb,amsfonts,enumitem}
\usepackage{mathrsfs}
\usepackage[nobysame]{amsrefs}
\usepackage{hyperref}
\usepackage[pdftex]{graphicx,color}
\usepackage{graphicx}
\usepackage{marvosym}
\usepackage{scalerel}
\usepackage{bbm}
\usepackage{bm}
\usepackage{setspace}
\usepackage{tikz}
\usepackage{soul}
\usepackage[normalem]{ulem}

\usepackage[T1]{fontenc}

\usepackage{fancyhdr} 
\usepackage[a4paper,tmargin=3cm,bmargin=3cm,lmargin=3cm,rmargin=3cm]{geometry}
\newtheorem{thm}{Theorem}[section]
\newtheorem{rem}[thm]{Remark}
\newtheorem{lem}[thm]{Lemma}

\newtheorem{cor}[thm]{Corollary}
\newtheorem{prop}[thm]{Proposition}
\newtheorem{defn}[thm]{Definition}

\newcommand{\R}{\mathbb{R}}
\newcommand{\bbR}{\mathbb{R}}
\newcommand{\C}{\mathbb C}
\newcommand{\bbC}{\mathbb{C}}

\newcommand{\Z}{\mathbb{Z}}
\newcommand{\bbZ}{\mathbb{Z}}
\newcommand{\nat}{\mathbb{Z}_{+}}
\newcommand{\N}{\mathbb{N}}
\newcommand{\bbN}{\mathbb{N}}

\newcommand{\supp}{\operatorname{supp}}

\newcommand{\dd}{\,\mathrm{d} }

\newcommand{\abs}[1]{\left|#1\right|}
\newcommand{\set}[1]{\left\{#1\right\}}
\newcommand{\brkt}[1]{\left(#1\right)}
\newcommand{\jap}[1]{\left\langle #1\right\rangle}

\newcommand{\m}[1]{\begin{equation*}
\begin{split}
#1
\end{split}
\end{equation*}}
\newcommand{\nm}[2]{\begin{equation}\label{#1}
\begin{split}
#2
\end{split}
\end{equation}}
\begin{document}

\title{{Christoffel transform and multiple orthogonal polynomials}}
\date{\today}
\author{Rostyslav Kozhan$^1$}
%\email{kozhan@math.uu.se}
\author{Marcus Vaktnäs$^{2}$}
%\email{marcus.vaktnas@math.uu.se}
\address{$^{1}$Department of Mathematics, Uppsala University, S-751 06 Uppsala, Sweden, kozhan@math.uu.se}
%\address{{\it{$^*$Correspondence to be sent to:} marcus.vaktnas@math.uu.se}}
\address{$^{2}$Department of Mathematics, Uppsala University, S-751 06 Uppsala, Sweden, marcus.vaktnas@math.uu.se}
\begin{abstract}
    %We study multiple orthogonal polynomials on the real line when one of the measures is supported on the finitely many points. We present the connection to the Christoffel/Geronimus spectral transformations and generalize the Christoffel formulas to the multiple orthogonality setting.
We investigate multiple orthogonal polynomials associated with the system of measures obtained by applying a Christoffel transform to each of the orthogonality measures. We present an algorithm for computing the transformed recurrence coefficients and determinantal formulas for the transformed multiple orthogonal polynomials of type I and type II.

We apply these results to show that zeros of multiple orthogonal polynomials of an Angelesco or an AT system interlace with the zeros of the polynomials corresponding to its one-step Christoffel transform. 
This allows us to prove a number of  interlacing properties satisfied by the
multiple orthogonality analogues of classical orthogonal polynomials. For the discrete polynomials, this also produces an estimate on the smallest distance between consecutive zeros.

We also identify a connection between the Christoffel transform of orthogonal polynomials and 
%multiple orthogonal polynomials with respect to systems of measures 
multiple orthogonality systems
containing a finitely supported measure.
    %We study orthogonal polynomials associated with a measure obtained by a multiplication  by a polynomial factor. %(multi-step Christoffel and Geronimus transforms, respectively).
    %We show that the familiar algorithms for the computation of the transformed Jacobi coefficients can be viewed as a special case of the compatibility relations for the nearest neighbour recurrence coefficients for multiple orthogonal polynomials.
     In consequence, the compatibility relations for the nearest neighbour recurrence coefficients provide a new algorithm for the computation of the Jacobi coefficients of the one-step or multi-step Christoffel transforms.
     %we show that a mild modification of the algorithm~\cite[Thm~3.1]{Computing NNR} based on the compatibility relations for the nearest neighbour recurrence coefficients can be used to compute the Jacobi coefficients of the one-step or multi-step Christoffel transforms.
\end{abstract}
\maketitle

\section{Introduction}
Let $\mu$ be a positive Borel measure on the real line with all the moments $c_j:=\int x^j \dd\mu(x)$ finite. Denote $P_n(x)$, $n\in\bbN :=\{k\in\bbZ: k\ge 0\}$, to be the  monic orthogonal polynomial of degree $n$ with respect to the inner product
\begin{equation}\label{eq:inner}
    \jap{f(x),g(x)} = \int_\bbR f(x) g(x) \dd\mu(x).
\end{equation}
These polynomials satisfy the famous three-term recurrence relation
\begin{equation}\label{eq:three-term0}
xP_n(x) = P_{n+1}(x) + b_nP_n(x) + a_nP_{n-1}(x)
\end{equation}
for some $a_n>0$ and $b_n\in\bbR$, called the Jacobi coefficients of $\mu$. % (in the case $n = 0$ we formally take $a_0 = 0$ and, say $P_{-1} = 0$).

Given a point $z_0\in\bbR$, a one-step Christoffel transform of $\mu$ is a new (potentially signed) measure $\widehat{\mu}$ defined by 
\begin{equation}\label{eq:ChrM10}
    \int f(x)\dd\widehat\mu(x) = 
    %\frac{1}{Z}
    \int f(x) ({x-z_0})\dd\mu(x).
    %, \qquad \mbox{ where } Z= \int ({x-z_0})d\mu(x).
\end{equation}
More generally, given a polynomial $\Phi(x)=\prod_{j=1}^m (x-z_j)$ we define the multi-step Christoffel transform to be $\widehat{\mu}$ given by
\begin{equation}\label{eq:ChrM0}
    \int f(x)\dd\widehat\mu(x)= %\frac{1}{Z}
    \int f(x){\Phi(x)}\dd\mu(x),
    %, \qquad \mbox{ where } Z= \int {\Phi(x)}d\mu(x).
\end{equation}
which, of course, can be viewed as the one-step transform repeated $m$ times for each of the roots of $\Phi(x)$.

It is a natural question to understand the relationship between the orthogonal polynomials $(P_n(x))_{n=0}^\infty$ of $\mu$  and $(\widehat{P}_n(x))_{n=0}^\infty$ of $\widehat\mu$. In the one-step case~\eqref{eq:ChrM10} %($\deg\Phi = 1$) 
the following simple relation holds true:
\nm{eq:one step ct0}{\widehat{P}_n(x) = \frac{1}{x-z_0} \left({P}_{n+1}(x) - \frac{{P}_{n+1}(z_0)}{{P}_{n}(z_0)} {P}_{n}(x)\right).}
More generally, if $\deg\Phi=m$, then $\widehat{P}_n$ can be expressed in terms of $P_n,P_{n+1},\ldots, P_{n+m}$ using the Christoffel determinantal formula, shown in~\cite{Christoffel}.%, see~\eqref{eq:christoffel theorem} below. 

Another matter of interest here is the relationship between the Jacobi coefficients of $\widehat\mu$ and of $\mu$. For the one-step case $\deg\Phi=1$ there 
are a number of closely related algorithms (among them are the $qd$ algorithm of Rutishauser~\cite{QD}, Galant's~\cite{Galant}, and Gautschi's~\cite{Gautschi book}) 
that allow to compute $\widehat{a}_n$'s and $\widehat{b}_n$'s recursively from $a_n$'s and $b_n$'s. %We present one of them below~\eqref{Galant algorithm}. 
There exist explicit algorithms for the quadratic factors, $\deg\Phi = 2$, see, e.g., ~\cite[Sect 2.4.3]{Gautschi book}. For the general case $\deg\Phi = m$, one typically applies the one-step or two-step algorithms repeatedly.

The Christoffel transform, as well as the closely connected topic of the Darboux transformations, is a very well-studied topic both in pure and applied mathematics, see, e.g.,~\cite{BaiDer,BueMar,Chihara,Christoffel,Mar91,SpiZhe,Zhe97} and~\cite{Bueno and Dopico,Galant,Gal92,Gautschi,Gautschi book,Golub and Kautsky,GolKau83,KauGol}.

%is an algorithm of Gautschi~\cite{Gautschi, Gautschi book} that allows to compute $\widehat{a}_n$'s and $\widehat{b}_n$'s recursively from $a_n$'s and $b_n$'s, see~\eqref{Galant algorithm} below. A modification of this for $\deg\Phi = 2$ can be found~\cite{otherGautschi}. For the general case $\deg\Phi = m$, one needs to apply the one-step or two-step algorithms repeatedly.

\medskip

Now let us introduce multiple orthogonal polynomials with respect to a system of two measures $(\mu_1,\mu_2)$ on $\bbR$ (we work in the more general setting of $r$ measures from Section~\ref{ss:prelim} onwards). For $(n_1,n_2)\in \bbN^2$, let $P_{n_1,n_2}(x)$ be a non-zero monic polynomial of degree $n_1+n_2$ satisfying
\begin{align}
    \label{eq0:orth1} \int_\bbR P_{n_1,n_2}(x)x^p \dd \mu_1(x) &= 0,\qquad p = 0,1,\dots,n_1-1, \\
     \label{eq0:orth2} \int_\bbR P_{n_1,n_2}(x)x^p \dd \mu_2(x) &= 0,\qquad p = 0,1,\dots,n_2-1. 
\end{align}
$P_{n_1,n_2}(x)$ is then called a multiple orthogonal polynomial at the multi-index $(n_1,n_2)$. We say that $(n_1,n_2)$ is normal for $(\mu_1,\mu_2)$ if such $P_{n_1,n_2}$ exists and is unique.

Assuming sufficiently many indices are normal, these polynomials satisfy~\cite{NNR,Ismail} the nearest neighbour recurrence relations (compare with~\eqref{eq:three-term0})
\begin{align}
   \label{eq0:NNRR1} xP_{n_1,n_2}(x) = P_{n_1+1,n_2}(x) + b_{n_1,n_2;1}P_{n_1,n_2}(x) +a_{n_1,n_2;1}P_{n_1-1,n_2}(x)+a_{n_1,n_2;2}P_{n_1,n_2-1}(x), \\
   \label{eq0:NNRR2} xP_{n_1,n_2}(x) = P_{n_1,n_2+1}(x) + b_{n_1,n_2;2}P_{n_1,n_2}(x) +a_{n_1,n_2;1}P_{n_1-1,n_2}(x)+a_{n_1,n_2;2}P_{n_1,n_2-1}(x).
\end{align}
Subtracting \eqref{eq0:NNRR1} and \eqref{eq0:NNRR2} we can also obtain 
\begin{equation}\label{eq0:CCPoly}
    P_{n_1+1,n_2}(x) - P_{n_1,n_2+1}(x)  = (b_{n_1,n_2;2}-b_{n_1,n_2;1})P_{n_1,n_2}(x).
\end{equation}

The coefficients $a_{n_1,n_2;1},a_{n_1,n_2;2},b_{n_1,n_2;1},b_{n_1,n_2;2}$ in \eqref{eq0:NNRR1} and \eqref{eq0:NNRR2} are called the nearest neighbour coefficients. 
These coefficients satisfy a set of partial difference equations
~\cite[Eq. (3.6)--(3.8)]{NNR} 
%(see \eqref{eq:CC1}--\eqref{eq:CC3} below) 
which we will call the compatibility conditions, or CC, for short. %nearest neighbour compatibility conditions (CC, for short). %,  explicitly written below 
\cite{Computing NNR} showed that these equations provide an algorithm that allows to recursively compute all the nearest neighbour recurrence coefficients $a_{n_1,n_2;1},a_{n_1,n_2;2},b_{n_1,n_2;1},b_{n_1,n_2;2}$ from the Jacobi coefficients of $\mu_1$ and $\mu_2$ (that is, from $a_{n_1,0;1},a_{0,n_2;2},b_{n_1,0;1},b_{0,n_2;2}$).%, see \cite{Computing NNR}. or Section~\ref{ss:NNCC} below for the details.

\medskip

The central idea of our paper is the simple observation that if $\mu_2$ is supported on $N$ distinct points $\{z_j\}_{j=1}^N$, then the multiple orthogonal polynomial $P_{n,N}(x)$ of the system $(\mu_1,\mu_2)$ coincides with $\widehat{P}_n(x) \Phi(x)$, where $\Phi(x) = \prod_{j=1}^N (x-z_j)$ and $\widehat{P}_n(x)$ is the $n$-th orthogonal polynomial of the Christoffel transform $\widehat{\mu}_1$ ~\eqref{eq:ChrM0} of $\mu_1$ corresponding to the polynomial $\Phi(x)$. Indeed, $\widehat{P}_n(x) \Phi(x)$ is monic, has the right degree, trivially satisfies~\eqref{eq0:orth1}--\eqref{eq0:orth2}, and it only remains to resolve the issue of uniqueness, which we do in Theorem~\ref{thm:ct moprl}. 

In particular, the nearest neighbour recurrence relation~\eqref{eq0:NNRR1} along locations $\big\{ (n,N): n\in\bbN\big\}$ reduces to the three-term recurrence relation for $\widehat{\mu}_1$ (one should observe that $a_{n,N;2} = 0$ for all $n$, see Theorem~\ref{thm:zero coefficients finite systems}), and the nearest neighbour coefficients along these locations coincide with the Jacobi coefficients of $\widehat{\mu}_1$. 

Taking the simplest case $N=1$, one realizes that Christoffel's formula~\eqref{eq:one step ct0} is
just~\eqref{eq0:CCPoly}, while Gautschi's algorithm~\cite{Gautschi book} for computing the Jacobi coefficients of the Christoffel transform is effectively the CC algorithm of~\cite{Computing NNR} (after minor modifications related to restricting the coefficients to the strip $\bbN\times \{0,1\}$, see Section~\ref{ss:NNCC}). Furthermore, one can show that the well-known Gauss--Radau quadrature rule for $\mu_1$ is just the multiple Gauss quadrature rule for $(\mu_1,\mu_2)$ with $N=|\supp\mu_2|=1$ (the Gauss--Lobatto rule corresponds to $N=|\supp\mu_2|=2$). 

%The first new non-trivial result is the observation that the modified CC algorithm (see Section~\ref{ss:NNCC}) for arbitrary $N$ 
For any $N$ the modified CC algorithm (see Section~\ref{ss:NNCC}) therefore provides an algorithm for computation of the Jacobi coefficients of the multi-step Christoffel transform. It would be interesting to find out if there is any computational benefit of this algorithm compared to the repeated use of the one-step/two-step Gautschi/Galant algorithm. Such questions are important in numerical mathematics, see, e.g.,~~\cite{Bueno and Dopico,Galant,Gal92,Gautschi,Gautschi book,Golub and Kautsky,GolKau83,KauGol} and references therein. %, see, e.g., ~\cite{Bueno and Dopico,...} and references therein.
We will not pursue this in this paper. 

Our main focus is the study of the {\it multiple} Christoffel transform 
$$
 (\widehat\mu_1,\ldots,\widehat\mu_{r})
 =
  (\Phi\widehat\mu_1,\ldots,\Phi\widehat\mu_{r})
$$
%$(\widehat{\mu}_1,\ldots,\widehat{\mu}_r)$ 
of the multiple orthogonality system $(\mu_1,\ldots,\mu_r)$ for $r\ge 2$. 
%Here $\widehat{\mu}_j$ are defined as in~\eqref{eq:ChrM0} with the same polynomial $\Phi(x)$ for each $j$.
Such a transform appears naturally when one studies the multiple Gauss quadrature with fixed nodes at the zeros of $\Phi$.

We show how one can use the CC algorithm to compute the nearest neighbour recurrence coefficients of $(\widehat{\mu}_1,\ldots,\widehat{\mu}_r)$ (Section~\ref{ss:NNCC}) and establish the determinantal formula for 
%$\widehat{P}_{\bm{n}}$, which are 
the multiple orthogonal polynomials for  $(\widehat{\mu}_1,\ldots,\widehat{\mu}_r)$ (see Section~\ref{ss:CT type 2} for type II, and Section~\ref{ss:CT type 1} for type I). 

For the special case $N=1$ the determinantal formulas 
are known from the earlier literature: see ~\cite[Prop~3.2]{ADMVA} for type II, and ~\cite{BFM22} for type I for the case of two measures and multi-indices along the step-line. 
%reduce to~\eqref{eq:type 1 det formula alternative version} for type I and to~\eqref{eq:ChrP1} for type II. The relation~\eqref{eq:ChrP1} was known earlier~\cite[Prop~3.2]{ADMVA}, while ~\eqref{eq:type 1 det formula alternative version} has appeared in~\cite{BFM22} for the case of two measures and multi-indices along the step-line. 
During the preparation of the manuscript there appeared ~\cite{ManasRojas} studying multiple Christoffel transforms using another approach (the Gauss--Borel factorization) for the step-line multi-indices.
 
%For the special case $N=1$ the determinantal formula for type II polynomials reduces to~\eqref{eq:ChrP1} which has appeared in~\cite[Prop~3.2]{ADMVA}.

In Section~\ref{ss:repeated} we demonstrate that CC algorithm can be used to compute {\it repeated} Christoffel transforms, which is the natural setting for the $qd$ algorithm of Rutishauser~\cite{QD} and the discrete-time Toda lattices in one (see, e.g.,~\cite{SpiZhe}) and multiple dimensions, see~\cite{ADMVA,Dol} and references therein. %, e.g., ~\cite{} and ~\cite{ADMVA}, respectively.

In Section~\ref{ss:recurrNorm} we classify all possible nearest neighbour recurrence coefficients $\{a_{\bm{n},j},b_{\bm{n},j}\}$ that can occur for maximally-normal systems  (such systems are called perfect). This was proved for systems with two  positive infinitely-supported measures in~\cite{Integrable Systems}. We provide an alternative simple proof that allows either measures or linear moment functionals which may be finitely or infinitely supported. The main difference is that for finitely supported $\mu_j$'s the $a_{\bm{n},j}$-coefficients must be zero not only on the initial marginal indices (that is, with $n_j=0$) but also on the final ones (with $n_j = |\supp\,\mu_j|$). %, see Theorem~\ref{thm:perfect systems}  for the details.

The next portion of the results (Section~\ref{ss:interlacing}) concerns interlacing of the zeros of multiple orthogonal polynomials $P_{\bm{n}}$ for $(\mu_1,\ldots,\mu_r)$ and $\widehat{P}_{\bm{n}}$ for $(\widehat{\mu}_1,\ldots,\widehat{\mu}_r)$ when $\deg\Phi = 1$. We show that the zeros of $P_{\bm{n}}$  and $\widehat{P}_{\bm{n}}$ interlace for a wide class of systems including all Angelesco and AT systems, and the same result holds for type I polynomials for a class of measures containing all Angelesco systems. In particular, this applies to multiple Laguerre of the first and second kind, Jacobi--Pi\~{n}eiro, Angelesco--Jacobi, Jacobi--Laguerre, Jacobi--Hermite, Charlier, Meixner of the first and second kind, Krawtchouk, and Hahn (Sections~\ref{ss:interlacingContinuous} and~\ref{ss:interlacingDiscrete}). This type of interlacing was shown very recently in~\cite{MarMor24,dosSan} for Angelesco--Jacobi, Jacobi--Laguerre, Jacobi--Hermite systems for type II polynomials along the step-line multi-indices using much more involved arguments, see also~\cite{MarMorPer} for related results which use the notion of free convolution. 

The interlacing results for the discrete systems then produces the lower bound 1 for the distance between two consecutive zeros (see Section~\ref{ss:Minimal}), a result that is well-known for the $r=1$ case, see~\cite{ChiSta,Lev67,KraZar}.

Note that the result of one-step or multi-step Christoffel transform is always a positive measure if all of the zeros of $\Phi(x)$ fall outside of the interior of the convex hull of $\supp\,\mu$ (or if there are an even number of them at each gap of the support). Otherwise however, $\widehat\mu$ is a signed/complex measure, which is convenient to view as a linear moment functional. It does not take too much extra effort to allow $\mu$ to be a linear moment functional from the beginning, which is what we do starting from Section~\ref{ss:prelim} onwards. %A reader interested in positive measures only can just restrict ...

%Finally, in Section~\ref{normality critera} we collect several criteria that allow to establish the normality of indices for multiple orthogonality systems. % for which one choice of $P_{\bm{n}}$ is known. 
%While basic, these results may be of interest on their own. For example, we apply them to find a necessary and sufficient condition for $(\widehat{\mu}_1,\ldots,\widehat{\mu}_r)$ to be perfect. It is given in terms of certain determinants involving $P_{\bm{n}}$'s (see Theorem~\ref{thm:ChrMOPRLII}). 

In a companion paper~\cite{KVGeronimus} we obtain determinantal formulas for type I and type II multiple orthogonal polynomials for rational perturbations of measures, which includes the general Geronimus~\cite{Geronimus} and Uvarov~\cite{Uvarov} transforms, as well as the Christoffel transforms with different polynomials $\Phi_j$ for each $\mu_j$.

\subsubsection*{Acknowledgements}
Most of the results that appear in the current paper were part of M.V.'s 2021 Master Thesis in Uppsala University under the supervision of R.K., as reported in ~\cite{MVthesis}.

\section{Preliminaries}\label{ss:prelim} 
\subsection{Orthogonal polynomials with respect to moment functionals }\label{ss:quasi}
\hfill\\

We use the notation $\bbN:=\{k\in\Z: k\ge 0\}$ %(notice that we include $0$) 
and  $\nat:=\{k\in\bbZ: k>0\}$.
Let us assume that we are given an arbitrary sequence $\{c_n\}_{n=0}^\infty$ of complex numbers which will be referred to as the moment sequence. Define the corresponding moment functional $\mu$ to be the linear map on the space of all polynomials such that
\begin{equation}\label{eq:momentFunctional}
\mu[x^n] = c_n, \qquad n\in\bbN.
\end{equation}
Associated to $\mu$ we have the bilinear form
\begin{equation}\label{eq:bilinear form}{\jap{P(x),Q(x)} = \mu[P(x)Q(x)].}
\end{equation}
In particular, note that
\begin{equation}\label{eq:bilinear form property}{\jap{P(x)R(x),Q(x)} = \jap{P(x),R(x)Q(x)}}
\end{equation}
for any choice of polynomials $P$, $Q$, and $R$. Orthogonal polynomials with respect to $\mu$ are %sequences of 
non-zero polynomials $P_n(x)$ such that $\deg{P_n} \leq n$ and
\begin{equation}\label{eq:op def}{\jap{P_n(x),x^p} = 0, \qquad p = 0,1,\dots,n-1.}
\end{equation}
Such polynomials always exist, since solving \eqref{eq:op def} for the first $n+1$ Maclaurin coefficients results in a homogeneous system of linear equations with more columns than rows. If we fix the coefficient at $x^n$ we get a linear system with coefficient matrix
\begin{equation}\label{eq:matrix}
{M_n = \begin{pmatrix}
c_0 & c_1 & \cdots & c_{n-1} \\
c_1 & c_2 & \cdots & c_{n} \\
\vdots & \vdots & \ddots & \vdots \\
c_{n-1} & c_n & \cdots & c_{2n-2} \\
\end{pmatrix}.}
\end{equation}
%We say that $n$ is normal (with respect to $\mu$) if $\Delta_n = \det{M_n} \neq 0$. Then $n$ is normal if and only if $P_n$ is unique up to multiplication by a constant. In the normal case we assume $P_n$ to be normalized to have leading coefficient $1$. If 
Then we see that $\Delta_n = \det{M_n} \neq 0$ if and only if $P_n$ is unique up to multiplication by a constant and $\deg{P_n} = n$. In this case we always take $P_n$ to be monic. %If $\Delta_n = 0$ then there is always at least one non-zero solution $P_n$ of \eqref{eq:op def} with $\deg{P_n} < n$. Indeed, there would be at least two different monic $P_n$ with degree $n$ solving \eqref{eq:op def}, and then their difference would be a non-zero solution of smaller degree.

Denote $\mathcal{L}_\infty$ to be the set of all {\it quasi-definite} moment functionals, which are those $\mu$ for which $\Delta_n \ne 0$ for all $n\in\bbN$. For such $\mu$ the monic orthogonal polynomial $P_n$ is unique for each $n\in\bbN$.
The polynomials satisfy  the three-term recurrence relation
\nm{eq:three-term}{xP_n(x) = P_{n+1}(x) + b_nP_n(x) + a_nP_{n-1}(x), \qquad n\in \bbN}
for some complex numbers $a_n$ and $b_n$, called the Jacobi coefficients of $\mu$ (in the case $n = 0$ we formally take $a_0 = 0$ and $P_{-1} = 0$).

It is well known that $a_n \neq 0$ for all $n > 0$ %(and a simple proof is included here in Lemma \ref{lem:normality vs recurrence})
. Conversely, Favard's Theorem states that any set of $a_n$ and $b_n$, with $a_0 = 0$ and $a_n \neq 0$ for $n>0$, generates a sequence of polynomials from \eqref{eq:three-term} that are the orthogonal polynomials 
with respect to some moment functional $\mu$. 

The tridiagonal matrix 
\begin{equation}\label{eq:jacobi}
J=\left(
\begin{array}{cccc}
b_0&1&{0}&\\
a_1 &b_1&1&\ddots\\
{0}&a_2 &b_3&\ddots\\
&\ddots&\ddots&\ddots\end{array}\right)
\end{equation}
will be called the (``monic'') Jacobi matrix associated with $\mu$. It is the matrix of the map $P(x) \mapsto xP(x)$ in the basis $\set{P_n}_{n = 0}^{\infty}$, see \eqref{eq:three-term}. 

Finally, define the Christoffel--Darboux kernel via
\begin{equation}\label{eq:CDkernel}
    K_n(x,y)= \sum_{j = 0}^{n-1}\frac{{P_j(x)}P_j(y)}{\jap{P_j,P_j}}.
\end{equation}
Then %, see e.g.~\cite{Chihara}, 
the Christoffel--Darboux identity takes place:
\begin{equation}\label{eq:CDformula}
    K_{n}(x,y)= \frac{1}{\jap{P_{n-1},P_{n-1}}}\frac{P_{n}(x)P_{n-1}(y)-P_{n-1}(x)P_{n}(y)}{x-y}.
\end{equation}

For more on the basics of orthogonal polynomials with respect to moment functionals, see for example \cite{Chihara}.

Let $\mathcal{M}_\infty$ be the set of measures $\mu$ on $\R$ with infinite support and all the moments 
\begin{equation}\label{eq:measure}
	%\mu[x^n]:=
    c_n=\int x^n d\mu(x), \quad n\in\bbN,
\end{equation}
finite. Such a measure generates a quasi-definite moment functional~\eqref{eq:momentFunctional} which with a mild abuse of notation we also denote by $\mu$. In particular,~\eqref{eq:bilinear form} becomes the usual inner product in $L^2(\mu)$. With this convention, we can view $\mathcal{M}_\infty$ as a subset of $\mathcal{L}_\infty$. The setting $\mu\in\mathcal{M}_\infty$ corresponds to the standard theory of orthogonal polynomials with $a_n > 0$ for all $n\ge 1$ and $b_n\in\bbR$ for all $n \ge 0$.

\subsection{$\mu$ associated with finite (complex) Jacobi matrices}
\hfill\\

If one takes a measure $\mu$ on $\bbR$ supported on exactly $N$ distinct points $\{z_j\}_{j=1}^N\subset\bbR$, $N\in\nat$, then it is known that $\Delta_n \ne 0$ for $0\le n\le N$ and $\Delta_n = 0$ for $n>N$. Consequently, only $\{P_n(x)\}_{n=0}^N$ are uniquely defined, with $P_N(x) = \prod_{j=1}^N (x-z_j)$. The three-term recurrence~\eqref{eq:three-term} holds for $0\le n \le N-1$ with $a_n>0$ for $1\le n\le N-1$. The corresponding Jacobi matrix $J$ in \eqref{eq:jacobi} is finite of size $N\times N$ with $\{b_n\}_{n=0}^{N-1}$ on the diagonal and $\{a_n\}_{n=1}^{N-1}$ on the subdiagonal. Denote the set of such $N$-finitely supported measures %(or their associated moment functionals) 
by $\mathcal{M}_N$.

In what follows we want to allow \textit{complex} finite Jacobi matrices and the associated linear functionals. Therefore we define $\mathcal{L}_N$, for each $N\in\nat$, to be the set of all the moment functionals $\mu$ for which $\Delta_n \ne 0$ for $1\le n\le N$ and $\Delta_n = 0$ for $n>N$. %We will refer to $\bigcup_{N=1}^\infty \mathcal{L}_N$ as the moment functionals of finite support. 

%We will refer to $\bigcup_{N=1}^\infty \mathcal{L}_N$ as the moment functionals \textit{of finite rank} since this set consists of exactly those moment functionals that correspond to finite (complex) Jacobi matrices, see Theorem~\ref{thm:Favard} below. %, which can be viewed as a Favard-type theorem for finite rank moment functionals. First we need the following lemma.

\begin{lem}\label{lem:finite rank}
    Suppose $\mu\in\mathcal{L}_N$ for some $N\in\nat$. Then
    \begin{equation}\label{eq:preOrthLemma}
        \jap{P_N(x),x^p} = 0, \qquad p \in \bbN. 
    \end{equation}
    Moreover, we have
    \begin{equation}\label{eq:orthogLemma}
    \jap{P(x),x^p} = 0, \qquad p=0,1,\ldots,N-1,
    \end{equation}
    if and only if $P(x)$ is divisible by $P_N(x)$.
\end{lem}
\begin{proof}
    Since $\Delta_{N+1} = 0$, there has to exist some non-zero $P_{N+1}$ solving \eqref{eq:op def} (at $N+1$) for $n = N+1$ with degree $\deg{P_{N+1}} < N+1$. This is because $\Delta_{N+1} = 0$ implies the existence of two linearly independent monic solutions to \eqref{eq:op def} of degree $\leq N+1$, and if they both have degree $N+1$ then their difference is also non-zero and solves \eqref{eq:op def} but has degree $< N + 1$. We can always scale $P_{N+1}$ to be monic. 
    
    Since $\Delta_N \neq 0$, we must have $P_{N+1} = P_N$, as there is only one monic solution at $N$. By the orthogonality relations of $P_{N+1}$ we then must have $\jap{P_N(x),x^N} = 0$. We proceed to prove by induction that $\jap{P_N(x),x^p} = 0$ for all $p \geq N$. By $\Delta_{n+1} = 0$ (assuming $n \geq N$), there is some monic $P_n$ (solving \eqref{eq:op def} at $n$) with $\deg{P_n} = n-k$, $0 \leq k \leq n-N$, such that $\jap{P_n(x),x^n} = 0$ (since there is some monic $P_{n+1}$ with $\deg{P_{n+1}} \leq n$). In particular, $\jap{P_n(x),Q(x)} = 0$ for every polynomial $Q$ with $\deg{Q} \leq n$, so we must have $\jap{P_N(x), x^kP_n(x)} = \jap{P_n(x),x^kP_N(x)} = 0$, by \eqref{eq:bilinear form property}. Assuming $\jap{P_N(x),x^p} = 0$ is true for all $p < n$ (and here $n > N$), we then end up with $\jap{P_N(x),x^n} = \jap{P_N(x),x^kP_n(x)} = 0$, since $x^kP_n(x) = x^n + o(x^n)$, and $\jap{P_N(x),o(x^n)} = 0$. This proves \eqref{eq:preOrthLemma}. 
    
    Now for the second part, we have $\jap{P_N(x)Q(x),x^p} = \jap{P_N(x),x^pQ(x)} = 0$ by \eqref{eq:preOrthLemma}. Conversely, suppose that $P$ satisfies~\eqref{eq:orthogLemma} and write $P = P_N q + r$ with $\deg r < N$. Since $P_N$ is orthogonal to any polynomial, we get that $r$ also satisfies \eqref{eq:orthogLemma}. This is only possible if $r\equiv 0$. %with $\deg P = n\ge N$ (and if $\deg{P} < N$ then $P = 0$ by $\Delta_N \neq 0$). Then $P = P_N q + r$ with $\deg r < N$. Since $P_N$ is orthogonal to any polynomial, we get that $r$ also satisfies \eqref{eq:orthogLemma}. This is only possible if $r\equiv 0$.
\end{proof}

%Now we are ready to prove the analogue of the Favard theorem for the functionals of  finite support.
\begin{thm}\label{thm:Favard}
    If $\mu\in\mathcal{L}_N$, $N\le \infty$, then $\{P_n(x)\}_{n=0}^N$ satisfy the three-term recurrence relation ~\eqref{eq:three-term} for $0\le n \le N-1$ with $a_n\in\C\setminus\{0\}$ for $1\le n\le N-1$ and $b_n\in\C$  for $0\le n\le N-1$. Conversely, for any non-zero $\set{a_n}_{n = 1}^{N-1}$ and any $\set{b_n}_{n = 0}^{N-1}$ there exists a functional $\mu\in\mathcal{L}_N$, unique up to multiplication by a non-zero constant, with exactly these Jacobi coefficients.
\end{thm}
\begin{rem}\label{rem:Gus}
    See~\cite[Thm 4]{Gus} for a more detailed description of $\mu$ from $\mathcal{L}_N$: if the $N$ roots of $P_N$ are all distinct then $\mu$ can be represented as integration against a (complex) measure supported on these roots. If some of the roots overlap then $\mu$ contains derivative(s) of the Dirac delta function at the corresponding root. %We remark that any functional has infinitely many representations  
\end{rem}
\begin{proof}
      The first part of the statement follows from the same arguments as in ~\cite[Ch. 1, Sect. 3--4]{Chihara}. The second part follows very similarly to the proof of~\cite[Ch.1, Thm 4.4]{Chihara}. First define $\mu[1] = c \neq 0$ and $\mu[P_n(x)] = 0$ for $n = 1,\dots,N-1$, where $P_n$ are generated by the recurrence coefficients through \eqref{eq:three-term}. If we also define $\mu[x^p P_N(x)] = 0$ for all $p \in \bbN$ (see Lemma~\ref{lem:finite rank}) we can then extend $\mu$ uniquely to all polynomials. From the recurrence relation one can prove that $P_n$ are orthogonal polynomials with respect to $\mu$, for $n = 0,\dots,N-1$, and also $\jap{P_n,x^n} = a_n \dots a_1 \neq 0$, which implies $\Delta_n \neq 0$ for $n = 1,\dots,N$. From $\langle P_N(x),x^{p+N} \rangle = 0$ we get $\Delta_n = 0$ for each $n > N$, so $\mu$ belongs to $\mathcal{L}_N$. Since these assumptions on $\mu$ were necessary for $\set{P_n}_{n = 0}^{N}$ to be orthogonal with respect to some $\mu \in \mathcal{L}_N$, and the recurrence coefficients are uniquely determined by the recurrence, we get the full result.
\end{proof}
%\begin{rem}\label{rem:m}
%    Denote $\{Q_j(x)\}_{j=0}^{N}$ to be the polynomials that satisy the same recurrence~\eqref{eq:three-term} for $n\ge 1$ with the initial conditions $Q_{0}(x)=0$, $Q_{1}(x)=1$ (these are called the monic second kind polynomials). $Q_j$ (with $\deg Q_j = j-1$) is the characteristic polynomial of the top-left $j\times j$ submatrix of $J$ with the first row and column removed. The function $m(z)= -\frac{Q_N(z)}{P_N(z)}$ is then easily seen to be the top-left entry of the resolvent of $J$. The moments $\{c_n\}_{n=0}^\infty$ of $\mu$ can be recovered from the Taylor expansion
    %\begin{equation}\label{eq:m-function}
    %\frac{Q_N(z)}{P_N(z)} =\sum_{j=0}^\infty c_n z^{-n-1}, \quad z\to\infty.
    %\end{equation}
    %\end{rem}

Finally, we define
\begin{equation}
    \mathcal{L} = \Big(\bigcup_{N=1}^\infty \mathcal{L}_N\Big) \cup \mathcal{L}_\infty.
\end{equation}
This is the set of moment functionals that we are working with throughout this paper. Note that $\mathcal{L}$ includes $\mathcal{M}=\Big(\bigcup_{N=1}^\infty \mathcal{M}_N\Big) \cup \mathcal{M}_\infty$, the set of all positive measures of finite or infinite support on $\bbR$ with finite moments.

\subsection{Christoffel transform}
\hfill\\

For $\mu\in\mathcal{L}$, a one-step Christoffel transform is a functional $\widehat{\mu}$ given by 
\begin{equation}\label{eq:ChrL1}
    \widehat\mu[x^n] = \mu[x^n(x-z_0)], \qquad n\in\bbN.
\end{equation}
If $\mu\in\mathcal{M}$, then $\widehat\mu$ becomes~\eqref{eq:ChrM10}.

More generally, if $\mu\in\mathcal{L}$ and $\Phi(x) = \prod_{j = 1}^m(x-z_j)$ is any polynomial, then we define the corresponding Christoffel transform
\begin{equation}\label{eq:ChrL}
    \widehat\mu[x^n] = \mu[x^n\Phi(x)], \qquad n\in\bbN.
\end{equation}
{We will also occasionally employ the notation $\Phi\mu$ for $\widehat\mu$ to make the dependence on $\Phi$ explicit.}
 If $\mu\in\mathcal{M}$ and $\Phi$ has real coefficients and does not change sign on the convex hull of $\supp(\mu)$, then $\Phi\mu$ is also in $\mathcal{M}$ and is given as in~\eqref{eq:ChrM0}.
 %{for the normalized transform,
%\begin{equation}\label{eq:ChrM}
%    \int f(x)d\widehat\mu(x)= \frac{1}{Z}\int f(x){\Phi(x)}d\mu(x), \qquad \mbox{ where } Z= \int {\Phi(x)}d\mu(x).
%\end{equation}}

{Some authors choose to work with the normalized version of $\Phi{\mu}$ given by
$\mu[\Phi(x)]^{-1}\mu[x^n\Phi(x)]$, under the assumption $\mu[\Phi(x)]\ne 0$. Note that the monic orthogonal polynomials and their recurrence coefficients for both versions of the Christoffel transform are the same, so this makes no significant difference.}
%(note under non-zero multiplicative normalization the monic orthogonal polynomials do not change). In particular,  we may want to consider the normalized transform $\mu[x-z_0]^{-1}\mu[x^n(x-z_0)]$.}

{Given any $\mu \in \mathcal{L}$ and its Christoffel transform $\widehat\mu=\Phi\mu$,} we write $\widehat{P}_n$ for the orthogonal polynomials with respect to $\widehat{\mu}$, and $\widehat{a}_n$ and $\widehat{b}_n$ for the Jacobi coefficients of $\widehat{\mu}$. {If $\mu, \widehat{\mu} \in \mathcal{L}_\infty$ then the following Christoffel determinantal formula~\cite{Christoffel} holds:}
\nm{eq:christoffel theorem}{\widehat{P}_n(x) = {\Phi(x)}^{-1}D_n^{-1}\det
%\begin{pmatrix}
%P_n(z_1) & P_{n+1}(z_1) & \cdots & P_{n+m}(z_1) \\
%\vdots & \vdots & \ddots & \vdots \\
%P_n(z_m) & P_{n+1}(z_m) & \cdots & P_{n+m}(z_m) \\ 
%P_n(x) & P_{n+1}(x) & \cdots & P_{n+m}(x) \\
%\end{pmatrix}
\begin{pmatrix}
P_{n+m}(x) & P_{n+m-1}(x) & \cdots &  P_n(x) \\
P_{n+m}(z_1)  & P_{n+m-1}(z_1) & \cdots &  P_n(z_1) \\
\vdots & \vdots & \ddots & \vdots \\
P_{n+m}(z_m)  & P_{n+m-1}(z_m) & \cdots & P_n(z_m) 
\end{pmatrix}
,}
where $D_n$ is the normalizing constant
\nm{eq:normalizing constant one measure}{D_n = \det
%\begin{pmatrix}
%P_n(z_1) & P_{n+1}(z_1) & \cdots & P_{n+m-1}(z_1) \\
%P_n(z_2) & P_{n+1}(z_2) & \cdots & P_{n+m-1}(z_2) \\
%\vdots & \vdots & \ddots & \vdots \\
%P_n(z_m) & P_{n+1}(z_m) & \cdots & P_{n+m-1}(z_m) \end{pmatrix}
\begin{pmatrix}
P_{n+m-1}(z_1)  & P_{n+m-2}(z_1) & \cdots & P_n(z_1)  \\
P_{n+m-1}(z_2)  & P_{n+m-2}(z_2) & \cdots & P_n(z_2)  \\
\vdots & \vdots & \ddots & \vdots \\
P_{n+m-1}(z_m)  & P_{n+m-2}(z_m) & \cdots & P_n(z_m) \end{pmatrix}
.}
In the one-step case $m=1$ this becomes 
\nm{eq:one step ct}{\widehat{P}_n(x) = \frac{1}{x-z_0} \left({P}_{n+1}(x) - \frac{{P}_{n+1}(z_0)}{{P}_{n}(z_0)} {P}_{n}(x)\right).}
By the Christoffel--Darboux formula~\eqref{eq:CDformula} we get
\nm{eq:cd one measure}{\widehat{P}_n(x) = 
\frac{\jap{P_n,P_n}}{P_n(z_0)} K_{n+1}(z_0,x).
}
{Polynomials $K_n(z_0,z)$ are sometimes referred to as the kernel polynomials of $\mu$.}

\eqref{eq:one step ct} together with the three-term recurrence \eqref{eq:three-term} can be used to generate the recurrence equations (see \cite{Gautschi book})
\nm{eq:Gautschi algorithm}{\widehat{b}_n - \delta_{n+1} & = b_{n+1} - \delta_n, \\
\widehat{a}_n - \delta_n\widehat{b}_n & = a_{n+1} - \delta_n b_n, \\
\delta_{n-1}\widehat{a}_n & = \delta_n a_n,}
with initial conditions $\widehat{a}_0 = 0$ and $\delta_0 = z_0 - b_0$. From this, it is possible to compute the Jacobi coefficients of $\widehat{\mu}$ from the Jacobi coefficients of $\mu$, through the following algorithm,
    \begin{align}\label{Galant algorithm}
    \begin{split}
        & \delta_{0}:= z_0 - b_{0};
    \\
    & \widehat{a}_{0}=0;
    \\
    & \widehat{b}_{0}= b_{0} - \frac{a_1}{\delta_{0}};
    \\
    & \mbox{for all } n\in\bbZ_+:
    \\
    & \qquad
     \delta_{n}:= \widehat{b}_{n-1} - b_{n} + \delta_{n-1};
    \\
    & \qquad
         \widehat{a}_{n} = a_{n}
    \frac{\delta_{n}}{\delta_{n-1}};
    \\
    & \qquad
    \widehat{b}_{n} = b_{n} + \frac{\widehat{a}_{n}-a_{n+1}}{\delta_{n}}.
    \end{split}
\end{align}

In Section~\ref{ss:CT type 2} we show how this can be generalized to the multi-step Christoffel transform using multiple orthogonal polynomials.

\subsection{Basics of multiple orthogonal polynomials on the real line (MOPRL)}
\hfill\\

Let $r\ge 1$ and consider a system of functionals $\bm{\mu} = (\mu_1,\dots,\mu_r) \in \mathcal{L}^r$.  Let us write $\jap{P(x),Q(x)}_j$ for $\mu_j[P(x)Q(x)]$. 
%a system of probability measures $\bm{\mu} = (\mu_1,\dots,\mu_r)\in\mathcal{M}^r$, or more generally a system of functional  be a system of probability measures  on $\bbR$.
\begin{defn}
Given a multi-index $\bm{n}\in\bbN^r$, a type II multiple orthogonal polynomial is a non-zero polynomial $P_{\bm{n}}(x)$ such that $\deg{P_{\bm{n}}} \leq \abs{\bm{n}}:=n_1+\ldots+n_r$, and
\nm{eq:moprlII}{\jap{P_{\bm{n}}(x),x^p}_j = 0,\qquad p = 0,1,\dots,n_j-1 ,\qquad j = 1,\dots,r.}
%where $\jap{\cdot,\cdot}_j$ is the inner product~\eqref{eq:inner} but with $d\mu_j$ instead of $d\mu$. 
\end{defn}
\begin{defn}
A type I multiple orthogonal polynomial is a non-zero vector of polynomials $\bm{A_{\bm{n}}} = (A_{\bm{n}}^{(1)},\dots,A_{\bm{n}}^{(r)})$ such that $\deg{A_{\bm{n}}^{(j)}} \leq n_j-1$, $j = 1,\dots,r$, and
\nm{eq:moprlI}
{\sum_{j = 1}^r \jap{A_{\bm{n}}^{(j)}(x), x^p}_j = 0,\qquad p = 0,1,\dots,\abs{\bm{n}}-2.}
\end{defn}
Note that $A_{\bm{n}}^{(j)}= 0$ when $n_j = 0$ (we take the degree of $0$ to be $-\infty$). Hence for $\bm{n} = \bm{0}$ there would be no non-zero solutions to \eqref{eq:moprlI}. In this case we take $\bm A_{\bm{0}} = \bm{0}$ (the $r$-vector of zero polynomials) as the only type I polynomial.

%In fact instead of $(\mu_1,\dots,\mu_r)\in\mathcal{M}^r$ one can take each $\mu_j\in\mathcal{L}$ be a quasidefinite linear moment functional as in Section~\ref{ss:quasi}. In that case one should treat $\jap{f(x),g(x)}_j$ as $\mu_j[f(x)g(x)]$. We 
%refer to this setting as quasidefinite MOPRL and 
%differentiate between these two setting by writing 
%$\bm{\mu} \in\mathcal{M}^r$  or $\bm{\mu} \in\mathcal{L}^r$ with the latter setting being the default one.

It is easy to show that for any multi-index $\bm{n}\in\bbN^r\setminus\set{\bm{0}}$ the following statements are equivalent:
\begin{itemize}
    %\item[(i)] $\deg P_{\bm{n}} = \abs{\bm{n}}$ for every non-zero solution of \eqref{eq:moprlII};
    \item[(i)] There is a unique monic type II multiple orthogonal polynomial $P_{\bm{n}}$ such that $\deg{P_{\bm{n}}} = \abs{\bm{n}}$;
    %\item[(iii)] $\sum_{j = 1}^r \jap{A_{\bm{n},j}(x), x^{\abs{\bm{n}}-1}}_j \neq 0$ for every non-zero solution of \eqref{eq:moprlI};
    \item[(ii)] There is a unique type I multiple orthogonal polynomial $(A_{\bm{n}}^{(1)},\dots,A_{\bm{n}}^{(r)})$ such that 
    \begin{equation}\label{eq:typeInorm}\sum_{j = 1}^r \jap{A_{\bm{n}}^{(j)}(x), x^{\abs{\bm{n}}-1}}_j = 1.
    \end{equation}
    \item[(iii)] $\deg P_{\bm{n}} = \abs{\bm{n}}$ for every non-zero solution of \eqref{eq:moprlII};
    \item[(iv)] $\sum_{j = 1}^r \jap{A_{\bm{n}}^{(j)}(x), x^{\abs{\bm{n}}-1}}_j \neq 0$ for every non-zero solution of \eqref{eq:moprlI};
    \item[(v)] $\det M_{\bm{n}} \ne 0$, where 
\begin{equation}
    \label{eq:moprl matrix}
M_{\bm{n}} = \begin{pmatrix}
c_0^{(1)} & c_1^{(1)} & \cdots & c_{\abs{\bm{n}}-1}^{(1)} \\
c_1^{(1)} & c_2^{(1)} & \cdots & c_{\abs{\bm{n}}}^{(1)} \\
\vdots & \vdots & \ddots & \vdots \\
c_{n_1-1}^{(1)} & c_{n_1}^{(1)} & \cdots & c_{\abs{\bm{n}}+n_1-2}^{(1)} \\
\hline
& & \vdots \\
\hline
c_0^{(r)} & c_1^{(r)} & \cdots & c_{\abs{\bm{n}}-1}^{(r)} \\
c_1^{(r)} & c_2^{(r)} & \cdots &  c_{\abs{\bm{n}}}^{(r)} \\
\vdots & \vdots & \ddots & \vdots \\
c_{n_r-1}^{(r)} & c_{n_r}^{(r)} & \cdots & c_{\abs{\bm{n}}+n_r-2}^{(r)} \\
\end{pmatrix},
\end{equation}
where $c_n^{(j)}$ are the moments
\m{c_n^{(j)} = \mu_j[x^n]. }
\end{itemize}
%Indeed, rewriting~\eqref{eq:moprlII} for a monic degree $\abs{\bm{n}}$ polynomial, and~\eqref{eq:moprlI}+\eqref{eq:typeInorm} as linear systems with respect to their respective Maclaurin coefficients, one can see that each of these statements is equivalent to (iii).
\begin{defn}
If any and therefore each of the above conditions (i)--(v) are satisfied for an index $\bm{n} \neq \bm{0}$ then we say that $\bm{n}$ is normal. We always take the index $\bm{0}$ to be normal. 
\end{defn}
Whenever $\bm{n}$ is normal we will only work with polynomials satisfying (i) and (ii) above, that is, $P_{\bm{n}}$ will be monic and $\bm{A}_{\bm{n}}$ will satisfy~\eqref{eq:typeInorm}.

\begin{prop}\label{rem:perf}
    If $\mu_j \in \mathcal{L}_{N_j}$ with $N_j<\infty$, then indices $\bm{n}$ with $n_j > N_j$ can never be normal. 
\end{prop}
\begin{proof}
    In the case $r = 1$ this follows from Lemma \ref{lem:finite rank}. In general, if $n_j > N_j$, we have
    \m{\jap{P_{\bm{n}-\bm{e}_j}(x),x^{n_j-1}}_j = \jap{P_{\bm{n}-\bm{e}_j}(x),Q(x)}_j} 
    for any degree $n_j-1$ monic polynomial $Q$, given any choice of type II polynomial $P_{\bm{n}-\bm{e}_j}$ (for the index $\bm{n}-\bm{e}_j$). If we choose $Q(x) = x^{n_j-1-N_j}P_{N_j}(x)$, where $P_{N_j}$ is the degree $N_j$ orthogonal polynomials with respect to $\mu_j$, then we get
    \m{\jap{P_{\bm{n}-\bm{e}_j}(x),x^{n_j-1}}_j = \jap{P_{N_j}(x),P_{\bm{n}-\bm{e}_j}(x)x^{n_j-1-N_j}}_j = 0}
    by Lemma \ref{lem:finite rank}. This means that $P_{\bm{n}-\bm{e}_j}$ satisfies all the orthogonality relations \eqref{eq:moprlII} for the index $\bm{n}$, but $\deg{P_{\bm{n}-\bm{e}_j}} < \abs{\bm{n}}$, so $\bm{n}$ cannot be normal. 
    %Indeed, choose $Q(x)$ to be the characteristic polynomial of the Jacobi matrix associated with $\mu_j$. Then $\jap{Q(x),x^k}_j = 0$ for every $k\in\bbN$, $\deg Q = N_j$. This means that the $j$-th rectangular block of $M_{\bm{n}}$, see~\eqref{eq:moprl matrix}, has rank $\le N_j$ if $n_j>N_j$. This implies that $\det M_{\bm{n}}=0$.
\end{proof}

A system $\bm{\mu}=(\mu_1,\dots,\mu_r)$ in $\mathcal{L}^r$ or in $\mathcal{M}^r$ is usually said to be perfect if every $\bbN^r$-index is normal. We modify this notion to the case when $\mu_j \in \mathcal{L}_{N_j}$ where $N_j$ may be finite. 
%Given $\bm{\mu}\in\mathcal{M}^r$, let $N_j\in\nat\cup\{+\infty\}$ be the cardinality of $\supp\,\mu_j$. If $(\mu_1,\dots,\mu_r)\in\mathcal{L}^r$ we define $N_j\in\nat\cup\{+\infty\}$ to be the size of the associated Jacobi matrix (see Section~\ref{ss:quasi}). 
We write $\bm{N} = (N_1,\ldots,N_r)\in(\nat\cup\{+\infty\})^r$ and denote
$$
\N_{\bm{N}}^r:=\left\{ \bm{n}\in\bbN^r \mid 0\le n_j \le N_j \right\}.
$$

\begin{defn}\label{def:perf}
    We say that $(\mu_1,\dots,\mu_r)$ is a perfect system if all indices in $\N_{\bm{N}}^r$ are normal.
\end{defn}

One important class of perfect systems is Angelesco systems. Usually they are defined for measures with infinite support, but we extend the notion of Angelesco system to allow measures of finite support. 
If $\bm{\mu}\in\mathcal{M}^r$ we write $\Delta_j$ for the convex hull of $\supp(\mu_j)$ and $\Delta_j^{\mathrm{o}}$ for the interior of $\Delta_j$.
\begin{defn}
    An Angelesco system is a $\mathcal{M}^r$-system such that $\Delta_k^{\mathrm{o}}\cap\Delta_j^{\mathrm{o}} = \varnothing$ and if $\mu_k$ and $\mu_j$ are both finitely supported then we also require $\Delta_k\cap\Delta_j = \varnothing$ ($k \neq j$). 
\end{defn}

With this definition one can show that any Angelesco system is perfect in the sense of Definition~\ref{def:perf}. The proof goes along the same lines (counting real zeros of odd multiplicity of $P_{\bm{n}}$) as the standard argument for the case $\bm{N}=(\infty,\ldots,\infty)$, see e.g. ~\cite[Thm 23.1.3]{Ismail}.

\subsection{MOPRL: nearest neighbour recurrence relations}
\label{ss:NNRR}
\hfill\\

We remind the reader that whenever some index is normal then the type II multiple orthogonal polynomial at that location is always taken to be monic, and the type I polynomial is taken with the normalization~\eqref{eq:typeInorm}.

It was shown by Van Assche~\cite{NNR,Ismail} that multiple orthogonal polynomials of type II and type I satisfy the following set of equations, called the nearest neighbour recurrence relations (NNRR for short). Although it was orinigally stated in $\mathcal{M}_\infty^r$, the proof is no different in $\mathcal{L}^r$. 

\begin{thm} \label{thm:nnr} {\normalfont{\textbf{[Van Assche, \cite{NNR}]}}}
Let $\bm{\mu}\in\mathcal{L}^r$. If $\bm{n}$ and $\bm{n} + \bm{e}_j$ are normal, then 
\nm{eq:one nnr}{xP_{\bm{n}}(x) = P_{\bm{n} + \bm{e}_j}(x) + b_{\bm{n},j}P_{\bm{n}}(x) + \sum_{i = 1}^{r}a_{\bm{n},i}P_{\bm{n} - \bm{e}_i}(x)}
for some constants $a_{\bm{n},1},\dots,a_{\bm{n},r}$ and $b_{\bm{n},j}$. Here $a_{\bm{n},l} := 0$ %and $P_{\bm{n} - \bm{e}_l} = 0$ 
whenever $n_l = 0$. %, and all the other recurrence coefficients are uniquely determined by~\eqref{eq:one nnr}.
\end{thm}
%\begin{rem}\label{rem:NNRRNormality}
%    The only difference of this statement from the original result in \cite{NNR} is that 
%    we allow $\mu_j$ to be in $\mathcal{L}$ rather than $\mathcal{M}$. Van Assche's original proof \cite{Ismail} works with no changes. 
%\end{rem}
\begin{rem}\label{rem:NNRcoef1} 
The recurrence coefficient $b_{\bm{n},j}$ is defined when both $\bm{n}$ and $\bm{n} + \bm{e}_j$ are normal, and it is always unique, given our choice of normalization for the polynomials.
The coefficient $a_{\bm{n},i}$ is defined when $\bm{n}$ is normal and $\bm{n} + \bm{e}_l$ is normal for some $l$. $a_{\bm{n},i}$ is unique if $\bm{n} - \bm{e}_i$ is also normal.  Indeed,~\eqref{eq:one nnr} implies
\begin{equation}\label{eq:a}
a_{\bm{n},i} = \frac{\jap{P_{\bm{n}}(x),x^{n_i}}_i}{\jap{P_{\bm{n}-\bm{e}_i}(x),x^{n_i-1}}_i}
\end{equation}
if $n_i>0$. Therefore the constants $a_{\bm{n},i}$ are independent of the choice of $j$ in \eqref{eq:one nnr} (i.e., if $\bm{n}+\bm{e}_j$ are normal for several choices of $j$), but may depend on the choice of $P_{\bm{n}-\bm{e}_1},\dots,P_{\bm{n}-\bm{e}_r}$ if any of the indices $\bm{n}-\bm{e}_1,\ldots,\bm{n}-\bm{e}_r$ are not normal. 
\end{rem}
%\begin{rem}\label{rem:NNRcoef2} 
%    [about $a$'s equal to zero on $n_j=N_j$]. What about $a_{\bm{N},j}?$
%\end{rem}
\begin{rem}\label{rem:perfectAB}
    In particular, for perfect systems $\bm{\mu}\in\mathcal{L}^r$ with $N_j\le \infty$, we have that $a_{\bm{n},j}$ is defined for all $1\le j\le r$ and all $\bm{n}\in\bbN^r_{\bm{N}}$ except for $\bm{n}=\bm{N}$. This exception is only relevant if $N_j<\infty$ for all $j$, of course. Let us adopt the convention $a_{\bm{N},j}:=0$ for that case, for reasons that will become more clear later on in the paper. Assuming this, all perfect systems then have well-defined NNRR coefficiens $\{a_{\bm{n},j}\}_{\bm{n}\in\bbN^r_{\bm{N}}}$  and $\{b_{\bm{n},j}\}_{\bm{n}\in\bbN^r_{\bm{N}-\bm{e}_j}}$ for each $j$. 
\end{rem}

%From the nearest neighbour recurrence relations, we have
%\m{xP_{\bm{n}}(x) = P_{\bm{n} + \bm{e}_j}(x) + b_{\bm{n},j}P_{\bm{n}}(x) + \sum_{i = 1}^{r}a_{\bm{n},i}P_{\bm{n} - \bm{e}_i}(x), \\
%xP_{\bm{n}}(x) = P_{\bm{n} + \bm{e}_k}(x) + b_{\bm{n},k}P_{\bm{n}}(x) + \sum_{i = 1}^{r}a_{\bm{n},i}P_{\bm{n} - \bm{e}_i}(x). 
%} 

It is well-known that the type I polynomials satisfy similar recurrence relations. Since the explicit proof is hard to locate we provide a proof in Appendix~\ref{ss:appendix}, using minimal normality assumptions.

\begin{thm}\label{thm:nnrTypeI}
Let $\bm{\mu}\in\mathcal{L}^r$. If $\bm{n}$ and $\bm{n} - \bm{e}_j$ are normal, then 
\nm{eq:type 1 recurrence relation}{x\bm{A}_{\bm{n}}(x) = \bm{A}_{\bm{n} - \bm{e}_j}(x) + b_{\bm{n}-\bm{e}_j,j}\bm{A}_{\bm{n}}(x) + \sum_{i = 1}^{r}c_{\bm{n},i}\bm{A}_{\bm{n} + \bm{e}_i}(x),}
for some constants $c_{\bm{n},1},\dots,c_{\bm{n},r}$. % and $b_{\bm{n}-\bm{e}_j,j}$. 
If $\bm{n} + \bm{e}_i$ and $\bm{n} - \bm{e}_i$ are normal then $c_{\bm{n},i} = a_{\bm{n},i}$.
\end{thm}
\begin{rem}
    As above, $c_{\bm{n},i}$ are independent of the choice of $j$, but may depend on the choice of the vectors $(A_{\bm{n}+\bm{e}_i,1},\dots,A_{\bm{n}+\bm{e}_i,r})$ when $\bm{n} + \bm{e}_i$ is not normal (for the details, see the Appendix).
\end{rem}

By comparing~\eqref{eq:one nnr} for two different $j$, and doing the same for~\eqref{eq:type 1 recurrence relation}, %(and combining with~\eqref{eq:CC1}), 
we obtain the following. 

\begin{cor} \label{cor:nnr cor}
Assume $\bm{n}$, $\bm{n} + \bm{e}_j$ and $\bm{n} + \bm{e}_l$ are normal. Then
\begin{equation}
\label{eq:nnr cor}
 P_{\bm{n} + \bm{e}_l} - P_{\bm{n} + \bm{e}_j} = (b_{\bm{n},j} - b_{\bm{n},l})P_{\bm{n}}.
\end{equation}
Similarly, if $\bm{n}$, $\bm{n} - \bm{e}_j$ and $\bm{n} - \bm{e}_l$ are normal, then
\begin{equation}\label{eq:nnr cor type 1}
    \bm{A}_{\bm{n} - \bm{e}_l} - \bm{A}_{\bm{n} - \bm{e}_j} 
    = (b_{\bm{n}-\bm{e}_j,j} - b_{\bm{n}-\bm{e}_l,l})\bm{A}_{\bm{n}}
   %= (b_{\bm{n}-\bm{e}_j-\bm{e}_k,k} - b_{\bm{n}-\bm{e}_j-\bm{e}_k,j})A_{\bm{n},m}
\end{equation}
\end{cor}

%\subsection{MOPRL: the Compatibility Conditions and the NNCC algorithm}\label{ss:NNCC}
%\hfill\\
%In the case of one measure there is a measure $\mu$ such that $a_n$ and $b_n$ are the Jacobi coefficients of $\mu$, as long as $a_n > 0$ when $n > 0$ and $a_0 = 0$, by Favard's Theorem. In the case of several measures, there is more structure required on the nearest neighbour recurrence coefficients. First, if we don't have sufficient normality, the existence and uniqueness of recurrence coefficients is not guaranteed. On top of this, 
Van Assche showed in \cite{NNR} that (assuming sufficient normality conditions) the recurrence coefficients satisfy the partial difference equations in the theorem below. Their proof works for $\bm{\mu} \in\mathcal{L}^r$ without any changes. When we write $\bm{m} \leq \bm{n}$ we mean $m_j \leq n_j$ for each $j = 1,\dots,r$.

\begin{thm} \label{thm:compatibility conditions} \normalfont{\textbf{[Van Assche, \cite{NNR}]}} Let $\bm{\mu}\in\mathcal{L}^r$. Assume all indices $\bm{m} \leq \bm{n}+\bm{e}_j + \bm{e}_l$ is normal. If $j \neq l$ then
\begin{align}
\label{eq:CC1}
    & b_{\bm{n} + \bm{e}_l,j} - b_{\bm{n} + \bm{e}_j,l} = b_{\bm{n},j} - b_{\bm{n},l},\\
    \label{eq:CC2}
    & b_{\bm{n},l}b_{\bm{n} + \bm{e}_l,j} - b_{\bm{n},j}b_{\bm{n} + \bm{e}_j,l} = \sum_{i=1}^r a_{\bm{n} + \bm{e}_l,i} - \sum_{i=1}^r a_{\bm{n} + \bm{e}_j,i}, \\
    \label{eq:CC3}
    & a_{\bm{n} + \bm{e}_l,j}\brkt{b_{\bm{n} - \bm{e}_j,j} - b_{\bm{n} - \bm{e}_j,l}} = a_{\bm{n},j}\brkt{b_{\bm{n},j} - b_{\bm{n},l}}.
\end{align}
If $n_j = 0$ or $n_l=0$ then we only get the first two equations.
\end{thm}

We can use \eqref{eq:CC1} to rewrite the left hand side of \eqref{eq:CC2} as
\m{b_{\bm{n},l}b_{\bm{n} + \bm{e}_l,j} - b_{\bm{n},j}b_{\bm{n} + \bm{e}_j,l} & = b_{\bm{n},l}b_{\bm{n} + \bm{e}_l,j} - b_{\bm{n},j}(b_{\bm{n} + \bm{e}_l,j}+ b_{\bm{n},l} - b_{\bm{n},j}) \\ & = (b_{\bm{n},l} - b_{\bm{n},j})b_{\bm{n} + \bm{e}_l,j} - (b_{\bm{n},l} - b_{\bm{n},j})b_{\bm{n},j}.}
We write $\delta_{\bm{n},j,l}$ for $b_{\bm{n},l} - b_{\bm{n},j}$. Now we can rewrite \eqref{eq:CC1}--\eqref{eq:CC3} in the alternative form,
\begin{align}
    \label{eq:CC1.2}
    & b_{\bm{n} + \bm{e}_l,j} - \delta_{\bm{n} + \bm{e}_j,j,l} = b_{\bm{n}+\bm{e}_j,j} - \delta_{\bm{n},j,l},\\
    \label{eq:CC2.2}
    & \sum_{i=1}^r a_{\bm{n} + \bm{e}_l,i} - \delta_{\bm{n},j,l}b_{\bm{n}+\bm{e}_l,j} = \sum_{i=1}^r a_{\bm{n} + \bm{e}_j,i} - \delta_{\bm{n},j,l}b_{\bm{n},j}, \\
    \label{eq:CC3.2}
    & \delta_{\bm{n}-\bm{e}_j,j,l}a_{\bm{n} + \bm{e}_l,j} = \delta_{\bm{n},j,l}a_{\bm{n},j},
\end{align}
which will be better adapted for our purposes. 

{Consider a sequence of normal indices $\set{\bm{n}_k}_{k = 0}^{\infty}$ that forms a path starting at the origin, i.e., $\bm{n}_0 = \bm{0}$ and $\bm{n}_{k+1} = \bm{n}_k + \bm{e}_{j_k}$ for some $j_k \in \set{1,\dots,r}$, $k \in \bbN$. 
Define the generalized Christoffel--Darboux kernel $\bm{K}_{\bm{n}}(x,y) = (K_{\bm{n}}^{(1)}(x,y),\dots,K_{\bm{n}}^{(r)}(x,y))$  via
\begin{equation}\label{eq:cd kernel moprl}
    \bm{K}_{\bm{n}}(x,y)
    = 
    \sum_{k = 0}^{N-1}P_{\bm{n}_k}(x)\bm{A}_{\bm{n}_{k+1}}(y).
\end{equation}
Using the NNRR one can show~\cite{NNR} that the following Christoffel--Darboux formula holds
\begin{equation}\label{eq:cd formula moprl}
(x-y) \bm{K}_{\bm{n}}(x,y)
= P_{\bm{n}_N}(x)\bm{A}_{\bm{n}_N}(y) - \sum_{j = 1}^{r}a_{\bm{n},j}P_{\bm{n}_N-\bm{e}_j}(x)\bm{A}_{\bm{n}_N+\bm{e}_{j}}(y),
\end{equation}
originally due to~\cite{DK04}. %, where it is written for $P$ polynomials and $Q$ forms. 
}

\subsection{Normality criteria}\label{normality critera}
\hfill\\

Here we collect a series of results that allow to establish normality of indices for certain multiple orthogonality systems $\bm{\mu}\in\mathcal{L}^r$. While basic, these results may be of interest on their own. We will use these results in Section~\ref{ss:CT type 2}  to find necessary and sufficient conditions for $(\widehat{\mu}_1,\ldots,\widehat{\mu}_r)$ to be perfect.

%For example, we apply them to find a necessary and sufficient condition for $(\widehat{\mu}_1,\ldots,\widehat{\mu}_r)$ to be perfect. It is given in terms of certain determinants involving $P_{\bm{n}}$'s (see Theorem~\ref{thm:ChrMOPRLII}). 

We start with the following simple lemma. One direction is standard, see, e.g., ~\cite[Cor 23.1.1--23.1.2]{Ismail}, while the other has appeared in~\cite{BCVA}, as well as~\cite{mopuc recurrence} (for the unit circle), see also~\cite{KVMOPUC}.

%which is  an  improvement of~\cite[Cor 23.1.1--23.1.2]{Ismail}, see~\cite{mopuc recurrence}. It is useful, for example when studying recurrence relations. The full statement is rare to find, but we note that it has appeared \cite{Cruz Barroso}

\begin{lem} \label{lem:small normality lemma 2}
$\bm{n} + \bm{e}_j$ is normal if and only if 
%\nm{eq:non zero integral}{\int x^{n_j}P_{\bm{n}}(x) \dd \mu_j(x) \neq 0}
\nm{eq:non zero integral}{\jap{P_{\bm{n}}(x), x^{n_j}}_j \neq 0}
for every type II multiple orthogonal polynomial $P_{\bm{n}}$ corresponding to the index $\bm{n}$.

Similarly, $\bm{n} - \bm{e}_j$ is normal if and only if $A_{\bm{n}}^{(j)}$ has degree $n_j - 1$ for every vector of type I multiple orthogonal polynomials $(A_{\bm{n}}^{(1)},\dots,A_{\bm{n}}^{(r)})$.
\end{lem}
\begin{proof}
If $\jap{P_{\bm{n}}(x), x^{n_j}}_j $ is zero for some multiple orthogonal polynomial $P_{\bm{n}}$, then $P_{\bm{n}}$ satisfies every orthogonality condition for the index $\bm{n} + \bm{e}_j$, but $P_{\bm{n}}$ has degree $\leq \abs{\bm{n}}$, so $\bm{n} + \bm{e}_j$ cannot be normal. Conversely, if $\bm{n} + \bm{e}_j$ is not normal then (\ref{eq:moprlII}) has a solution $P_{\bm{n} + \bm{e}_j}$ of degree $\leq \abs{\bm{n}}$. $P_{\bm{n} + \bm{e}_j}$ satisfies the orthogonality conditions (\ref{eq:moprlII}) for the index $\bm{n}$, so it is a multiple orthogonal polynomial for the index $\bm{n}$. In other words, $P_{\bm{n} + \bm{e}_j} = P_{\bm{n}}$ for some choice of $P_{\bm{n} + \bm{e}_j}$ and $P_{\bm{n}}$, but then (\ref{eq:moprlII}) for the index $\bm{n} + \bm{e}_j$ implies
$\jap{P_{\bm{n}}(x),x^{n_j}}_j = 0$.

If there is some $A_{\bm{n}}^{(j)}$ of degree $< n_j - 1$ then we have a non-zero solution to the system~\eqref{eq:moprlI} for the index $\bm{n} - \bm{e}_j$, so $\bm{n} - \bm{e}_j$ is not normal since normality condition (iv) is not satisfied. Conversely, if $\bm{n} - \bm{e}_j$ is not normal then there is a non-zero solution $(A_{\bm{n}-\bm{e}_j}^{(1)},\dots,A_{\bm{n}-\bm{e}_j}^{(r)})$ to
\m{\sum_{i = 1}^r \jap{A_{\bm{n}-\bm{e}_j}^{(i)}(x), x^p}_i = 0
,\qquad p = 0,1,\dots,\abs{\bm{n}}-2,}
but then this is also a solution for the index $\bm{n}$, which means that there is an $A_{\bm{n},j}$ of degree $< n_j - 1$.
\end{proof}

In order to define the recurrence coefficients we required some indices to be normal. %(see Remark~\ref{rem:NNRRNormality}). 
In turn, the recurrence coefficients can give us information about the normality of neighbouring indices, as we show next.

\begin{lem} \label{lem:normality vs recurrence} The following holds true for the recurrence coefficients of a system $\mu \in \mathcal{L}^r$. 
\begin{enumerate} [label=\alph*)] 
    \item Let $\bm{n}$ and $\bm{n} + \bm{e}_i$ be normal for some $i$ and $n_j>0$. Then $a_{\bm{n},j} \neq 0$ if and only if %$n_j > 0$ and 
    $\bm{n} + \bm{e}_j$ is normal.
    \item Let $\bm{n}$, $\bm{n} + \bm{e}_j$ and $\bm{n} + \bm{e}_l$ be normal for some $j\ne l$. Then $b_{\bm{n},j} \neq b_{\bm{n},l}$ if and only if %$j \neq l$ and 
    $\bm{n} + \bm{e}_j + \bm{e}_l$ is normal. 
\end{enumerate}
\end{lem}

\begin{proof} {\it a)} follows by Lemma \ref{lem:small normality lemma 2} and \eqref{eq:a}. To show {\it b)}, take the $j$-th inner product with respect to $x^{n_j}$ in \eqref{eq:nnr cor} to get
\nm{eq:2nd integral expression}{b_{\bm{n},j} - b_{\bm{n},l} = \frac{\jap{P_{\bm{n}+\bm{e}_l}(x),x^{n_j}}_j}{\jap{ P_{\bm{n}}(x),x^{n_j}}{_j}}.}
Then {\it b)} follows from Lemma \ref{lem:small normality lemma 2}.
\end{proof}

To check perfectness of certain systems the following two results may be useful. Note that direct application of Lemma~\ref{lem:small normality lemma 2} would require to check the condition~\eqref{eq:non zero integral} for {\it every} orthogonal polynomial at every index $\bm{n}$. We show that it is enough to check it for only {\it one} choice of orthogonal polynomial at every index $\bm{n}$. This is useful since occasionally finding one such choice is easy from a Rodrigues-type formula or other arguments. One such application can be found in Section \ref{ss:CT type 2} below. The first of the two results appeared in \cite[Lemma 3.4]{BCVA}.

\begin{thm}[\cite{BCVA}] \label{thm:perfectness thm}
$\bm{\mu}\in\mathcal{L}^r$ is a perfect system if and only if there is a choice of type II multiple orthogonal polynomials $P_{\bm{n}}$ for each $\bm{n}\in\bbN^r_{\bm{N}}$, such that %when $0 \leq n_j < N_j$,
\begin{equation}\label{eq:z}{\jap{P_{\bm{n}}(x),x^{n_j}}_j \neq 0%\qquad \bm{n} \in \bbN^r_{\bm{N}-\bm{e}_j}.}
}
\end{equation}
whenever $\bm{n} \in \bbN^r_{\bm{N}-\bm{e}_j}$.
\end{thm}
\begin{proof}
Necessity of~\eqref{eq:z} is trivial from Lemma~\ref{lem:small normality lemma 2}. Let us now show sufficiency. Note that~\eqref{eq:z} is only imposed on one choice of $P_{\bm{n}}$ so this direction is not immediate from Lemma~\ref{lem:small normality lemma 2}. We use induction on $N = \abs{\bm{n}}$ to prove that $\bm{n}$ is normal whenever $0 \leq n_j \leq N_j$. The case $N = 0$ is obvious. If $N > 0$, assume $\bm{n}$ is normal whenever $\abs{\bm{n}} < N$ and take $\bm{m} = (m_1,\dots,m_r)$ such that $\abs{\bm{m}} = N$. There is some $j$ such that $\bm{m} = \bm{n} + \bm{e}_j$. $\bm{n}$ is normal by assumption, so the monic $P_{\bm{n}}$ is unique, and 
\m{\jap{P_{\bm{n}}(x),x^{n_j}}_j \neq 0,}
so $\bm{n} + \bm{e}_j$ is normal by Lemma \ref{lem:small normality lemma 2}.
\end{proof}
\begin{rem}\label{rem:perf thm}
    In fact the same proof shows that  all indices along any path $\{\bm{n}_l\}_{l=0}^m$ (where $m$ may be infinite) with $\bm{n}_0 = \bm{0}$ and $\bm{n}_{l+1}=\bm{n}_l+\bm{e}_{j_l}$, where $1\le j_l \le r$, are normal if and only if
    \begin{equation}
    \jap{P_{\bm{n}_l}(x),x^{(\bm{n}_l)_{j_l}}}_{j_l} \neq 0, \qquad l = 0,\dots,m-1,
    \end{equation}
    for a choice of type II multiple orthogonal polynomials $P_{\bm{n}_l}$.
\end{rem}

%The analogue of Theorem \ref{thm:perfectness thm} in terms of recurrence relations is the following.

\begin{thm}\label{thm:recurrence relation gives perfect system}
$\bm{\mu}\in\mathcal{L}^r$ is a perfect system if and only if there is a choice of type II multiple orthogonal polynomials $P_{\bm{n}}$ for each $\bm{n}\in\bbN^r_{\bm{N}}$, such that 
%if $n_j < N_j$ and $n_k < N_k$, then
for each $\bm{n}\in \bbN^r_{\bm{N}-\bm{e}_j-\bm{e}_l}$
\nm{eq:normality check}{P_{\bm{n} + \bm{e}_l} - P_{\bm{n} + \bm{e}_j} = c_{\bm{n},j,l}P_{\bm{n}}}
for some constants $c_{\bm{n},j,l} \neq 0$ ($j \neq l$).
\end{thm}
\begin{proof}
We prove that $\bm{n}$ is normal by induction on $M = \abs{\bm{n}}$. The cases $M = 0$ and $M= 1$ are obvious. For $M > 1$, assume $\bm{n}$ is normal whenever $\abs{\bm{n}} < M$ and take $\bm{m} = (m_1,\dots,m_r)$ such that $\abs{\bm{m}} = M$. Note that we know the normality in the case $\bm{m} = m\bm{e}_i$, so assume $m_j > 0$ and $m_l > 0$ for some $j$ and $k$ with $j \neq l$. Then $\bm{n} = \bm{m} - \bm{e}_j - \bm{e}_l$ is normal, as well as $\bm{n} + \bm{e}_j = \bm{m} - \bm{e}_l$ and $\bm{n} + \bm{e}_l = \bm{m} - \bm{e}_j$. Hence by Corollary \ref{cor:nnr cor} we get
\m{(b_{\bm{n},j} - b_{\bm{n},l})P_{\bm{n}} = P_{\bm{n} + \bm{e}_l} - P_{\bm{n} + \bm{e}_j} = c_{\bm{n},j,l}P_{\bm{n}},}
so we must have $b_{\bm{n},j} - b_{\bm{n},l} = c_{\bm{n},j,l} \neq 0$, which implies that $\bm{m} = \bm{n} + \bm{e}_j + \bm{e}_l$ is normal, by Lemma \ref{lem:normality vs recurrence} b).
\end{proof}

A related result to Theorem~\ref{thm:recurrence relation gives perfect system}
%, using the nearest neighbour relation \eqref{eq:one nnr} instead of \eqref{eq:nnr cor}, 
can be found in \cite[Prop. 3]{ADL}.

\subsection{Appendix: Proof of Theorem \ref{thm:nnrTypeI}}\label{ss:appendix}

\begin{lem} \label{lem:nnr lem 2}
Suppose $\bm{n}$ is normal. Then the vectors $\set{(A_{\bm{n}+\bm{e}_i}^{(1)},\dots,A_{\bm{n}+\bm{e}_i}^{(r)})}_{i=1}^{r}$ are linearly independent. 
\end{lem}

\begin{proof}
Consider the equation
\m{\sum_{i = 1}^r c_i (A_{\bm{n}+\bm{e}_i}^{(1)},\dots,A_{\bm{n}+\bm{e}_i}^{(r)}) = 0.}
Since $\bm{n}$ is normal $A_{\bm{n}+\bm{e}_j}^{(j)}$ has degree $n_j$, by Lemma \ref{lem:small normality lemma 2}, but $A_{\bm{n}+\bm{e}_i}^{(j)}$ has degree $< n_j$ when $i \neq j$, so $c_j = 0$, which proves linear independence.
\end{proof}

\begin{rem}
    The type II equivalent of Lemma \ref{lem:nnr lem 2} is the linear independence of $P_{\bm{n}-\bm{e_1}},\dots,P_{\bm{n}-\bm{e}_r}$, which also follows from Lemma \ref{lem:small normality lemma 2} (see \cite{DK04,Ismail}).
\end{rem}

\begin{proof}[Proof of Theorem~\ref{thm:nnrTypeI}]
We have
\m{\sum_{k = 1}^r \jap{xA_{\bm{n}}^{(k)}(x) - A_{\bm{n}-\bm{e}_j}^{(k)}, x^p}_k = 0, \qquad p = 0,1,\dots,\abs{\bm{n}}-2.}
Now choose $d_{\bm{n},j}$ such that
\m{\sum_{k = 1}^r \jap{xA_{\bm{n}}^{(k)}(x) - A_{\bm{n}-\bm{e}_j}^{(k)} - d_{\bm{n},j}A_{\bm{n}}^{(k)} ,x^p}_k = 0, \qquad p = 0,1,\dots,\abs{\bm{n}}-1.}
The polynomial $B_{\bm{n}}^{(k)} = xA_{\bm{n}}^{(k)}(x) - A_{\bm{n}-\bm{e}_j}^{(k)} - d_{\bm{n},j}A_{\bm{n}}^{(k)}$ has degree $n_k$. Hence the above system of orthogonality relations is homogeneous with matrix $M_{\bm{n}}^t$ with $r$ columns added, where $M_{\bm{n}}$ is given by \eqref{eq:matrix}. This matrix has nullity $r$, since $M_{\bm{n}}^t$ has nullity $0$, by normality. Hence the solution space is of dimension $r$, and $(B_{\bm{n}}^{(1)},\dots,B_{\bm{n}}^{(r)}) = (A_{\bm{n}+\bm{e}_i}^{(1)},\dots,A_{\bm{n}+\bm{e}_i}^{(r)})$ are solutions for each $i = 1,\dots,r$. By Lemma \ref{lem:nnr lem 2} every solution is thus on the form
\m{B_{\bm{n}}^{(k)} = \sum_{i=1}^r c_{\bm{n},i}A_{\bm{n}+\bm{e}_i}^{(k)}.}
In other words we have the recurrence relation
\m{xA_{\bm{n}}^{(k)}(x) = A_{\bm{n}-\bm{e}_j}^{(k)} + d_{\bm{n},j}A_{\bm{n}}^{(k)} + \sum_{i=1}^r c_{\bm{n},i}A_{\bm{n}+\bm{e}_i}^{(k)}.}

To show that $d_{\bm{n},j} = b_{\bm{n}-\bm{e}_j,j}$, we compare with the recurrence relation
\m{xP_{\bm{n}-\bm{e}_j}(x) = P_{\bm{n}}(x) + b_{\bm{n}-\bm{e}_j,j}P_{\bm{n}-\bm{e}_j}(x) + \sum_{i = 1}^{r}a_{\bm{n}-\bm{e}_j,i}P_{\bm{n} - \bm{e}_j - \bm{e}_i}(x).}
From here we can write
\m{\sum_{k = 1}^r \jap{xA_{\bm{n}}^{(k)}(x),P_{\bm{n}-\bm{e}_j}(x)} & = \sum_{k = 1}^r \jap{A_{\bm{n}}^{(k)}(x),P_{\bm{n}}(x)} \\ & + b_{\bm{n}-\bm{e}_j,j}\sum_{k = 1}^r\jap{ A_{\bm{n}}^{(k)}(x),P_{\bm{n}-\bm{e}_j}(x)} \\ & \sum_{i = 1}^r a_{\bm{n}-\bm{e}_j,i} \sum_{k = 1}^r \jap{A_{\bm{n}}^{(k)}(x),P_{\bm{n} - \bm{e}_j - \bm{e}_i}(x)}.}
The first term vanishes by the orthogonality relations of $P_{\bm{n}}$, and the last term vanishes by the orthogonality relations of $\bm{A}_{\bm{n}}$. The second factor in the middle term is $1$, so we have
\m{\sum_{k = 1}^r \jap{xA_{\bm{n}}^{(k)}(x),P_{\bm{n}-\bm{e}_j}(x)}_k = b_{\bm{n}-\bm{e}_j,j}.}
On the other hand, if we instead apply a similar argument using the type I recurrence relation we end up with
\m{\sum_{k = 1}^r \jap{xA_{\bm{n}}^{(k)}(x),P_{\bm{n}-\bm{e}_j}(x)}_k = d_{\bm{n},j},}
so $d_{\bm{n},j} = b_{\bm{n}-\bm{e}_j,j}$.

Write $\kappa_{\bm{n},j}$ for the degree $n_j-1$ coefficient of $A_{\bm{n}}^{(j)}$, so that $A_{\bm{n}}^{(j)}(x) = \kappa_{\bm{n},j}x^{n_j-1} + o(x^{n_j-1})$. By comparing the degree $n_i$ coefficients in (\ref{eq:type 1 recurrence relation}) we see that $\kappa_{\bm{n},i} = c_{\bm{n},i}\kappa_{\bm{n}+\bm{e}_i,i}$. By Lemma \ref{lem:small normality lemma 2} we know that $\kappa_{\bm{n}+\bm{e}_i,i} \neq 0$, so we get
\m{c_{\bm{n},i} = \frac{\kappa_{\bm{n},i}}{\kappa_{\bm{n}+\bm{e}_i,i}},}
which shows independence of $j$. 

What remains is to show that $c_{\bm{n},i} = a_{\bm{n},i}$. Note that
\m{\sum_{k = 1}^r \jap{A_{\bm{n}}^{(k)}(x),P_{\bm{n}-\bm{e}_i}(x)}_k = 1,}
by the orthogonality relations of $\bm{A}_{\bm{n}}$, assuming $\bm{n}-\bm{e}_i$ is normal. By the orthogonality relations of $P_{\bm{n}-\bm{e}_i}$ we instead get 
\m{\sum_{k = 1}^r \jap{A_{\bm{n}}^{(k)}(x),P_{\bm{n}-\bm{e}_i}(x)}_k & = \jap{A_{\bm{n}}^{(i)}(x),P_{\bm{n}-\bm{e}_i}(x)}_i \\ & = \kappa_{\bm{n},i} \jap{x^{n_i-1},P_{\bm{n}-\bm{e}_i}(x)}_i.}
Similarly, if $\bm{n}+\bm{e}_i$ is normal then 
\m{\kappa_{\bm{n}+\bm{e}_i,i} \jap{x^{n_i},P_{\bm{n}}(x)} = 1,}
so we get
\m{a_{\bm{n},i} = \frac{\jap{x^{n_i},P_{\bm{n}}(x)}}{\jap{x^{n_i - 1},P_{\bm{n}-\bm{e}_i}(x)}} = \frac{\kappa_{\bm{n},i}}{\kappa_{\bm{n}+\bm{e}_i,i}} = c_{\bm{n},i},}
which completes the proof.
\end{proof}

\section{Main results}

Let us setup the following notation for the remainder of this paper. We write
\begin{align*}
    \bm{\mu} & = (\mu_1,\dots,\mu_r) \in \mathcal{L}^r, \\
    \bm{\nu} & =(\mu_1,\ldots,\mu_{r-1})\in\mathcal{L}^{r-1},
\end{align*}
that is, $\bm{\nu}$ is $\bm{\mu}$ with $\mu_r$ removed, where $r\ge 2$. 

As usual we allow $\mu_j\in\mathcal{L}_{N_j}$ with $N_j\le \infty$, and denote $\bm{N} = (N_1,\ldots,N_r)$. We use 
%$\bm{n}=(n_1,\ldots,n_r)$ 
$\bm{n}$ for a multi-index in $\bbN^r_{\bm{N}}$, and $\bm{k}=(k_1,\ldots,k_{r-1})$ for a multi-index in $\bbN^{r-1}_{\bm{K}}$, where we denote $\bm{K} = (N_1,\ldots,N_{r-1})$.

Finally, it is clear that the type II multiple orthogonal polynomials  for $\bm{\nu}$ at a location $\bm{k}$ coincide with type II multiple orthogonal polynomials for $\bm{\mu}$ at the location $(\bm{k},0)$. %$=:\bm{n}$. 
So we use the same label $P$ %$P_{\bm{n}}=
for type II polynomials for $\bm{\mu}$ and $\bm{\nu}$ interchangeably, as it cannot lead to any ambiguity: $P_{(\bm{k},0)}=P_{\bm{k}}$, and similarly for type I polynomials $A^{(j)}_{(\bm{k},0)} = A^{(j)}_{\bm{k}}$ for $1\le j \le r-1$.

\subsection{Christoffel Transforms of Type II Polynomials}\label{ss:CT type 2} \hfill\\

Let $\mu_r\in\mathcal{L}_m$ with $m = N_r <\infty$, and write $\Phi(x) = \prod_{j=1}^m (x-z_j)$ for the unique monic orthogonal polynomial of degree $m$ with respect to $\mu_r$. 
For example, if all $z_j$'s are distinct (which is true if $\mu_r\in\mathcal{M}_m\subset\mathcal{L}_m$), then $\mu_r$ is of the form
\begin{equation}\label{eq:simpleMu}
    \sum_{j=1}^m w_j\delta_{z_j} %, \quad \mbox{with } \sum_{j=1}^M w_j = 1 \mbox{ and } w_j\ne 0 \mbox{ for all } j
\end{equation} 
where $w_j \in \C\setminus\set{0}$ for each $j = 1,\dots,m$. 
%Note that these $m$ points $z_j$ are real and distinct if $\mu_r$ happen to belong to $\mathcal{M}$ (they are the points in the support of $\mu_r$), 
But in general $z_j$'s may overlap leading to more complicated functionals, see Remark~\ref{rem:Gus}. 

%$\bm{N}$ will now represent the index $(N_1,\dots,N_{r-1}) \in (\nat\cup\set{\infty})^{r-1}$, determined by $\mu_j \in \mathcal{L}_{N_j}$ for $1\le j \le r-1$ as before. If $\bm{n} = (n_1,\dots,n_{r-1}) \in \bbN^{r-1}_{\bm{N}}$ then we write $(\bm{n},n_r) = (n_1,\dots,n_r) \in \bbN^{r}_{(\bm{N},m)}$. Note that the type II polynomials for the system $\bm{\nu}$ are then $P_{(\bm{n},0)}$ with $\bm{n}\in\bbN^{r-1}_{\bm{N}}$. We start with the following determinantal formula for $P_{(\bm{n},m)}$, which we connect with generalized Christoffel's Theorem in Section~\ref{ss:CT type 1}.

%We now continue with the same notation as in Section \ref{determinantal formula}. 
%$\Phi(x)=\prod_{j=1}^m (x-z_j)$ as in the previous section, with arbitrary $z_j\in\bbC$.
%Observe that the only condition imposed on $\mu_r\in\mathcal{L}_m$ is $m < \infty$, and we write $\Phi(x) = \prod_{j = 1}^{m}(x-z_j)$ for its degree $m$ orthogonal polynomial. 

Consider the Christoffel transforms $\widehat{\mu}_j = \Phi\mu_j$ for $1 \leq j \leq r-1$, as in~\eqref{eq:ChrL}. We refer to 
$$\widehat{\bm{\nu}} = \Phi\bm{\nu} = (\widehat\mu_1,\ldots,\widehat\mu_{r-1})
$$
as the Christoffel transform of $\bm{\nu} = (\mu_1,\ldots,\mu_{r-1})$. We write $P_{\bm{k}}$ for the type II polynomials with respect to $\bm{\nu}$ and $\widehat{P}_{\bm{k}}$ for the type II polynomials with respect to $\widehat{\bm{\nu}}$, $\bm{k}\in\bbN^{r-1}_{\bm{K}}$. Let $\{a_{\bm{k},j},b_{\bm{k},j}\}$  denote the recurrence coefficients of $\bm{\nu}$, and $\{\widehat{a}_{\bm{k},j},\widehat{b}_{\bm{k},j}\}$ denote the recurrence coefficients of $\widehat{\bm{\nu}}$. The connection between the systems $\widehat{\bm{\nu}}$ and $\bm{\mu} = (\bm{\nu},\mu_r)$ is summarized by the following result.

\begin{thm}\label{thm:ct moprl}
An index $(\bm{k},m)$ is normal for the system $\bm{\mu} = (\bm{\nu},\mu_r)$ if and only if $\bm{k}$ is normal for the system $\widehat{\bm{\nu}}$. In that case we have $P_{(\bm{k},m)}(x) = \Phi(x)\widehat{P}_{\bm{k}}(x)$.
\end{thm}
\begin{proof}
Suppose $(\bm{k},m)$ is normal. By the orthogonality relations 
\nm{eq:ct orthogonality r}{\jap{P_{(\bm{k},m)}(x),x^p}_r = 0, \qquad p = 0,1,\dots,m-1,}
we must have $P_{(\bm{k},m)}(x) = \Phi(x)Q(x)$ for some polynomial $Q$ with $\deg{Q} \leq \abs{\bm{k}}$, by Lemma \ref{lem:finite rank}. For the other inner products we then have 
\nm{eq:ct orthogonality}{\jap{\Phi(x)Q(x),x^p}_j = 0, \qquad p = 0,1,\dots,k_j-1, \qquad j = 1,\dots,r-1.}
Hence $Q$ must be a multiple orthogonal polynomial with respect to $\widehat{\bm{\nu}}$. Conversely, any multiple orthogonal polynomial $Q = \widehat{P}_{\bm{k}}$ with respect to $\widehat{\bm{\nu}}$ satisfies \eqref{eq:ct orthogonality} by definition. \eqref{eq:ct orthogonality r} is then satisfied with $P_{(\bm{k},m)}(x)$ replaced by $\Phi(x)\widehat{P}_{\bm{k}}(x)$, by Lemma \ref{lem:finite rank}. Now, $(\bm{k},m)$ is normal for $(\bm{\nu},\mu_r)$ if and only if every $P_{(\bm{k},m)}$ has degree $\abs{\bm{k}}+m$, if and only if every $\widehat{P}_{\bm{k}}$ has degree $\abs{\bm{k}}$, if and only if $\bm{k}$ is normal for $\widehat{\bm{\nu}}$. 
% If $\bm{n}$ is not normal for $\widehat{\bm{\mu}}$ then there is some non-zero $\widehat{P}_{\bm{n}}$ with $\deg{\widehat{P}_{\bm{n}}} < \abs{\bm{n}}$, and
%    \m{\jap{\Phi(x)\widehat{P}_{\bm{n}}(x),x^p}_j = 0, \qquad p = 0,1,\dots,n_j-1, \qquad j = 1,\dots,r-1.}
%    By Lemma \ref{lem:finite rank} we also have 
%    \m{\jap{\Phi(x)\widehat{P}_{\bm{n}}(x),x^p}_r = 0, \qquad p = 0, 1, \dots}
%    so that $\Phi(x)\widehat{P}_{\bm{n}}$ is a multiple orthogonal polynomial with respect to $(\mu,\mu_r)$ for the index $(\bm{n},m)$. Then $(\bm{n},m)$ cannot be normal since $\deg{\Phi(x)\widehat{P}_{\bm{n}}} \leq \abs{\bm{n}}$. We also 
\end{proof}

Now we find the determinantal formula, together with a necessary and sufficient condition on the perfectness of $\widehat{\bm{\nu}}$. %Recall Theorem~\ref{thm:ct moprl} where we observed that $\widehat{P}_{\bm{k}}(x) = \Phi(x)^{-1} {P}_{(\bm{k},m)}(x)$.

\begin{thm}\label{thm:ChrFull}
    Suppose that $\bm{\nu} \in \mathcal{L}_{\infty}^{r-1}$ is {a} perfect {system}, and all $\{z_j\}_{j=1}^m$ are distinct. %\mvno{Then $\widehat{\bm{\nu}}$ is perfect if and only if $D_{\bm{n}} = \det(P_{\bm{n}+\bm{s}_{m-j}}(z_k))_{j,k=1}^m \ne 0$ for any path of $\bbN^{r-1}$-indices $\set{\bm{s}_j}_{j = 0}^m$.}
    % $\N_{\bm{K}}^{r-1}$-indices $\set{\bm{n} + \bm{s}_j}_{j = 0}^m$ with $\bm{s}_j\in\bbN^{r-1}$ and $|\bm{s}_j| = j$ for each $0 \leq j \leq m$, where $\bm{n}+\bm{s}_j$ and $P_{\bm{n}+\bm{s}_j}$ may have to be modified according to Remark \ref{rem:modification} (when $N_j < \infty$ for all $j = 1,\dots,r-1$). 
    %\mvno{In that case,}
\begin{itemize}
\item[i)] {If $(\bm{k},m)$ is normal,
%$\widehat{\bm{\nu}}$ is perfect, 
then for any sequence of $\bbN^{r-1}$-indices $\set{\bm{s}_j}_{j = 0}^m$ with $\bm{s}_0 = 0$ and $\abs{\bm{s}_j} = j$, we have $D_{\bm{k}} = \det(P_{\bm{k}+\bm{s}_{m-i}}(z_j))_{i,j=1}^m \ne 0$ and the determinantal formula}%We have the determinantal formula 
\nm{eq:ct det}{
        \widehat{P}_{\bm{k}}(x) =
    \Phi(x)^{-1}D_{\bm{k}}^{-1}
    \det\begin{pmatrix}
P_{\bm{k} + \bm{s}_m}(x) & P_{\bm{k}  + \bm{s}_{m-1}}(x) & \cdots & P_{\bm{k} }(x) \\
P_{\bm{k} + \bm{s}_m}(z_1) & P_{\bm{k} + \bm{s}_{m-1}}(z_1) & \cdots & P_{\bm{k}}(z_1) \\
\vdots & \vdots & \ddots & \vdots \\
P_{\bm{k} + \bm{s}_m}(z_m) & P_{\bm{k} + \bm{s}_{m-1}}(z_m) & \cdots & P_{\bm{k}}(z_m) \end{pmatrix}, \\ \bm{k}\in\bbN^{r-1}_{\bm{K}}.}
\item[{ii)}] {If $D_{\bm{k}} = \det(P_{\bm{k}+\bm{s}_{m-i}}(z_j))_{i,j=1}^m \ne 0$ for any $\bm{k}\in\bbN^{r-1}_{\bm{K}}$ and any sequence of $\bbN^{r-1}$-indices $\set{\bm{s}_j}_{j = 0}^m$ with $\bm{s}_0 = 0$ and $\abs{\bm{s}_j} = j$, then $\widehat{\bm{\nu}}$ is perfect.}
%\item[(ii{i})] %For any $1\le j\le r-1$,
%        \begin{align}
%            &\widehat{a}_{\bm{k},j}  =a_{(\bm{k},m),j}, &\bm{k}\in\bbN^{r-1}_{\bm{K}}, \qquad &j = 1,\dots,r-1, \\
%           &\widehat{b}_{\bm{k},j}  =b_{(\bm{k},m),j}, &\bm{k}\in\bbN^r_{\bm{K}-\bm{e}_j}, \qquad &j = 1,\dots,r-1.
%        \end{align}
%In particular, the recurrence coefficients of $\widehat{\bm{\nu}}$ can be computed from the recurrence coefficients of $\bm{\nu}$ using the CC algorithm (see Theorem~\ref{thm:NNCC}) applied to ${\bm{\mu}}$.
    \end{itemize}
\end{thm}

%Let us also define $\widecheck{\bm{\mu}}\in\mathcal{L}^{r+1}$ to be the system $(\mu_1,\ldots,\mu_r,\delta_{z_0})$, where $\delta_{z_0}\in\mathcal{L}$ is the Dirac delta at $z_0$ (i.e., $\delta_{z_0}[x^n]:=z_0^n$). Note that with $\widecheck{N}_{r+1}=|\supp\,\delta_{z_0}|=1$.

\begin{proof}
{\it i)} If $D_{\bm{n}} = 0$ then the columns of $(P_{\bm{k}+\bm{s}_{m-i}}(z_j))_{i,j=1}^m$ are linearly dependent. This would imply 
\m{P(x) = \sum_{j = 0}^{m-1}c_jP_{\bm{k}+\bm{s}_j}(x) = 0, \qquad x = z_1,\dots,z_r,}
for some $(c_1,\dots,c_r) \neq (0,\dots,0)$. However, since we have a linear combination of polynomials with different degrees, $P$ cannot be identically $0$. $P$ satisfies the orthogonality conditions at the index $(\bm{k},m)$ with respect to $\mu_j$ with $1 \leq j \leq r-1$ since $\bm{k} \leq \bm{k} + \bm{s}_i$ for every $i = 0,\dots,m-1$, and with respect to $\mu_r$ by Lemma \ref{lem:finite rank} since $P$ is divisible by $\Phi$. Since $\deg{P} \leq \abs{\bm{k}}+m-1$ this would contradict the normality of $(\bm{k},m)$.

The determinant in~\eqref{eq:ct det} is a polynomial of degree $\abs{\bm{k}}+m$. It is a linear combination of $P_{\bm{k}},\dots,P_{\bm{k} + \bm{s}_m}$, and $\bm{k} \leq \bm{k}+\bm{s}_i$, so it is orthogonal to $x^p$ with respect to $\mu_j$ for $p = 0,\dots,n_j-1$, $j = 1,\dots,r-1$. Since it vanishes at $z_1,\dots,z_m$ it is orthogonal to everything with respect to $\mu_r$, by Lemma \ref{lem:finite rank}. Hence the normality of $(\bm{k},m)$ implies that $P_{(\bm{k},m)}$ coincides with the determinant in~\eqref{eq:ct det} divided by the normalization constant $D_{\bm{n}}$. Finally, Theorem~\ref{thm:ct moprl} implies~\eqref{eq:ct det}.

%If $\bm{\nu}$ and $\widehat{\bm{\nu}}$ are perfect then by Theorem \ref{thm:ct moprl} the conditions of Theorem \ref{thm:christoffel formula moprl type 2} are satisfied, or at least the modification in Remark \ref{rem:modification}. Hence $D_{\bm{n}} \neq 0$, and the determinantal formula \eqref{eq:MOPRLDeterminantalNormalize} holds.

{\it ii)} Conversely, assume $\bm{\nu}$ is perfect and $D_{\bm{k}}\ne 0$ for any choice of path %and polynomials 
for all $\bm{k}\in\bbN^{r-1}_{\bm{K}}$. Fix any path $\set{\bm{k}_l}_{k = 0}^M$  of multi-indices in $\bbN^{r-1}_{\bm{K}}$ with $\bm{k}_0 = \bm{0}$ and $\bm{k}_{l+1} = \bm{k}_l + \bm{e}_{j_l}$. 
For each $\bm{k}_l\in\N_{\bm{K}}^{r-1}$, denote $Q_{\bm{k}_l}$ to be the polynomial on the right-hand side of \eqref{eq:ct det}, where we choose $\bm{s}_j$ in such a way that $\bm{k}+\bm{s}_j$ for each $j$ belongs to the chosen path $\set{\bm{k}_l}_{l = 0}^M$. It is clear that $Q_{\bm{k}_l}$ is monic of degree $\abs{\bm{k}_l}$ and satisfies all the type II orthogonality conditions for $\widehat{\bm{\nu}}$. Let us show that $Q_{\bm{k}_l}$ satisfy the conditions of Remark~\ref{rem:perf thm} with respect to the system $\widehat{\nu}$. Indeed, by~\eqref{eq:ct det}, 
\begin{equation}
\jap{\Phi(x)Q_{\bm{k}_l}(x),x^{(\bm{k}_l)_{j_l}}_{j_l}} = \frac{(-1)^m D_{\bm{k}_{l+1}}}{D_{\bm{k}_l}} \jap{P_{\bm{k}_l}(x),x^{(\bm{k}_l)_{j_l}}}_{j_l}, %\neq 0, \qquad k = 0,\dots,M-1.}
\end{equation}
which is non-zero by the assumptions and Lemma \ref{lem:small normality lemma 2} applied to $\bm{\nu}$. Hence by Remark~\ref{rem:perf thm} every index along the path $\set{\bm{k}_l}_{l = 0}^M$ is normal for $\widehat{\bm{\nu}}$, and by taking different paths that cover all $\N_{\bm{K}}^{r-1}$-indices, we see that $\widehat{\bm{\nu}}$ is perfect.

%{\it iii)} is clear from Theorem~\ref{thm:ct moprl}, Theorem~\ref{thm:zero coefficients finite systems}, and Theorem~\ref{thm:NNCC}.
\end{proof}

%\begin{proof}
%    If $\bm{\nu}$ and $\widehat{\bm{\nu}}$ is perfect then the conditions of Theorem~\ref{thm:christoffel formula moprl} hold. Then~\eqref{eq:MOPRLDeterminantal} shows that $P_{\bm{n}}(z_0)\ne 0$ for all $\bm{n}\in\bbN^{r-1}_{\bm{N}}$. Determinantal formulas ~\eqref{eq:MOPRLDeterminantal} and ~\eqref{eq:MOPRLDeterminantalTypeI} then easily imply~\eqref{eq:ChrP1} and \eqref{eq:ChrP1A}.
    
%    Conversely, assume $\bm{\nu}$ is perfect and $P_{\bm{n}}(z_0)\ne 0$ for all $\bm{n}\in\bbN^{r-1}_{\bm{N}}$. Denote $\widecheck{P}_{\bm{n}}(x)$ to be the right-hand side of~\eqref{eq:ChrP1}. It is clear that $\widecheck{P}_{\bm{n}}(x)$ are monic polynomials of degree $|\bm{n}|$ that satisfy all the type II orthogonality conditions for $\widehat{\nu}$ at $\bm{n}$. They also satisfy the conditions of Theorem~\ref{thm:perfectness thm} with respect to $\widehat{\nu}$. Therefore $\widehat{\nu}$ is also perfect.
%\end{proof}

\begin{rem}\label{rem:modification}
    If we remove the assumption $\bm{\nu} \in \mathcal{L}_{\infty}^{r-1}$ then we get complications when $N_j<\infty$ for all $j = 1,\dots,r-1$, since a sequence of normal indices $\set{\bm{k} + \bm{s}_j}_{j = 0}^m$ will not exist for every $\bm{k} \in \N_{\bm{K}}^{r-1}$. Indeed, if $|\bm{k}| > |\bm{K}|-m$, then $\bm{k}+\bm{s}_j$ will have to be outside of $\N_{\bm{K}}^{r-1}$ for $j> |\bm{K}|-|\bm{k}|$ and therefore will not be normal. Theorem~\ref{thm:ChrFull}
    %~\ref{thm:christoffel formula moprl type 2NEW}
    still works if one replaces $P_{\bm{k}+\bm{s}_j}(x)$ %$P_{(\bm{n}+\bm{s}_j,0)}(x)$
    for those $j$'s with, for example, $x^{\abs{\bm{k}}-\abs{\bm{K}}+j} P_{\bm{K}} (x)$. The above proof works without any significant change. More generally, as long as $\deg{P_{\bm{k}+\bm{s}_j}} = \abs{\bm{k}} + j$ for each $j = 1,\dots,m$, the proof still holds even if these indices are not normal. 
\end{rem}

\begin{rem}\label{rem:modification 2} If we remove the assumption that $\{z_j\}_{j=1}^m$ are distinct, then the theorem still holds %(the second proof works) 
but the matrix in 
~\eqref{eq:ct det}
should be modified as follows. If $z_j$ is a root of $\Phi$ of multiplicity $l$, then instead of $l$ copies of the row
$\begin{pmatrix} P_{\bm{k} + \bm{s}_m}(z_j) & P_{\bm{k} + \bm{s}_{m-1}}(z_j) & \cdots & P_{\bm{k}}(z_j) \end{pmatrix}$
we have the rows
\m{\begin{pmatrix} P_{\bm{k} + \bm{s}_m}(z_j) & P_{\bm{k} + \bm{s}_{m-1}}(z_j) & \cdots & P_{\bm{k}}(z_j) \\ 
P_{\bm{k} + \bm{s}_m}'(z_j) & P_{\bm{k} + \bm{s}_{m-1}}'(z_j) & \cdots & P_{\bm{k}}'(z_j) \\
\vdots & \vdots & \ddots & \vdots \\
P_{\bm{k}+\bm{s}_m}^{(l-1)}(z_j) & P_{\bm{k} + \bm{s}_{m-1}}^{(l-1)}(z_j) & \cdots & P_{\bm{k}}^{(l-1)}(z_j)
\end{pmatrix},}
and the corresponding modifications should be done in the determinant $D_{\bm{k}}$
%\eqref{eq:MOPRLDeterminantalNormalizeNEW} 
%and the rows in the normalizing constant \eqref{eq:MOPRLDeterminantalNormalizeNEW} should be modified correspondingly 
to make $\widehat{P}_{\bm{k}}$ monic.
This change ensures that $z_j$ is a root of multiplicity $l$ of the right-hand side of~\eqref{eq:ct det}. %, and the proof of the theorem works out with minimal changes.
The previous remark still holds.
\end{rem}

%%%%%%%%%%%%%%%%%%%%%%%%%%%%%%%%%%%%%%%%%%%

%\subsection{Computation Algorithm for the Recurrence Coefficients}\label{ss:NNCC}
\subsection{Nearest Neighbour Recurrence Coefficients for the Christoffel Transforms}\label{ss:NNCC}
\hfill\\

In Theorem~\ref{thm:ct moprl} we stated the relationship between the type II polynomials of $\widehat{\bm{\nu}}$ and $\bm{\mu}$ (for type I polynomials, see Theorem~\ref{thm:det formula type I} below).
It will come as no surprise that the nearest neighbour recurrence coefficients of $\widehat{\bm{\nu}}$ and $\bm{\mu}$ are also related.

\begin{thm}\label{thm:zero coefficients finite systems}
    Suppose $\bm{\mu}\in\mathcal{L}^r$ with $\mu_{r}\in\mathcal{L}_{m}$, where $m<\infty$. Then:
    \begin{itemize}
        \item[i)] $a_{(\bm{k},m),r} = 0 $ for all $\bm{k}\in\bbN^{r-1}_{\bm{K}}$.

        \item[ii)]  %For any $1\le j\le r-1$,
        Additionally, suppose $\bm{\nu}$ and $\widehat{\bm{\nu}}$ are perfect. Then
        \begin{alignat}{5}
            &a_{(\bm{k},m),j}=\widehat{a}_{\bm{k},j}  , \qquad &\bm{k}&\in\bbN^{r-1}_{\bm{K}}, \qquad &j& = 1,\dots,r-1, \\
            &b_{(\bm{k},m),j} = \widehat{b}_{\bm{k},j}  , &\bm{k}&\in\bbN^r_{\bm{K}-\bm{e}_j}, \qquad &j& = 1,\dots,r-1.
        \end{alignat}
    \end{itemize}  
\end{thm}
\begin{proof}
    {\it i)} This follows immediately from Lemma~\ref{lem:normality vs recurrence} a) along with Proposition \ref{rem:perf} (except in the case $\bm{k} = \bm{K}$, in which case it follows by definition through Remark \ref{rem:perfectAB}, for the sake of convenience).

    {\it ii) } This is clear from Theorem~\ref{thm:ct moprl} and {\it i)}.
\end{proof}

%We now want to compute the recurrence coefficients of the polynomials defined in \eqref{eq:MOPRLDeterminantal}. For this section, we write $\bm{n}$ for indices $(n_1,\dots,n_r) \in \N_{\bm{N}}^r$ (rather than for $\N_{\bm{N}}^{r-1}$-indices, as in the previous section)  and $\bm{\mu}$ for perfect systems $(\mu_1,\ldots,\mu_r)\in\mathcal{L}^r$ (in the sense of Definition~\ref{def:perf}). 

Now let 
%$\bm{\mu}=(\mu_1,\ldots,\mu_r)\in\mathcal{L}^r$
$\bm{\mu}\in\mathcal{L}^r$ be perfect in the extended sense of Definition~\ref{def:perf}. Suppose that the Jacobi coefficients of each $\mu_j\in\mathcal{L}$ is given, i.e., $\{a_{n\bm{e}_j,j}\}_{n=1}^{N_j}$  and $\{b_{n\bm{e}_j,j}\}_{n=0}^{N_j-1}$ for each $1\le j \le r$ (here $a_{n\bm{e}_j,N_j}=0$ if $N_j<\infty$). It is clear that this information is sufficient to determine all the NNRR coefficients  $\{a_{\bm{n},j}\}_{\bm{n}\in\bbN^r_{\bm{N}}}$  and $\{b_{\bm{n},j}\}_{\bm{n}\in\bbN^r_{\bm{N}-\bm{e}_j}}$ for each $1\le j \le r$.
%(see Remark~\ref{rem:perfectAB}).  
Filipuk, Haneczok, and Van Assche in \cite{Computing NNR} put forward a recursive algorithm for this, based on the compatibility conditions~\eqref{eq:CC1}--\eqref{eq:CC3}. In the next result, we state their algorithm in the more general setting $\bm{\mu}\in\mathcal{L}^r$, where we allow $N_j$'s to be finite. The proof, which is a straightforward application of Theorem~\ref{thm:compatibility conditions}, is following the same lines as the proof of~\cite[Thm 3.3]{Computing NNR} except that extra care needs to be taken with the indices $\bm{n}$ with $n_j=N_j$ if $N_j<\infty$. Indeed, the $a_{\bm{n},j}$ coefficients vanish there, as we just showed in Theorem~\ref{thm:zero coefficients finite systems}{\it i)}.

%We also note the following property of the indices of the form $(\bm{n},m)$. 

%This theorem suggests that we could view the coefficients $a_{(\bm{k},m),j}$ and $b_{(\bm{k},m),j}$ as the nearest neighbour recurrence coefficients $\widehat{a}_{\bm{k},j}$ and $\widehat{b}_{\bm{k},j}$ for some other system of $r-1$ measures. This will be made explicit in Section~\ref{ss:CT type 2}.

%for the case of infinitely supported measures. For the sake of the completeness, we include a quick outline of the proof. 

\begin{thm}[The CC Algorithm, \cite{Computing NNR}]\label{thm:NNCC}
Suppose $\bm{\mu} = (\mu_1,\dots,\mu_r)$ is perfect. Given the Jacobi coefficients of each $\mu_j$, $j = 1,\dots,r$, the following algorithm produces all the NNRR coefficients: 
\begin{align*}
    & \mbox{for all } 1\le j, k\le r \mbox{ with } j\ne k:
    \\
    & \qquad a_{\bm{n},j} := 0  \mbox{ for all } \bm{n}\in\bbN^r_{\bm{N}} \mbox{ with } n_j=0; %\mbox{ and all } 1\le j \le r,
    \\
    & \qquad \delta_{\bm{0},j,k}:= b_{\bm{0},k} - b_{\bm{0},j};
    \\
    & \mbox{for all } d\in\bbN:
    \\
    & \qquad \mbox{for all } 1\le j, k\le r \mbox{ with } j\ne k:
    \\
    & \qquad\qquad \mbox{for all } \bm{n}\in\N_{\bm{N}}^r \mbox{ with } n_k>0,n_j>0, |\bm{n}|=d:  
    \\
    & \qquad\qquad\qquad a_{\bm{n} ,j} =
    \begin{cases}
         0 &\mbox{if } a_{\bm{n}-\bm{e}_k,j} = 0,\\
         a_{\bm{n}-\bm{e}_k,j}
    \frac{\delta_{\bm{n}-\bm{e}_k,k,j}}{\delta_{\bm{n}-\bm{e}_k-\bm{e}_j,k,j}}  &\mbox{otherwise};
    \end{cases}
     %& \qquad\qquad\qquad a_{\bm{n} ,j} =0 \mbox{ if } a_{\bm{n}-\bm{e}_k,j} = 0, \mbox{ otherwise}
    %\\
    %& \qquad\qquad\qquad a_{\bm{n} ,j} = a_{\bm{n}-\bm{e}_k,j}
    %\frac{\delta_{\bm{n}-\bm{e}_k,k,j}}{\delta_{\bm{n}-\bm{e}_k-\bm{e}_j,k,j}};
    %\frac{b_{\bm{n}-\bm{e}_k,j} - b_{\bm{n}-\bm{e}_k,k}}{b_{\bm{n} - \bm{e}_j-\bm{e}_k,j} - b_{\bm{n} - \bm{e}_j-\bm{e}_k,k}}
    \\
    & \qquad \mbox{for all } 1\le j, k\le r \mbox{ with } j\ne k:
    \\
    & \qquad\qquad \mbox{for all } \bm{n}\in\N_{\bm{N}-\bm{e}_j}^r \mbox{ with } n_k>0, |\bm{n}|=d: %N_j>n_j\ge 0;  
    \\
    & \qquad\qquad \qquad
    b_{\bm{n},j} = b_{\bm{n}-\bm{e}_k,j} + \frac{\sum_{i=1}^r a_{\bm{n} ,i} - \sum_{i=1}^r a_{\bm{n} + \bm{e}_j-\bm{e}_k,i}}{\delta_{\bm{n}-\bm{e}_k,k,j}};
    %b_{\bm{n}-\bm{e}_k,k} - b_{\bm{n}-\bm{e}_k,j} 
    \\
    & \qquad \mbox{for all } 1\le j, k\le r \mbox{ with } j\ne k:
    \\
    & \qquad\qquad \mbox{ for all } \bm{n}\in\N_{\bm{N}-\bm{e}_j-\bm{e}_k}^r, |\bm{n}|=d:
    \\
    & \qquad\qquad \qquad 
     \delta_{\bm{n},j,k}:= b_{\bm{n},k} - b_{\bm{n},j}.
\end{align*}
\end{thm}
\begin{proof}
    Assume all coefficients of indices with size $< d$ are computed and consider indices $\bm{n}$ with $\abs{\bm{n}} = d$. Using \eqref{eq:CC3.2} we can compute the coefficients $a_{\bm{n},j}$ when $0 < n_j < N_j$ by
    \m{a_{\bm{n},j} = a_{\bm{n}-\bm{e}_k,j}
    \frac{\delta_{\bm{n}-\bm{e}_k,k,j}}{\delta_{\bm{n}-\bm{e}_k-\bm{e}_j,k,j}},} 
    where $k \neq j$ and $n_k > 0$ (if such a $k$ does not exist then $\bm{n} = d\bm{e}_j$, in which case $a_{\bm{n},j}$ is already given). If $n_j = 0$ then we have $a_{\bm{n},j} = 0$ by definition, and if $n_j = N_j$ then we have $a_{\bm{n},j} = 0$ by Theorem \ref{thm:zero coefficients finite systems}, which agrees with the algorithm. 
    
    With every $a_{\bm{n},j}$ computed, we can then use \eqref{eq:CC2.2} to compute the coefficients $b_{\bm{n},j}$ by
    \m{b_{\bm{n},j} = b_{\bm{n}-\bm{e}_k,j} + \frac{\sum_{i=1}^r a_{\bm{n} ,i} - \sum_{i=1}^r a_{\bm{n} + \bm{e}_j-\bm{e}_k,i}}{\delta_{\bm{n}-\bm{e}_k,k,j}},}
    with $k$ chosen such that $k \neq j$ and $n_k > 0$. Hence all the NNRR coefficients can be computed, by induction on $\abs{\bm{n}}$ (in the case $\abs{\bm{n}} = 0$ there is nothing to compute).

\end{proof}
\begin{rem} %We include the proof in the appendix. 
Note that the algorithm never breaks down for perfect systems since $\delta_{\bm{n},j,k}\ne 0$ for all $\bm{n}\in\N_{\bm{N}-\bm{e}_j-\bm{e}_k}^r$, by Lemma \ref{lem:normality vs recurrence} b).
%(follows from~\eqref{eq:CC3},~\eqref{eq:a}, and Lemma~\ref{lem:small normality lemma 2}). 
\end{rem}
\begin{rem}
    The algorithm in the way that is written here is not optimized since most of the coefficients are computed more than once (generically $r$ times, of course all leading to the same answer); if one cares about the computational efficiency then one would need to be more careful with the loops to avoid repetitions.
\end{rem}

\begin{cor}
Suppose $\bm{\mu}$ is perfect and $\mu_r\in \mathcal{L}_m$. Then the recurrence coefficients of $\widehat{\bm{\nu}}$ can be computed from the recurrence coefficients of $\bm{\nu}$ using the CC algorithm (see Theorem~\ref{thm:NNCC}) applied to ${\bm{\mu}}$.
\end{cor}
\begin{proof}
    This is immediate from Theorem~\ref{thm:zero coefficients finite systems} and Theorem~\ref{thm:NNCC}.
\end{proof}

    \begin{figure}[ht]
\begin{tikzpicture}

\foreach \x in {0,1,...,11}{
\foreach \y in {0,1,...,6}{
\node[draw,circle,inner sep=1pt,fill] at (\x,\y) {};}}
\draw [ultra thick,-latex,black] (0,0) -- (0,6) node[midway,above,sloped] {$\mu_2$};
\draw [ultra thick,-latex,black] (0,0) -- (13,0) node [above left] {$\mu_1$};
\draw [ultra thick,-latex,cyan] (0,6) -- (13,6) node [above left,black] {$\widehat{\mu}_1$};
\draw [ultra thick,cyan] (1,0) -- (0,1);
\draw [ultra thick,cyan] (2,0) -- (0,2);
\draw [ultra thick,cyan] (3,0) -- (0,3);
\draw [ultra thick,cyan] (4,0) -- (0,4);
\draw [ultra thick,cyan] (5,0) -- (0,5);
\draw [ultra thick,-latex,cyan] (0,0) -- (0.5,0.5) node{} [above left];
\draw [ultra thick,-latex,cyan] (0.5,0.5) -- (1,1) node{} [above left];
\draw [ultra thick,-latex,cyan] (1,1) -- (1.5,1.5) node{} [above left];
\draw [ultra thick,-latex,cyan] (1.5,1.5) -- (2,2) node{} [above left];
\draw [ultra thick,-latex,cyan] (2,2) -- (2.5,2.5) node{} [above left];
\draw [ultra thick,-latex,cyan] (2.5,2.5) -- (3,3) node{} [above left];
\node[scale=2] at (12,0) {$\ldots$};
\node[scale=2] at (12,1) {$\ldots$};
\node[scale=2] at (12,2) {$\ldots$};
\node[scale=2] at (12,3) {$\ldots$};
\node[scale=2] at (12,4) {$\ldots$};
\node[scale=2] at (12,5) {$\ldots$};
\node[scale=2] at (12,6) {$\ldots$};

\end{tikzpicture} 
\caption{The recurrence coefficients of $(\mu_1,\mu_2)$ can be computed via the CC algorithm (Theorem~\ref{thm:NNCC}). On the upper boundary we find the recurrence coefficients of ${\widehat{\mu}}_1$.}
\end{figure}
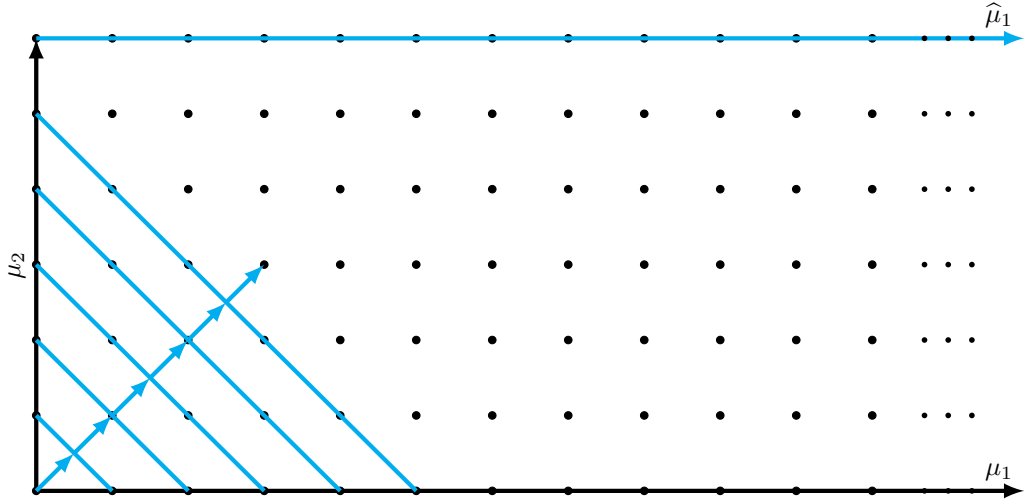

%\begin{rem} \label{rem:extension of uniqueness lemma} We include the proof in the appendix. , along with the proof of Theorem \ref{thm:compatibility conditions}. It can be easily seen that this proof can be extended from to any set of indices $\Pi$ such that if $\bm{n} \in \Pi$ then every $\bm{m}$ with $\bm{m} \leq \bm{n}$ satisfies $\bm{m} \in \Pi$. The recurrence coefficients defined for $\bm{n} \in \Pi$ will be $a_{\bm{n},j}$ whenever $\bm{n} + \bm{e}_i \in \Pi$ for some $i$, and $b_{\bm{n},j}$ whenever $\bm{n}+\bm{e}_j \in \Pi$. Again we require the compatibility conditions to hold whenever all the recurrence coefficients involved in the equations are defined, and we need $b_{\bm{n},j} \neq b_{\bm{n},k}$ whenever $\bm{n} + \bm{e}_j \in \Pi$ and $\bm{n} + \bm{e}_k \in \Pi$. The boundary conditions will be $a_{n\bm{e}_j,j}$ and $b_{n\bm{e}_j,j}$ when $(n+1)\bm{e}_j \in \Pi$, as well as $a_{\bm{n},j}$ when $n_j = 0$ and when $\bm{n} + \bm{e}_j \notin \Pi$. Also note that Theorem \ref{thm:perfectness thm} can be extended to any region $\Pi$ with the above order property. This also goes for Theorem \ref{thm:recurrence relation gives perfect system} and Theorem \ref{thm:perfect systems} below.
%\end{rem}

%%%%%%%%%%%%%%%%%%%%%%%%%%%%%%%%%%%%%%%%%%%%%%%%%%%%%%%%%%%%

%%%%%%%%%%%%%%%%%%%%%%%%%%%%%%%%%%%%%%%%%%%%%%%%%

\subsection{Two examples}\hfill\\

When $m=1$ in Theorem \ref{thm:ChrFull} we get the one-step Christoffel transform $\widehat{\bm{\nu}}$ of a multiple orthogonal system $\bm{\nu}$. This corresponds to taking $\mu_r$ to be the Dirac delta measure $\delta_{z_0}$ at $z_0$ for some $z_0\in\bbC$ (i.e., $\delta_{z_0}[x^n]:=z_0^n$) and $\Phi(x)=x-z_0$. Here, the determinantal condition is particularly interesting. The formula~\eqref{eq:ChrP1} has appeared earlier in~\cite{ADMVA}. In the case $r - 1 = 1$ this is a result Gautschi \cite{Gautschi}.

\begin{thm}\label{thm:ChrMOPRLII}
   Suppose $\bm{\nu}\in\mathcal{L}^{r-1}$ is perfect. Then $\widehat{\bm{\nu}}$ is perfect if and only if $z_0$ is not a root of any $P_{\bm{k}}(x)$, $\bm{k}\in\bbN^{r-1}_{\bm{K}}$, and then the type II polynomials with respect to $\bm{\nu}$ are given by
       \begin{equation}\label{eq:ChrP1}
       \widehat{P}_{\bm{k}}(x) = \frac{1}{x-z_0} \left({P}_{\bm{k}+\bm{e}_j}(x) - \frac{{P}_{\bm{k}+\bm{e}_j}(z_0)}{{P}_{\bm{k}}(z_0)}  {P}_{\bm{k}}(x)\right), \qquad \bm{k}\in\bbN^{r-1}_{\bm{K}-\bm{e}_j}, \qquad j = 1,\dots,r-1.
       \end{equation}
       %\item[ii)] For any $\bm{n}\in\bbN^{r-1}_{\bm{N}}$ and any $1\le j\le r-1$:
       %\begin{equation}
       %\widehat{A}^{(j)}_{\bm{n}}(x) = \frac{1}{x-z_0} \left({A}^{(j)}_{\bm{n}+\bm{e}_j}(x) - ?? {A}^{(j)}_{\bm{n}}(x)\right).
       %\end{equation}
\end{thm}
\begin{proof}
    $D_{\bm{k}} = P_{\bm{k}}(z_0)$.
\end{proof}

The compatibility conditions/CC algorithm now starts to resemble the original computation algorithm \eqref{eq:Gautschi algorithm}. If we write $\delta_{\bm{k},j}$ for $\delta_{(\bm{k},0),j,r}$, $j = 1,\dots,r-1$, the equations \eqref{eq:CC1.2}--\eqref{eq:CC3.2} for the system $\bm{\mu} = (\bm{\nu},\delta_{z_0})$ with $k = r$ then turn into 
\begin{align}
\label{eq:CC1.3}
    & \widehat{b}_{\bm{k},j} - \delta_{\bm{k} + \bm{e}_j,j} = b_{\bm{k}+\bm{e}_j,j} - \delta_{\bm{k},j},\\
    \label{eq:CC2.3}
    & \sum_{i=1}^{r-1} \widehat{a}_{\bm{k},i} - \delta_{\bm{k},j}\widehat{b}_{\bm{k},j} = \sum_{i=1}^{r-1} a_{\bm{k} + \bm{e}_j,i} - \delta_{\bm{k},j}b_{\bm{k},j}, \\
    \label{eq:CC3.3}
    & \delta_{\bm{k}-\bm{e}_j,j}\widehat{a}_{\bm{k},j} = \delta_{\bm{k},j}a_{\bm{k},j}, 
\end{align}
with the initial conditions $\delta_{\bm{0},j} = z_0 - b_{\bm{0},j}$ and $\widehat{a}_{\bm{k},j} = 0$ when $k_j = 0$. In the case $r = 2$ we end up exactly with Gautschi's equations \eqref{eq:Gautschi algorithm}. %from \cite{Gautschi book} %We can compute the recurrence coefficients n w$\set{\widehat{a}_{\bm{k},j},\widehat{b}_{\bm{k},j}}$. 
We can now generalize the computation algorithm to one-step Christoffel transforms for multiple orthogonal polynomials.
\begin{thm} Suppose $\bm{\nu}$ is perfect and $z_0$ is not a root of any $P_{\bm{k}}$, $\bm{k} \in \N_{\bm{K}}^{r-1}$. Given the NNRR coefficients of $\bm{\nu}$, the following algorithm produces all the NNRR coefficients of $\widehat{\bm{\nu}}:=(z-z_0)\bm\nu$:
\begin{align*}
    & \mbox{for all } 1\le j \le r-1:
    \\
    & \qquad \delta_{\bm{0},j}:= z_0 - b_{\bm{0},j};
    \\
    & \qquad \widehat{a}_{\bm{k},j}:=0  \mbox{ for all } \bm{k}\in\N^{r-1}_{\bm{K}} \mbox{ with } k_j=0; %\mbox{ and all } 1\le j \le r,
    \\
    & \qquad \widehat{b}_{\bm{0},j}:= b_{\bm{0},j} - \frac{a_{\bm{e_j},j}}{\delta_{\bm{0},j}};
    \\
    & \mbox{for all } d\in\bbN:
    \\
    & \qquad \mbox{for all } 1\le j\le r-1:
    \\
    & \qquad\qquad \mbox{for all } \bm{k}\in\N_{\bm{K}-\bm{e}_j}^{r-1} \mbox{ with } |\bm{k}|=d:
    \\
    & \qquad\qquad \qquad 
     \delta_{\bm{k},j}:= \widehat{b}_{\bm{k}-\bm{e}_j,j} - b_{\bm{k},j} + \delta_{\bm{k}-\bm{e}_j,j};
    \\
    & \qquad \mbox{for all } 1\le j\le r-1:
    \\
    & \qquad\qquad \mbox{for all } \bm{k}\in\N_{\bm{K}}^{r-1} \mbox{ with } k_j>0, |\bm{k}|=d:  
    \\
    & \qquad\qquad\qquad
    \widehat{a}_{\bm{k},j} =
    \begin{cases}  
    a_{\bm{k},j}\frac{\delta_{\bm{k},j}}{\delta_{\bm{k}-\bm{e}_j,j}} & \mbox{if } k_j < N_j, \\ 
    0 & \mbox{if } k_j = N_j;
    \end{cases}
     %& \qquad\qquad\qquad a_{\bm{k} ,j} =0 \mbox{ if } a_{\bm{k}-\bm{e}_k,j} = 0, \mbox{ otherwise}
    %\\
    %& \qquad\qquad\qquad a_{\bm{k} ,j} = a_{\bm{k}-\bm{e}_k,j}
    %\frac{\delta_{\bm{k}-\bm{e}_k,k,j}}{\delta_{\bm{k}-\bm{e}_k-\bm{e}_j,k,j}};
    %\frac{b_{\bm{k}-\bm{e}_k,j} - b_{\bm{k}-\bm{e}_k,k}}{b_{\bm{k} - \bm{e}_j-\bm{e}_k,j} - b_{\bm{k} - \bm{e}_j-\bm{e}_k,k}}
    \\
    & \qquad \mbox{for all } 1\le j\le r-1:
    \\
    & \qquad\qquad \mbox{ for all } \bm{k}\in\N_{\bm{K}-\bm{e}_j}^{r-1} \mbox{ with } |\bm{k}|=d: %N_j>k_j\ge 0;  
    \\
    & \qquad\qquad\qquad
    \widehat{b}_{\bm{k},j} = b_{\bm{k},j} + \frac{\sum_{i=1}^{r-1} \widehat{a}_{\bm{k},i} - \sum_{i=1}^{r-1} a_{\bm{k}+\bm{e}_j,i}}{\delta_{\bm{k},j}}.
    %b_{\bm{k}-\bm{e}_k,k} - b_{\bm{k}-\bm{e}_k,j}
\end{align*}
\end{thm}

Another example one often wants to consider is the two-step Christoffel transform of $\nu\in\mathcal{M}$ with $\Phi(z) = (x-z_0)(x-\bar{z}_0)$ where $z_0\in \C\setminus\bbR$.
%Let $\mu\in\mathcal{M}$ and let $P_n(x)$ be its orthogonal polynomials. 
%As discussed in Section~\ref{ss:Chr}, the usual one-step Christoffel transform of $P_n(x)$ can be studied by considering the MOPRL system $(\mu,\delta_{z_0})$. The system is perfect (in the sense of Definition~\ref{def:perf}) if and only if $z_0$ is away from the zeros of orthogonal polynomials of $\mu$.
The most obvious choice of the multiple orthogonality system $(\nu,\tfrac12 \delta_{z_0} +\tfrac12\delta_{\bar{z}_0})$ might not be perfect however. For example,  
if $\nu$ is symmetric with respect to $\operatorname{Re} z_0$ and $n$ is odd, then $P_n(x)$, the degree $n$ orthogonal polynomial of $\nu$, satisfies
%if $\operatorname{Re} z_0$ happens to be a zero of some $P_n(x)$ then $P_n(x)$ is a degree $n$ polynomial that satisfies 
all the orthogonality conditions for $(\nu,\tfrac12 \delta_{z_0} +\tfrac12\delta_{\bar{z}_0})$ at the location $(n,1)$, which shows that $(n,1)$ is not normal in that case. Instead, let us consider the multiple orthogonality system system $(\nu, \omega)$, where $\omega\in\mathcal{L}_2$ given by
    \begin{equation}
        \omega = w_0 \delta_{z_0} +(1-w_0)\delta_{\bar{z}_0}
    \end{equation}
    with $w_0\in\bbR\setminus\{0,\tfrac12,1\}$. Indices $(n,0)$ and $(n,2)$ are  normal  for all $n$  since $\nu$ and its transform are both in $\mathcal{M}$. Suppose that some index $(n,1)$ is not normal. Then $P_n(x)$ must be one of the type II polynomials at index $(n,1)$, implying
    \begin{equation}
        0=\omega[P_n(x)] = w_0 P_n(z_0) + (1-w_0) \overline{P_n(z_0)}.
    \end{equation}
    The latter equality implies $\operatorname{Re} P_n(z_0) = 0$ and $(2w_0-1)\operatorname{Im} P_n(z_0) = 0$. Hence $P_n(z_0)=P_n(\bar{z}_0)=0$, but then $P_n$ solves all orthogonality relations with respect to the index $(n,2)$, which contradicts the normality of $(n,2)$. This proves that $(n,1)$ is normal and therefore $(\nu, \omega)$ is perfect for any $w_0\in\bbR\setminus\{0,\tfrac12,1\}$. In order to run the CC algorithm with $(\nu,\omega)$ for computing the Jacobi coefficients of $\widehat{\nu}$, one needs to know the Jacobi coefficients~\eqref{eq:jacobi} of $\omega$. Elementary calculations show that these are
    \begin{align}
        b_0 & = x_0 - iy_0 (1-2w_0),
        \\
        b_1 & = x_0 + iy_0 (1-2w_0), 
        \\ 
        a_1 & = -4 w_0(1-w_0) y_0^2,
    \end{align}
    where $x_0 = \operatorname{Re} z_0$, $y_0 = \operatorname{Im} z_0$. Note that all the arguments used here work just as well in the more general case $\bm{\nu} \in \mathcal{M}^{r-1}$ with $r-1>2$ except that the perfectness of $\nu$ and of $\widehat{\nu}$ is no longer given for free (but this easily holds if we assume that
    $\bm{\nu}$ is an Angelesco or an AT system, for example). 
    %\rkno{It would be interesting to check if the CC algorithm with such $\omega$ gives any calculational improvement over the doubly repeated standard Christoffel/Galant algorithm \cite{Gal} at $z_0$ and then at $\bar{z}_0$. }

\subsection{An example on the step-line}
\hfill\\

In this section we work with a perfect system of two functionals $\bm{\nu} = (\mu_1,\mu_2)$ and write $p_n$ for the polynomials defined by $p_{2n} = P_{n,n}$ and $p_{2n+1} = P_{n+1,n}$ for $n \in \N$. $\set{p_n}_{n = 0}^{\infty}$ are called the multiple orthogonal polynomials on the step-line. In the previous sections we have worked with recurrence coefficients from the nearest neighbour recurrence relation, but another recurrence relation restricted to the step-line is also commonly used. 

We assume indices on the step-line $\set{(n,n)}_{n \in \N}\cup\set{(n+1,n)}_{n \in \N}$ are normal. The step-line recurrence relation is then given by
\begin{equation}\label{eq:stepline recurrence}
    xp_n(x) = p_{n+1}(x) + c_n p_n(x) + b_n p_{n-1}(x) + a_n p_{n-2}(x).
\end{equation}
If we put $a_0 = a_1 = 0$ and $b_1 = 0$ then the recurrence relation holds for any $n \in \N$ (with any choice of $p_{-1}$ and $p_{-2}$). To see that \eqref{eq:stepline recurrence} holds, simply choose $b_n$ and $c_n$ such that the polynomial $xp_n(x) - p_{n+1}(x) - c_n p_n(x) - b_n p_{n-1}(x)$ is of minimal degree, and then verify that it satisfies all the orthogonality conditions for the polynomial $p_{n-2}$. 

We put $\Phi(x) = x - z_0$ and assume that $p_n(z_0) \neq 0$ for all $n \in \N$. A slight modification of Theorem \ref{thm:ChrMOPRLII} %and Theorem \ref{thm:ChrFull} 
shows that $\widehat{\nu} = \Phi\nu$ has all indices normal on the step-line. The polynomials $\widehat{p}_n$ on the step-line then satisfy the recurrence relation
\begin{equation}\label{eq:stepline recurrence ct}
    x\widehat{p}_n(x) = \widehat{p}_{n+1}(x) + \widehat{c}_n \widehat{p}_n(x) + \widehat{b}_n \widehat{p}_{n-1}(x) + \widehat{a}_n \widehat{p}_{n-2}(x).
\end{equation}

Let us show how to compute the recurrence coefficients $\widehat{a}_n$, $\widehat{b}_n$, and $\widehat{c}_n$, assuming we are given $a_n$, $b_n$, and $c_n$ for each $n \in \N$. Note that \cite{Computing NNR} describes how to compute the step-line recurrence coefficients from the nearest neighbour recurrence coefficients (see~\cite[Sect 2.2]{DFK} for the same problem for an arbitrary increasing path of indices). Here, however, we would like to go directly from the step-line coefficients of $\bm{\nu}$ to the step-line coefficients of $\widehat{\bm{\nu}}$.

For the system $\bm{\mu} = (\mu_1,\mu_2,\delta_{z_0})$ we have $P_{n,m,0}(x) = P_{n,m}(x)$ and $P_{n,m,1}(x) = (x-z_0)\widehat{P}_{n,m}(x)$. When $(n,m)$ is on the step-line we have the recurrence relations \eqref{eq:stepline recurrence} and $\eqref{eq:stepline recurrence ct}$. We also have the recurrence relation \eqref{eq:nnr cor}, which turns into 
\begin{equation}\label{eq:nnr cor stepline}
    p_{n+1}(x) - (x-z_0)\widehat{p}_n(x) = \delta_{n}p_{n}(x).
\end{equation}

We proceed by computing $(x-z_0)\widehat{p}_{n+1}(x)$ using \eqref{eq:stepline recurrence}, \eqref{eq:stepline recurrence ct}, and \eqref{eq:nnr cor stepline}, in two different ways (similarly to the proof of Theorem \ref{thm:compatibility conditions} in \cite{NNR}). First, we have 
\m{(x-z_0)\widehat{p}_{n+1} = p_{n+2} - \delta_{n+1} p_{n+1} = xp_{n+1} - (c_{n+1}+\delta_{n+1})p_{n+1} - b_{n+1}p_n - a_{n+1}p_{n-1}.}
On the other hand, we have
\m{(x-z_0)\widehat{p}_{n+1} & = x(x-z_0)\widehat{p}_n - \widehat{c}_n(x-z_0)\widehat{p}_n - \widehat{b}_n(x-z_0)\widehat{p}_{n-1} - \widehat{a}_n(x-z_0)\widehat{p}_{n-2} \\ & = x(p_{n+1} - \delta_n p_n) - \widehat{c}_n(p_{n+1} - \delta_{n}p_n) - \widehat{b}_n(p_n - \delta_{n-1}p_{n-1}) - \widehat{a}_n(p_{n-1} - \delta_{n-2}p_{n-2}) \\ & = xp_{n+1} -\delta_n(p_{n+1} + c_np_n + b_np_{n-1} + a_np_{n-2}) - \widehat{c}_np_{n+1} - (\widehat{b}_n - \delta_n\widehat{c}_n)p_n \\ & \qquad \qquad \qquad \qquad \qquad \qquad \qquad \qquad \qquad \qquad - (\widehat{a}_n - \delta_{n-1}\widehat{b}_n)p_{n-1} + \delta_{n-2}\widehat{a}_np_{n-2} \\ & = xp_{n+1} - (\delta_n + \widehat{c}_{n})p_{n+1} - (\delta_nc_n + \widehat{b}_n - \delta_n\widehat{c}_n)p_n - (\delta_nb_n + \widehat{a}_n - \delta_{n-1}\widehat{b}_n)p_{n-1} \\ & \qquad \qquad \qquad \qquad \qquad \qquad \qquad \qquad \qquad \qquad - (\delta_na_n - \delta_{n-2}\widehat{a}_n)p_{n-2}.}
By comparing the two results, using liner independence, we get the following result.

\begin{thm}\label{thm:compatibility conditions stepline}
    If all indices on the step-line are normal for $\bm{\nu}$ and $\widehat{\bm{\nu}}$, then 
    \begin{align}
\label{eq:CC1 stepline}
    & \widehat{c}_{n} - \delta_{n+1} = c_{n+1} - \delta_n,\\
    \label{eq:CC2 stepline}
    & \widehat{b}_n - \delta_n\widehat{c}_n = b_{n+1} - \delta_nc_n, \\
    \label{eq:CC3 stepline}
    & \widehat{a}_n - \delta_{n-1}\widehat{b}_n = a_{n+1} - \delta_nb_n, \\
    \label{eq:CC4 stepline}
    & \delta_{n-2}\widehat{a}_n = \delta_na_n.
\end{align}
If $n = 0$ we only get the first two equations, and if $n = 1$ we only get the first three equations.
\end{thm}

We now get the following computation algorithm. 

\begin{thm}\label{thm:NNCC stepline}
Suppose all indices on the step-line are normal for $\bm{\nu} = (\mu_1,\mu_2)$ and $\widehat{\bm{\nu}} = (x-z_0)\bm{\nu}$. Given the step-line coefficients of $\bm{\nu}$, the following algorithm produces all step-line coefficients of $\widehat{\bm{\nu}}$: 
\begin{align*}
    \begin{split}
        & \delta_{0}:= z_0 - c_{0};
    \\
    & \widehat{a}_{0}=0;
    \\
    & \widehat{b}_{0}=0;
    \\
    & \widehat{c}_{0}= c_{0} - \frac{b_1}{\delta_{0}};
    \\
    & \delta_1 = \widehat{c}_0 - c_1 + \delta_0;
    \\
    & \widehat{a}_{1}=0;
    \\
    & \widehat{b}_{1} = \frac{\delta_1 b_1 - a_2}{\delta_0};
    \\
    & \widehat{c}_1 = c_1 + \frac{\widehat{b}_1-b_2}{\delta_1};
    \\
    & \mbox{for all } n \geq 2:
    \\
    & \qquad
     \delta_{n}:= \widehat{c}_{n-1} - c_{n} + \delta_{n-1};
    \\
    & \qquad
         \widehat{a}_{n} = a_{n}
    \frac{\delta_{n}}{\delta_{n-2}};
    \\
    & \qquad
    \widehat{b}_{n} = \frac{\delta_nb_n + \widehat{a}_n - a_{n+1}}{\delta_{n-1}};
    \\
    & \qquad 
    \widehat{c}_n = c_n + \frac{\widehat{b}_n - b_{n+1}}{\delta_n}.
    \end{split}
\end{align*}
\end{thm}

\begin{proof}
    for $n = 0$, \eqref{eq:nnr cor stepline} turns into $(x-c_0) - (x - z_0) = \delta_0$, which allows us to compute $\delta_0$. The rest clearly follows from Theorem \ref{thm:compatibility conditions stepline}. We can divide by $\delta_n$ since $\delta_n \neq 0$ by Lemma \ref{lem:normality vs recurrence} b).
\end{proof}

\subsection{Christoffel Transforms of Type I polynomials}\label{ss:CT type 1}
\hfill\\

We write $\bm{A}_{\bm{k}} = (A_{\bm{k}}^{(1)},\dots,A_{\bm{k}}^{(r-1)})$ and $\widehat{\bm{A}}_{\bm{k}} = (\widehat{A}_{\bm{k}}^{(1)},\dots,\widehat{A}_{\bm{k}}^{(r-1)})$ for the type I polynomials with respect to $\bm{\nu}$ and $\widehat{\bm{\nu}}$, respectively. We get similar results to the previous two sections here. However, in this case the matrix is larger and the sequence of indices requires the extra constraints $\bm{s}_j \leq m\bm{1}$ where $m\bm{1} = (m,\dots,m) \in \N_{}^{r-1}$. 
{For simplicity we assume that $\bm{\nu} \in \mathcal{L}_{\infty}^{r-1}$ and all zeros of $\Phi$ are simple, ignoring the details of the special cases discussed in Remark \ref{rem:modification}-\ref{rem:modification 2}.}
%and the NNRR coefficients $\{a_{\bm{k},j},b_{\bm{k},j}\}$, $\{\widehat{a}_{\bm{k},j},\widehat{b}_{\bm{k},j}\}$. 

\begin{thm}\label{thm:det formula type I}

Let $\bm{\nu}$ be a perfect system and $\bm{k} \in \N_{\bm{K}}^{r-1}$. Let $\set{\bm{k}+\bm{s}_j}_{j = 0}^{(r-1)m}$ be a sequence of $\N_{}^{r-1}$-indices where $\abs{\bm{s}_j} = j$ and $\bm{s}_j \in \N_{m\bm{1}}^{r-1}$ for each $j = 0,1,\dots,(r-1)m$. %Write $\bm{k}+m$ for the index $(n_1+m,\dots,n_{r-1}+m)$ and consider  $\abs{\bm{s}_j} = j$ for each $0 \leq j \leq rm$. 
Consider the determinant
%$D_{\bm{k}} = \det{(\bm{A}_{\bm{k}+\bm{s}_{rm-j+1}}(z_k))}_{j = 1,\dots,rm-j,k = 1,\dots,m}$
\nm{eq:Dn type 1}{D_{\bm{k}} = \det{
\begin{pmatrix}
\bm{A}_{\bm{k}+\bm{s}_1}(z_1) & \bm{A}_{\bm{k} + \bm{s}_2}(z_1) & \cdots & \bm{A}_{\bm{k} + \bm{s}_{(r-1)m}}(z_1) \\
\bm{A}_{\bm{k}+\bm{s}_1}(z_2) & \bm{A}_{\bm{k} + \bm{s}_2}(z_2) & \cdots & \bm{A}_{\bm{k} + \bm{s}_{(r-1)m}}(z_2) \\
\vdots & \vdots & \ddots & \vdots \\
\bm{A}_{\bm{k}+\bm{s}_1}(z_m) & \bm{A}_{\bm{k} + \bm{s}_2}(z_m) & \cdots & \bm{A}_{\bm{k} + \bm{s}_{(r-1)m}}(z_m) \end{pmatrix},
}}
where $\bm{A}_{\bm{k}+\bm{s}_j}(z_k)$ denotes the column vector with elements $A_{\bm{k}+\bm{s}_j}^{(1)}(z_k),\dots,A_{\bm{k}+\bm{s}_j}^{(r-1)}(z_k)$. 

If $(\bm{k},m)$ is normal then $D_{\bm{k}} \neq 0$ and the following determinantal formula holds
\nm{eq:type 1 determinantal formula}{
A_{{(}\bm{k},m{)}}^{(j)}(x) = \widehat{A}_{\bm{k}}^{(j)}(x) = {\Phi(x)}^{-1}D_{\bm{k}}^{-1}
\det{
\begin{pmatrix}
A_{\bm{k}}^{(j)}(x) & A_{\bm{k} + \bm{s}_1}^{(j)}(x) & \cdots & A_{\bm{k} + \bm{s}_{(r-1)m}}^{(j)}(x) \\
\bm{A}_{\bm{k}}(z_1) & \bm{A}_{\bm{k} + \bm{s}_1}(z_1) & \cdots & \bm{A}_{\bm{k} + \bm{s}_{(r-1)m}}(z_1) \\
\vdots & \vdots & \ddots & \vdots \\
\bm{A}_{\bm{k}}(z_m) & \bm{A}_{\bm{k} + \bm{s}_1}(z_m) & \cdots & \bm{A}_{\bm{k} + \bm{s}_{(r-1)m}}(z_m) 
\end{pmatrix}
}, \\ 1 \leq j \leq r-1.
}
\end{thm}

\begin{proof}
In the proof we use the notation $\jap{\bm{A}(x),P(x)}$, for polynomials $P(x)$ and vectors of polynomials $(A^{(1)}(x),\dots,A^{(r-1)}(x))$, to represent $\sum_{j = 1}^{r-1}\jap{A^{(j)}(x),P(x)}_j$. Note that this is clearly bilinear. Then the orthogonality relations for $\bm{A}_{\bm{k}}(x)$ can be rephrased as
\begin{equation}\label{eq:form}
\jap{\bm{A}_{\bm{k}}(x),x^p} = 
0, \qquad p = 0,\dots,\abs{\bm{k}}-2.
\end{equation}

We first prove that $D_{\bm{k}} \neq 0$. $D_{\bm{k}} = 0$ would imply that the columns are linearly dependent. In particular this means that 
\m{\sum_{j = 1}^{(r-1)m}c_j\bm{A}_{\bm{k}+\bm{s}_j}(x) = \bm{0}, \qquad x = z_1,\dots,z_r,}
for some $(c_1,\dots,c_{(r-1)m}) \neq (0,\dots,0)$. On the other hand, we note that the sum cannot vanish for every $x$, since in that case the smallest $l$ such that $c_l \neq 0$ would satisfy
\m{\jap{\sum_{j = 1}^{(r-1)m}c_j\bm{A}_{\bm{k}+\bm{s}_j}(x),x^{{\abs{\bm{k}}}-1+l}}=c_l\jap{\bm{A}_{\bm{k}+\bm{s}_l}(x),x^{{\abs{\bm{k}}}-1+l}}=0.}
Normality of $\bm{k}+\bm{s}_l$ would then imply $c_l = 0$. Hence we have a non-zero vector $\sum_{j = 1}^{(r-1)m}c_j\bm{A}_{\bm{k}+\bm{s}_j}$ which is also divisible by $\Phi$. Write $\bm{A} = (A^{(1)},\dots,A^{(r-1)})$ for $\Phi(x)^{-1}\sum_{j = 1}^{(r-1)m}c_j\bm{A}_{\bm{k}+\bm{s}_j}$. Since $\bm{s}_j \in \N_{m\bm{1}}$ for each $j = 0,1,\dots,(r-1)m$ we have $\deg{A^{(i)}} \leq k_i - 1$. This vector also satisfies 
\m{\jap{\Phi(x)\bm{A}(x),x^{p}} = 0, \qquad p = 0,\dots,\abs{\bm{k}}-1,}
{which implies that $\bm{k}$ is not $\widehat{\bm{\nu}}$-normal.} This contradicts the normality of $(\bm{k},m)$ by Theorem \ref{thm:ct moprl}. Hence we must have $D_{\bm{k}} \neq 0$.

Next, we note that the determinant in \eqref{eq:type 1 determinantal formula} vanishes at $z_1\dots,z_r$, so it is divisible by $\Phi(x)$ (for each $j = 0,\dots,r-1$). {Denote the rightmost expression of~\eqref{eq:type 1 determinantal formula} by $B_{\bm{k}}^{(j)}$ and $\bm{B}_{\bm{k}} = (B^{(1)}_{\bm{k}},\dots,B^{(r-1)}_{\bm{k}})$}. Clearly {$\bm{B}_{\bm{k}}$} is {then} orthogonal  to $x^p$ for $p = 0,\dots,\abs{\bm{k}}-2$ {with respect to $\widehat{\bm{\nu}}$ in the sense of~\eqref{eq:form}}. {Now note that}
\begin{equation}
    {\jap{\bm{B}_{\bm{k}}(x),x^{\abs{\bm{k}}-1}}_{\widehat{\bm{\nu}}} = 
    \jap{D_{\bm{k}}^{-1} D_{\bm{k}}\bm{A}_{\bm{k}}(x),x^{\abs{\bm{k}}-1}}_{{\bm{\nu}}} = {1}.}
\end{equation}
This shows that $\widehat{\bm{A}}_{\bm{k}} = \bm{B}_{\bm{k}}$. Finally,
\m{\jap{\bm{A}_{(\bm{k},m)},x^{p+m}}_{\bm{\mu}} = \jap{\bm{A}_{(\bm{k},m)},x^{p}\Phi(x)}_{\bm{\mu}} = \sum_{j = 1}^{r-1}\jap{\Phi(x)A_{(\bm{k},m)}^{(j)},x^p}_j = 
\begin{cases}
    0, \qquad p = 0,\dots,\abs{\bm{k}}-2, \\
    {1, \qquad p=|\bm{k}|-1,}
\end{cases}
}
by Lemma \ref{lem:finite rank}. Also $\deg{A_{(\bm{k},m)}} \leq k_j-1$, $j = 1,\dots,r-1$. 
{Since $\bm{k}$ is normal for $\widehat{\bm{\nu}}$, we conclude with  $\widehat{A}_{\bm{k}}^{(j)} = A_{(\bm{k},m)}^{(j)}$}. 
\end{proof}

\subsection{Determinantal Formula for Type I Polynomials: one-step case}
\hfill\\

For the case $m = 1$, the determinantal formula in Theorem \ref{thm:det formula type I} is not as simple as in Theorem \ref{thm:ChrMOPRLII}. Rather, we get a formula in terms of $\bm{A}_{\bm{k}},\bm{A}_{\bm{k}+\bm{s}_1},\dots,\bm{A}_{\bm{k}+\bm{s}_{r-1}}$. However, we can get new formulas by deforming the sequence $\set{\bm{k}+\bm{s}_j}_{j = 0}^{r-1}$ using \eqref{eq:nnr cor type 1}. Perhaps the nicest choice is writing the columns in terms of the nearest neighbours, which the following theorem shows is possible, even with very relaxed normality assumptions. %Rather, we get the formula
%\m{\widehat{A}_{\bm{k}}^{(j)}(x) = {\Phi(x)}^{-1}D_{\bm{k}}^{-1}\det{
%\begin{pmatrix}
%A_{\bm{k}}^{(j)}(x) & A_{\bm{k} + \bm{s}_1}^{(j)}(x) & \cdots & A_{\bm{k} + \bm{s}_{r-1}}^{(j)}(x) \\
%A_{\bm{k}}^{(1)}(z_0) & A_{\bm{k} + \bm{s}_1}^{(1)}(z_0) & \cdots & A_{\bm{k} + \bm{s}_{r-1}}^{(1)}(z_0) \\
%\vdots & \vdots & \ddots & \vdots \\
%A_{\bm{k}}^{(r-1)}(z_0) & A_{\bm{k} + \bm{s}_1}^{(r-1)}(z_0) & \cdots & A_{\bm{k} + \bm{s}_{r-1}}^{(r-1)}(z_0) \end{pmatrix}}, \qquad 1 \leq j \leq r-1,
%}
%where $\Phi(x) = x - z_0$. Note that the sequence $\set{\bm{k}+\bm{s}_j}_{j = 0}^{r-1}$ with the properties specified in Theorem \ref{thm:det formula type I} is not necessarily a path, i.e., when $\bm{s}_{k+1} = \bm{s}_k + \bm{e}_{j_k}$ where $1 \leq j_k \leq r-1$, $k = 0,\dots,r-2$. When we do have a path, we could use \eqref{eq:nnr cor type 1} to, e.g., rewrite the last column $A_{\bm{k}+\bm{s}_{r-1}}^{(j)}$ a linear combination of $A_{\bm{k}+\bm{s}_{r-2}}^{(j)}$ and $A_{\bm{k}+\bm{s}_{r-1}-\bm{e}_k}^{(j)}$ for some $k$ (with $j_{r-1} \neq k$). The term with $A_{\bm{k}+\bm{s}_{r-2}}^{(j)}$ vanishes since it appears in one of the other columns, and we can do the same for $D_{\bm{k}}$, so we can simply replace $A_{\bm{k}+\bm{s}_{r-1}}^{(j)}$ with $A_{\bm{k}+\bm{s}_{r-2}-\bm{e}_k}^{(j)}$ in the determinantal formula, and then keep deforming the path with \eqref{eq:nnr cor type 1}. Perhaps the nicest choice is writing the columns in terms of the nearest neighbours, which the following theorem shows is possible. 

\begin{thm}\label{thm:nice type I det formula}
    Suppose $\bm{k}$ is normal for both $\bm{\nu}$ and $\widehat{\bm{\nu}}=(x-z_0)\bm{\nu}$. Consider the matrix
    \nm{eq:type 1 det formula alternative normalization}{D_{\bm{k}} = \det\begin{pmatrix}
A_{\bm{k}+\bm{e}_1}^{(1)}(z_0) & A_{\bm{k} + \bm{e}_2}^{(1)}(z_0) & \cdots & A_{\bm{k} + \bm{e}_{r-1}}^{(1)}(z_0) \\
A_{\bm{k}+\bm{e}_1}^{(2)}(z_0) & A_{\bm{k} + \bm{e}_2}^{(2)}(z_0) & \cdots & A_{\bm{k} + \bm{e}_{r-1}}^{(2)}(z_0) \\
\vdots & \vdots & \ddots & \vdots \\
A_{\bm{k}+\bm{e}_1}^{(r-1)}(z_0) & A_{\bm{k} + \bm{e}_2}^{(r-1)}(z_0) & \cdots & A_{\bm{k} + \bm{e}_{r-1}}^{(r-1)}(z_0) \end{pmatrix}.}
Then ${D_{\bm{k}}} \neq 0$ and the following determinantal formula holds
\nm{eq:type 1 det formula alternative version}{\widehat{A}_{\bm{k}}^{(j)}(x) = (x-z_0)^{-1}D_{\bm{k}}^{-1}\det{\begin{pmatrix}
A_{\bm{k}}^{(j)}(x) & A_{\bm{k}+\bm{e}_1}^{(j)}(x) & \cdots & A_{\bm{k}+\bm{e}_{r-1}}^{(j)}(x) \\
A_{\bm{k}}^{(1)}(z_0) & A_{\bm{k}+\bm{e}_1}^{(1)}(z_0) & \cdots & A_{\bm{k} + \bm{e}_{r-1}}^{(1)}(z_0) \\
\vdots & \vdots & \ddots & \vdots \\
A_{\bm{k}}^{(r-1)}(z_0) & A_{\bm{k}+\bm{e}_1}^{(r-1)}(z_0) & \cdots & A_{\bm{k} + \bm{e}_{r-1}}^{(r-1)}(z_0) \end{pmatrix}}, \qquad 1 \leq j \leq r-1.}
\end{thm}

\begin{proof}
    If $D_{\bm{k}} = 0$ then for some non-trivial linear combination,
    \m{\sum_{j = 1}^{r-1}c_k\bm{A}_{\bm{k}+\bm{e}_j}(z_0) = 0, \qquad j = 1,\dots,r-1.}
    Then define the vector $\bm{A} = (A^{(1)},\dots,A^{(r-1)})$ by 
    \m{\bm{A}(x) = (x-z_0)^{-1}\sum_{j = 1}^{r-1}c_j\bm{A}_{\bm{k}+\bm{e}_j}(x).}
    Clearly $\deg A^{(i)} \leq k_i - 1$, $i = 1,\dots,r-1$, and $\bm{A}$ satisfies the orthogonality conditions 
    \m{\sum_{i = 1}^{r-1}\jap{(x-z_0)A^{(i)}(x),x^p}_i = 0, \qquad p = 0,\dots,\abs{\bm{k}}-1.}
    Since $\bm{A} \neq \bm{0}$ by Lemma \ref{lem:nnr lem 2} we get a contradiction with the normality of the index $\bm{k}$ for the system $\widehat{\bm{\nu}}$. We conclude that $D_{\bm{k}} \neq 0$, and \eqref{eq:type 1 det formula alternative version} now follows from similar arguments as in the proof of Theorem \ref{thm:det formula type I}.
\end{proof}

%\begin{rem}
%    Note the relaxed normality assumption in Theorem \ref{thm:nice type I det formula} in comparison with previous determinantal formulas. Since we do not assume normality of $\bm{n} + \bm{e}_1, \dots, \bm{n} + \bm{e}_{r-1}$, there may be several choices for $\bm{A}_{\bm{n} + \bm{e}_1},\dots,\bm{A}_{\bm{n} + \bm{e}_{r-1}}$, and then the determinantal formula will hold for each choice. 
%\end{rem}

\begin{rem}\label{rem:alternative sequence}
    We can get a generalized version of Theorem \ref{thm:nice type I det formula} for arbitrary $m$, by considering a determinantal formula in terms of $A_{\bm{k}}^{(j)}$ and $A_{\bm{k}+l\bm{e}_i}^{(j)}$ for each $l = 1,\dots,m$ and $i = 1,\dots,r-1$ %, or more generally, $r$ different non-intersecting paths starting at $\bm{n}$ 
    (although here we will need
    to assume more indices than just $\bm{k}$ to be normal). This can also be seen through repeated elementary row operations in the determinant in \eqref{eq:type 1 determinantal formula} combined with the relations in \eqref{eq:nnr cor type 1}.
    %\rkno{Repeated applications of \eqref{eq:nnr cor} in \eqref{eq:ct det} cannot lead to the sequence of all straight paths starting at $\bm{k}$. Rather, we can get the sequence of all straight paths going downwards from $\bm{k} + \bm{s}_m$. The reader may also verify using Lemma \ref{lem:small normality lemma 2}
    %Roots of $\Phi$ with higher multiplicity are dealt with by adding rows of derivatives of each type I polynomial at the given roots. }
\end{rem}

We can also relate the one-step Christoffel transform to the kernel polynomials \eqref{eq:cd kernel moprl}, similarly to \eqref{eq:cd one measure}, by the following result.

\begin{thm}
Suppose $\bm{k}$ is normal for both $\bm{\nu}$ and $\widehat{\bm{\nu}}=(x-z_0)\bm{\nu}$. Then 
%     With $\Phi(x) = x - z_0$, and assuming $\bm{n}$ is normal for $\widehat{\bm{\mu}}$, we have 
$P_{\bm{k}}(z_0) \neq 0$ and
    \nm{eq:cd kernel christoffel moprl}{\bm{K}_{\bm{k}}(z_0,x) = -P_{\bm{k}}(z_0)\widehat{\bm{A}}_{\bm{k}}(x).}
\end{thm}
\begin{proof}
    If $P_{\bm{k}}(z_0) = 0$ then $P_{\bm{k}}$ satisfies all the orthogonality conditions for the index $(\bm{k},1)$ with respect to the system $(\bm{\nu},\delta_{z_0})$, which is impossible by Theorem \ref{thm:ct moprl}, so $P_{\bm{k}}(z_0) \neq 0$. By \eqref{eq:cd kernel moprl} we have
    \m{\sum_{j = 1}^r\jap{(x-z_0)K_{\bm{k}}^{(j)}(z_0,x),x^p}_j = -P_{\bm{k}}(z_0)\sum_{j = 1}^r\jap{A_{\bm{k}}^{(j)}(x),x^p}_j, \qquad p = 0,\dots,\abs{\bm{k}}-1,}
    so $\bm{K}_{\bm{k}}(z_0,x)$ satisfies the same orthogonality relations as $-P_{\bm{k}}(z_0)\widehat{\bm{A}}_{\bm{k}}(x)$. Clearly we also have $\deg{K_{\bm{k}}^{(j)}(z_0,x)} \leq k_j - 1$, so we get \eqref{eq:cd kernel christoffel moprl}. 
\end{proof}

\begin{rem}
    For the type II polynomials we see that \eqref{eq:ChrP1} does not resemble $\bm{K}_{\bm{k}}(x,z_0)$. We can however get a similar result to \eqref{eq:cd kernel christoffel moprl} by considering linear combinations of $\widehat{P}_{\bm{k}-\bm{e}_1},\dots,\widehat{P}_{\bm{k}-\bm{e}_r}$. A quick check using \eqref{eq:ChrP1} and $\eqref{eq:cd formula moprl}$ shows that
    \nm{eq:cd kernel type II}{P_{\bm{k}}(z_0)K_{\bm{k}}^{(i)}(z_0,x) = \sum_{j = 1}^r a_{\bm{k},j}P_{\bm{k}-\bm{e}_j}(z_0)A_{\bm{k}+\bm{e}_j}^{(i)}(z_0)\widehat{P}_{\bm{k}-\bm{e}_j}(x).}
\end{rem}

\subsection{Repeated Christoffel transform}\label{ss:repeated}
\hfill\\

In the MOPRL construction, even for the discrete case $\mu_l\in\mathcal{L}_{N_l}$, we always have $a_{j\bm{e}_l,l} \ne 0$ for $1\le j \le N_l-1$, and we terminate $a$'s as soon as we reach $a_{N_l\bm{e}_l,l} = 0$. 
If one thinks in terms of the Jacobi matrices rather than in terms of the orthogonality measures then it is natural to allow some of the $a_{j\bm{e}_l,l}$'s to be zero. Let us do so for the last measure $\mu_r$ only. 

To this end, for each $l=1,\ldots,r$, let $J_l$ be an $N_l\times N_l$  Jacobi matrix (with $N_l$ finite or infinite) with Jacobi coefficients $\{a_{j\bm{e}_l,l}\}_{j=1}^{N_l-1}$ and $\{b_{j\bm{e}_l,l}\}_{j=0}^{N_l-1}$. We assume that $J_1,\ldots,J_{r-1}$ are proper, that is, $a_{j\bm{e}_l,l}\ne 0$ for each $1\le j \le N_l-1$  and $1\le l\le r-1$. Suppose $J_{r}$ is the direct sum $J_{r}=\bigoplus_{j=1}^\infty  J^{(j)}_{r}$ (which is improper), where $J^{(j)}_{r}$ is a proper Jacobi matrix of size $m_j\times m_j$, where $m_j < \infty$. Denote $\det(z-J^{(j)}_{r}) = \Phi_j(z)$. We then have $\deg\Phi_j = m_j$.

Now observe that we can still run the CC algorithm without changes. We then treat $N_{r}$ as $\infty$, and {\it not} as $m_1$, as we would normally. The algorithm will be able to compute all the recurrence coefficients $\{a_{\bm{n},j}\}_{\bm{n}\in\bbN^r_{(\bm{K},\infty)}}$  and $\{b_{\bm{n},j}\}_{\bm{n}\in\bbN^r_{(\bm{K},\infty)-\bm{e}_j}}$ %(here $\bm{K}=(N_1,\ldots,N_{r-1})$ as in the previous sections)
for each $1\le j \le r$, as long as all the $\delta_{\bm{n},j,i}$ coefficients are nonzero. All of these coefficients satisfy the compatibility equations~\eqref{eq:CC1}--\eqref{eq:CC3} on the extended lattice $\bbN^r_{(\bm{K},\infty)}$ (here $\bm{K} = (N_1,\ldots,N_{r-1})$ as usual), and therefore one can uniquely define the polynomials $P_{\bm{n}}$ and $\bm{A}_{\bm{n}}$ using~\eqref{eq:one nnr} and~\eqref{eq:type 1 recurrence relation} for all $\bm{n}\in\bbN^r_{(\bm{K},\infty)}$.

By the results of the previous section it is clear that if $\bm{n}\in\bbN^r_{(\bm{K},\infty)}$ has the form $\bm{n} = (\bm{k},m_1+\ldots+m_s)$ with $\bm{k}\in\bbN^{r-1}_{\bm{K}}$ for some $s\in\bbN$, then:
\begin{itemize}
    \item[\it{i)}] $a_{\bm{n},r} = 0$;
    \item[\it{ii)}] $a_{\bm{n},j}$ and $b_{\bm{n},j}$ for $1\le j\le r-1$ are the NNRR coefficients of $(\Phi_1\ldots\Phi_s)\bm{\nu}$, the Christoffel transform of the system $\bm{\nu}=(\mu_1,\ldots,\mu_{r-1})$ corresponding to the polynomial $\Phi_1(x)\ldots \Phi_s(x)$; 
    \item[\it{iii)}] $P_{\bm{n}}(x)$ is equal to $\Phi_1(x)\ldots \Phi_s(x)$ times the type II multiple orthogonal polynomial with respect to the system $(\Phi_1\ldots\Phi_s)\bm{\nu}$.
    \item[\it{iv)}] $A_{\bm{n}}^{(j)}(x)$ for $1\le j \le r-1$ is the type I multiple orthogonal polynomial with respect to the system $(\Phi_1\ldots\Phi_s)\bm{\nu}$.
\end{itemize}
Furthermore, in the region $\bm{n}\in\bbN^r_{(\bm{K},\infty)}$ with $m_1+\ldots+m_s\le n_r< m_1+\ldots+m_{s+1}$ for some $s\in\bbN$, we have:
\begin{itemize}
    \item[\it{v)}] $a_{\bm{n},j}$ and $b_{\bm{n},j}$ for $1\le j\le r$ are the NNRR coefficients of the system $((\Phi_1\ldots\Phi_s)\bm{\nu},\omega_s)$, where $\omega_s$ is the spectral measure/moment functional corresponding to the Jacobi submatrix $J_r^{(s)}$.
    \item[\it{vi)}] $P_{\bm{n}}(x)$ is equal to $\Phi_1(x)\ldots \Phi_{s-1}(x)$ times the type II multiple orthogonal polynomial with respect to the system $((\Phi_1\ldots\Phi_s)\bm{\nu},\omega_s)$.
    \item[\it{vii)}] $A_{\bm{n}}^{(j)}(x)$ is equal to the type I multiple orthogonal polynomial with respect to the system $((\Phi_1\ldots\Phi_s)\bm{\nu},\omega_s)$.
\end{itemize}
The condition $\delta_{\bm{n},j,l}\ne 0$ that ensures that the CC algorithm in Theorem~\ref{thm:NNCC} does not break down is then equivalent to the perfectness of all the systems $((\Phi_1\ldots\Phi_s)\bm{\nu},\omega_s)$ for all $s\ge 1$. For example, this is easily seen to be true if $(\mu_1,\ldots,\mu_{r-1})\in\mathcal{M}^{r-1}$ is Angelesco and $\sigma(J_r)$ is disjoint from each $\Delta_j$ (the convex hull of $\supp\,(\mu_j)$).  The same holds true if  $(\mu_1,\ldots,\mu_{r-1})\in\mathcal{M}^{r-1}$ is an AT system  (see, e.g., ~\cite[Sect 23.1.2]{Ismail} for the definition) on an interval $I$, and $\sigma(J_r)$ is disjoint from $I$.

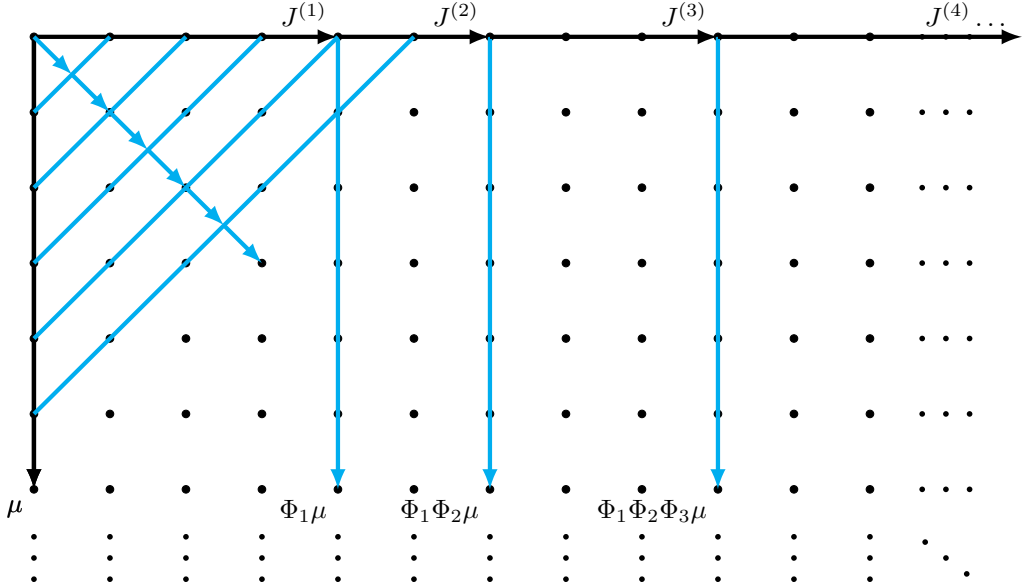
\begin{figure}[ht]
\begin{tikzpicture}

\foreach \x in {0,1,...,11}{
\foreach \y in {0,1,...,6}{
\node[draw,circle,inner sep=1pt,fill] at (\x,\y) {};}}
\draw [ultra thick,-latex,black] (0,6) -- (0,0) node[below left,sloped] {$\mu$};
\draw [ultra thick,-latex,cyan] (4,6) -- (4,0) node[below left,sloped,black]{$\Phi_1\mu$};
\draw [ultra thick,-latex,cyan] (6,6) -- (6,0) node[below left,sloped,black]{$\Phi_1\Phi_2\mu$};
\draw [ultra thick,-latex,cyan] (9,6) -- (9,0) node[below left,sloped,black]{$\Phi_1\Phi_2\Phi_3\mu$};
\draw [ultra thick,-latex,black] (0,6) -- (0,0) node[below left,sloped] {$\mu$};
\draw [ultra thick,-latex,black] (0,6) -- (4,6) node [above left,black]{$J^{(1)}$};
\draw [ultra thick,-latex,black] (4,6) -- (6,6) node [above left,black]{$J^{(2)}$};
\draw [ultra thick,-latex,black] (6,6) -- (9,6) node [above left,black]{$J^{(3)}$};
\draw [ultra thick,-latex,black] (9,6) -- (13,6) node [above left,black]{$J^{(4)}\dots$};
\draw [ultra thick,cyan] (1,6) -- (0,5);
\draw [ultra thick,cyan] (2,6) -- (0,4);
\draw [ultra thick,cyan] (3,6) -- (0,3);
\draw [ultra thick,cyan] (4,6) -- (0,2);
\draw [ultra thick,cyan] (5,6) -- (0,1);
\draw [ultra thick,-latex,cyan] (0,6) -- (0.5,5.5) node{} [above left];
\draw [ultra thick,-latex,cyan] (0.5,5.5) -- (1,5) node{} [above left];
\draw [ultra thick,-latex,cyan] (1,5) -- (1.5,4.5) node{} [above left];
\draw [ultra thick,-latex,cyan] (1.5,4.5) -- (2,4) node{} [above left];
\draw [ultra thick,-latex,cyan] (2,4) -- (2.5,3.5) node{} [above left];
\draw [ultra thick,-latex,cyan] (2.5,3.5) -- (3,3) node{} [above left];
\node[scale=2] at (0,-0.7) {$\vdots$};
\node[scale=2] at (1,-0.7) {$\vdots$};
\node[scale=2] at (2,-0.7) {$\vdots$};
\node[scale=2] at (3,-0.7) {$\vdots$};
\node[scale=2] at (4,-0.7) {$\vdots$};
\node[scale=2] at (5,-0.7) {$\vdots$};
\node[scale=2] at (6,-0.7) {$\vdots$};
\node[scale=2] at (7,-0.7) {$\vdots$};
\node[scale=2] at (8,-0.7) {$\vdots$};
\node[scale=2] at (9,-0.7) {$\vdots$};
\node[scale=2] at (10,-0.7) {$\vdots$};
\node[scale=2] at (11,-0.7) {$\vdots$};
\node[scale=2] at (12,-0.7) {$\ddots$};
\node[scale=2] at (12,0) {$\ldots$};
\node[scale=2] at (12,1) {$\ldots$};
\node[scale=2] at (12,2) {$\ldots$};
\node[scale=2] at (12,3) {$\ldots$};
\node[scale=2] at (12,4) {$\ldots$};
\node[scale=2] at (12,5) {$\ldots$};
\node[scale=2] at (12,6) {$\ldots$};
\end{tikzpicture} 
\caption{Along the blue vertical arrows we find the Christoffel transforms $\Phi_1\mu,(\Phi_1\Phi_2)\mu,(\Phi_1\Phi_2\Phi_3)\mu,\dots$. The Jacobi coefficients of each transform can be computed via the NNCC algorithm.} 
\end{figure}

The simplest example of this construction is to take $r=2$ with an arbitrary $J_1$, the Jacobi matrix of some $\mu_1\in\mathcal{M}$ with $\supp\,(\mu_1)\subseteq\bbR_+$, and $J_2$ to be the zero matrix, i.e., $m_j=1$ with $J_r^{(j)}$ being the $1\times 1$ zero matrices. Then for each $k\in\bbN$,   
%\rkno{$\{a_{(k,n),1}\}_{n}$ and $\{b_{(k,n),1}\}_{n}$} 
{$\{a_{(k,s),1}\}_{k=1}^\infty$ and $\{b_{(k,s),1}\}_{k=0}^\infty$} are then the  Jacobi coefficients of the measure 
%\rkno{$x^k d\mu_1(x)$} 
{$x^s d\mu_1(x)$}. The corresponding polynomials are often called the {\it associated polynomials}. The CC algorithm from above has a close connection to the $qd$-algorithm of Rutishauser~\cite{QD} and the Toda lattice with discrete time.

Similarly, the special case of the given construction with $r>2$ and $m_j=1$ for all $j$ (so that $J_r^{(j)}$ is the diagonal matrix with $(\lambda_1,\lambda_2,\ldots)$ on the diagonal) leads to the multi-dimensional Toda lattice with discrete time (``dm-Toda lattice''), studied in~\cite{ADMVA}. The integrable system requires the perfectness of each $\bm{\nu}, (x-\lambda_1)\bm{\nu}, (x-\lambda_1)(x-\lambda_2)\bm{\nu}, \ldots$, see~\cite[Remark~3.5]{ADMVA}. As was discussed in the current section, perfectness holds true if $\bm{\nu}$ is any Angelesco or any AT system, and each $\lambda_j$ lies outside of the interior of the convex hull of each $\supp\,\mu_l$. This provides a wealth of examples of well-defined integrable systems of~\cite{ADMVA} corresponding to the repeated one-step Christoffel transforms. The same holds true for repeated one-step Geronimus transform  if $\bm{\nu}$ is any Angelesco or any AT system, and each $\lambda_j$ lies outside of the convex hull of each $\supp\,\mu_l$. %Both questions were raised in~\cite[Remark~3.5]{ADMVA}.

%As was explained in this section, if $\bm{\nu}$ is any Angelesco or any AT system, and each $\lambda_j$ ($j\ge 1$) is outside of the convex hull of each $\supp\,\mu_l$, then each MOPRL system $\bm{\nu}, (x-\lambda_1)\bm{\nu}, (x-\lambda_1)(x-\lambda_2)\bm{\nu}, \ldots$ is perfect, and therefore the integrable system is well-defined at each location, 

%We would like to comment that in their Remark~3.5 they posed the question of finding  examples beyond the Laguerre polynomials of the second kind with $\lambda_j\equiv 0$ for all $j$ when this integrable system %all the recurrence coefficients (see our {\it v)} above) are
%is well-defined, or equivalently, when each of the MOPRL systems $\bm{\nu}, (x-\lambda_1)\bm{\nu}, (x-\lambda_1)(x-\lambda_2)\bm{\nu}, \ldots$ is perfect. %They only treat Laguerre polynomials of the second kind with $\lambda_j\equiv 0$ for all $j$.
%As was explained in this section, such perfectness happens quite generally, for example if the initial multiple orthogonality system $\bm{\nu}$ is Angelesco or AT and each $\lambda_j$ ($j\ge 1$) is outside of the convex hull of $\supp\,\mu_l$ for any $1\le l\le r-1$. 

\subsection{Recurrence Coefficients and Normality}\label{ss:recurrNorm}
\hfill\\

%\rk
{As discussed in Section~\ref{ss:NNRR}, every perfect system (in the extended sense of Definition~\ref{def:perf}) $\bm{\mu}\in\mathcal{L}^r$ produces the nearest neighbour recurrence coefficients $\{a_{\bm{n},j}\}_{\bm{n}\in\bbN^r_{\bm{N}}}$  and $\{b_{\bm{n},j}\}_{\bm{n}\in\bbN^r_{\bm{N}-\bm{e}_j}}$ for each $1\le j \le r$ that satisfy the compatibility conditions~\eqref{eq:CC1}, ~\eqref{eq:CC2},~\eqref{eq:CC3}. In this section we %discuss the inverse problem: we %show 
%that this condition (together with vanishing of $a$-coefficients on {\it each} corresponding margin) is 
identify what conditions are necessary and sufficient in order for coefficients $\{a_{\bm{n},j}\}_{\bm{n}\in\bbN^r_{\bm{N}}}$  and $\{b_{\bm{n},j}\}_{\bm{n}\in\bbN^r_{\bm{N}-\bm{e}_j}}$ to be the nearest neighbour recurrence coefficients of some perfect system.

This corresponds to a result central to the paper \cite{Integrable Systems}, solved for $\bm{\mu}\in \mathcal{L}_{\infty}^r$, see also~\cite[Prop.~3]{ADL}. Here we give an alternative simple proof, based on the lemmas in Section \ref{normality critera}, along with the CC algorithm. The special feature that appears when some of the measures/functionals are finitely supported is property {\it b)} below.
}

%\rkno{In this section we look at the CC algorithm from a different perspective. If we are given recurrence coefficients $a_{\bm{n},j}, b_{\bm{n},j}$ satisfying the compatibility conditions with $\delta_{\bm{n}} \neq 0$, we can run the algorithm to compute every recurrence coefficient from the initial data described in Theorem \ref{thm:NNCC}. The initial data describes recurrence coefficients of a uniquely determined set of functionals $\mu_1,\dots,\mu_r \in \mathcal{L}$, by the Favard Theorem. It follows that each set of recurrence coefficients corresponds to a unique system of functionals $(\mu_1,\dots,\mu_r) \in \mathcal{L}^r$. Here we show that the condition $\delta_{\bm{n}} \neq 0$ also guarantees the perfectness of $(\mu_1,\dots,\mu_r)$. Then the recurrence coefficients we started with also corresponds to the recurrence coefficients of $(\mu_1,\dots,\mu_r)$, by Theorem \ref{thm:NNCC}. In other words, we get a generalization of Favard Theorem to perfect systems of multiple functionals/measures. This corresponds to a result central to the paper \cite{Integrable Systems}, at least in the case $\mathcal{L}_{\infty}^r$. Here we give our own proof, showing that all of this follows easily from the simple lemmas in section \ref{normality critera}, along with the CC algorithm. This also extends the result from \cite{Integrable Systems} to our finitely supported perfect systems.}

\begin{thm} \label{thm:perfect systems} Let $\bm{N}=(N_j)_{j=1}^r$, where $N_j\le \infty$.
Suppose we are given the sets of complex coefficients  $\{a_{\bm{n},j}\}_{\bm{n}\in\bbN^r_{\bm{N}}}$  and $\{b_{\bm{n},j}\}_{\bm{n}\in\bbN^r_{\bm{N}-\bm{e}_j}}$ for each $1\le j \le r$.
Assume they satisfy the partial difference equations
~\eqref{eq:CC1}, ~\eqref{eq:CC2},~\eqref{eq:CC3} along with the boundary conditions
\begin{enumerate}[label=\alph*)]
\item $a_{\bm{n},j} = 0$ when $n_j = 0$,
\item $a_{\bm{n},j} = 0$ when $n_j = N_j$.
\end{enumerate}
and the normality conditions
\begin{enumerate}[resume,label=\alph*)]
\item $a_{\bm{n},j} \neq 0$ when $0 < n_j < N_j$,
\item $b_{\bm{n},j} - b_{\bm{n},k} \neq 0$ when $j \neq k$.
\end{enumerate}
Then there is a unique perfect system $\bm{\mu} = (\mu_1,\dots,\mu_r)\in \mathcal{L}^r$ such that $a_{\bm{n},i}$ and $b_{\bm{n},i}$ are the nearest neighbour recurrence coefficients of $\bm{\mu}$ %, $a_{n\bm{e}_j,j}$ and $b_{n\bm{e}_j,j}$ are the Jacobi coefficients of $\mu_j$, and 
with $\mu_j\in\mathcal{L}_{N_j}$ for each $j = 1,\dots,r$.
%$N_j=|\supp\,(\mu_j)|$, for each $j = 1,\dots,r$.
\end{thm}

\begin{rem}
    By induction using \eqref{eq:CC3}, we see that d) can be replaced with $b_{\bm{0},j} - b_{\bm{0},k} \neq 0$ when $j \neq k$. Alternatively, c) can be replaced with $a_{n\bm{e}_j,j} \neq 0$ for each $0<n<N_j$ when $1 \leq j \leq r$.
\end{rem}

\begin{rem}
    If we also assume $a_{n\bm{e}_j,j} > 0$ for each $0<n<N_j$ {and some $1\le j\le r$}, then {additionally $\mu_j\in \mathcal{M}_{N_j}$}.
\end{rem}

\begin{proof}
By the Favard Theorem~\ref{thm:Favard}, we can get functionals $\mu_1,\dots,\mu_r$ such that $a_{n\bm{e}_j,j}$ and $b_{n\bm{e}_j,j}$ are the Jacobi coefficients of $\mu_j\in\mathcal{L}_{N_j}$ %. and $N_j$ is the number of points in $\supp(\mu_j)$, 
for each $j = 1,\dots,r$, by \textit{c)}. We now want to prove that $\bm{\mu} = (\mu_1,\dots,\mu_r)$ is perfect and that $a_{\bm{n},j}$ and $b_{\bm{n},j}$ are the recurrence coefficients of $\bm{\mu}$. We will use induction, to prove that for each $M \in \bbN$,
\begin{enumerate}[label=\alph*)]
    \item[1)] Every index $\bm{n} \leq \bm{N}$ such that $\abs{\bm{n}} \leq M$ is normal.
    \item[2)] The nearest neighbour recurrence coefficients of $\bm{\mu}$, for every index $\bm{n} \leq \bm{N}$ such that $\abs{\bm{n}} \leq M - 1$ are exactly the recurrence coefficients $\set{a_{\bm{n},i},b_{\bm{n},i}}_{i = 1}^{r}$.
\end{enumerate}
This is true when $M = 0$ and $M = 1$, since then we already know which indices of the form $\bm{n} = n_j\bm{e}_j$ are normal, and the nearest neighbour recurrence coefficients of $\bm{n} = \bm{0}$ are $a_{\bm{0},j}$ and $b_{\bm{0},j}$. Now assume 1) and 2) hold for $1,\dots,M$ and suppose $\bm{n} \leq \bm{N}$ and $\abs{\bm{n}} = M + 1$. If $\bm{n} = n_j\bm{e}_j$ then we already know that 1) holds. Otherwise we can write $\bm{n} = \bm{m} + \bm{e}_j + \bm{e}_k$ for some index $\bm{m}$ and $k \neq j$. Since $\abs{\bm{m}} = M - 1$ we know that the nearest neighbour recurrence coefficients for $\bm{m}$ are $a_{\bm{m},i}$ and $b_{\bm{m},i}$. In particular $b_{\bm{m},j} \neq b_{\bm{m},k}$, so that $\bm{n} = \bm{m} + \bm{e}_j + \bm{e}_k$ is normal by Corollary \ref{cor:nnr cor}. Now the conditions of Theorem \ref{thm:compatibility conditions} are satisfied and we can apply CC algorithm to compute the nearest neighbour recurrence coefficients of the system $\bm{\mu}$ for every $\bm{n}$ with $\abs{\bm{n}} = M$. Since our recurrence coefficients given in the Theorem are computed we get 2), and the perfectness of $\bm{\mu}$ follows.
\end{proof}

%\section{Now I'm thinking section on computing recurrence coefficients}

%\subsection{The FHV Algorithm}

%\subsection{Computing Christoffel Transforms}

%\subsection{Favard's Theorem for perfect systems}

\subsection{Zero interlacing for Christoffel transforms: general results}\label{ss:interlacing}
\hfill\\

In this section we prove some new interlacing results for multiple orthogonal polynomials of type I and type II. %that follow from the formulas~\eqref{} and ~\eqref{}. 
%We collect some of the special cases of classical families of multiple orthogonal polynomials for which zero interlacing follows from Corollary~\ref{cor:interlacing} and Corollary~\ref{cor:interlacingGer}.
%Given a real polynomial $p(x)$ we denote its zeros are by $\{ z_j(p) \}$. 
Given two real polynomials $p(x)$ and $q(x)$ we 
write $p(x)\sim q(x)$ and
say that the zeros of polynomials $p(x)$ and $q(x)$  interlace if all the zeros of $p(x)$ and $q(x)$ are pairwise distinct, real, simple, and between every two consecutive zeros of one of the polynomials there lies exactly one zero of the other polynomial. 
%We also employ the notation $p\interl q$ or $q\interr p$ if zeros of $p(x)$ and $q(x)$ %strictly
%interlace and
%$$
%z_1(p) < z_1(q) < z_2(p) < \ldots,
%$$
%where $z_j(p)$ and $z_j(q)$ stand for the $j$-th lowest zero of $p$ and $q$, respectively. Note that this means that either $\deg p = \deg q = n$ with $z_{n}(p)<z_n(q)$ or $\deg p = \deg q+1 = n$ with $z_{n-1}(q)<z_n(p)$. 
%either $\deg p = \deg q = k$ with
%\begin{equation}
%z_1(q)< z_1(p)< z_2(q)<z_2(p)<\ldots<z_{k}(q) < z_{k}(p),   
%\end{equation}
%or $\deg p = \deg q+1 =k+1$ with
%\begin{equation}
%z_1(p)< z_1(q)< z_2(p)<z_2(q)<\ldots<z_{k}(q) < z_{k+1}(p).   
%\end{equation}

Let us call a polynomial real-rooted if all of its zeros are real. Recall (see, e.g.,~\cite{Ismail}) that type II and type I polynomials for any Angelesco system are real-rooted. Type II polynomials for any AT systems are also real-rooted. 

The following result is well-known (this is part of the Hermite--Kakeya--Obreschkoff theorem).

\begin{lem}\label{lem:Hermite}
    Suppose $p$ and $q$ are two polynomials with interlacing zeros $p(x)\sim q(x)$. % and positive leading coefficients.  
    %are two real polynomials with positive leading coefficients such that $p \sim q$ and 
    %and $\deg p = \deg q+1$. 
    Then
    \begin{align}
        & p(x)+\alpha q(x)\sim q(x), \mbox{ for any } \alpha\in\bbR, \\
        & p(x)+\alpha q(x) \sim p(x), \mbox{ for any } \alpha\ne 0.
    \end{align}
    %Then $p(x)+\alpha q(x)\inter q(x)$ for any $\alpha\in\bbR$. Moreover, $p(x)+\alpha q(x) \inter p(x)$ for any $\alpha>0$, and $p(x)\inter p(x)+\alpha q(x)$ for any $\alpha<0$
\end{lem}

\begin{cor}\label{cor:interlacing}
    Suppose $\bm{\nu}$ and $\widehat{\bm{\nu}} = (x-z_0)\bm{\nu}$ are perfect for some $z_0\in\bbR$. 
    %In the setting of Theorem~\ref{thm:ChrMOPRLII}, 
    If ${P}_{\bm{k}+\bm{e}_j}(x)\sim {P}_{\bm{k}}(x)$, then $\widehat{P}_{\bm{k}}$ is real-rooted. Moreover,
    \begin{align}
    \label{eq:newInt1}
        & (x-z_0)\widehat{P}_{\bm{k}}(x) \sim P_{{\bm{k}}}(x) , \\
        \label{eq:newInt2}
        & (x-z_0)\widehat{P}_{\bm{k}}(x) \sim P_{{\bm{k}}+\bm{e}_j}(x).
    \end{align}
    
    %$(x-x_0)\widehat{P}_{\bm{k}}(x) \sim P_{{\bm{k}}}(x) $ and $(x-x_0)\widehat{P}_{\bm{k}}(x) \sim P_{{\bm{k}}+\bm{e}_j}(x)$.
    
    %If ${P}_{\bm{k}+\bm{e}_j}(x)\interl {P}_{\bm{k}}(x)$, then $\widehat{P}_{\bm{k}}$ is real-rooted. Moreover, $(x-x_0)\widehat{P}_{\bm{k}}(x) \interl P_{{\bm{k}}}(x) $ and $(x-x_0)\widehat{P}_{\bm{k}}(x) \sim P_{{\bm{k}}+\bm{e}_j}(x)$.

    %zeros of $(x-x_0)\widehat{P}_{\bm{k}}(x)$ interlace with zeros of $P_{{\bm{k}}+\bm{e}_j}(x)$ and of $P_{{\bm{k}}}(x)$
\end{cor}
\begin{proof}
    Immediate from~\eqref{eq:ChrP1} and Lemma~\ref{lem:Hermite}.
\end{proof}
\begin{rem}
    Let $\bm{\nu}$ be any Angelesco or AT system and suppose $z_0$ is not a zero of any $P_{\bm{k}}(x)$. By Theorem~\ref{thm:ChrMOPRLII}, $\widehat{\bm{\nu}}$ is perfect. Since interlacing for $\bm{\nu}$ holds at every multi-index \cite{Jacobi operator on trees, Ismail,Interlacing}, Corollary~\ref{cor:interlacing} applies for any $z_0\in\bbR$, and we obtain that %$\widehat{\bm{\nu}}$ is perfect, and 
    every $\widehat{P}_{\bm{k}}$ is real-rooted and interlacings ~\eqref{eq:newInt1}--\eqref{eq:newInt2} hold true for all such systems. 
    
    In particular, if $z_0$ belongs to the support of one of the measures, then $\widehat{\bm{\nu}}$ is no longer a system of positive measures, so its perfectness is not trivial in that case.
    %In particular, If $\bm{\nu}$ is an Angelesco, AT, or Nikishin system, then it is well-known that zeros of ${P}_{\bm{k}+\bm{e}_j}(x)$ and ${P}_{\bm{k}}(x)$ interlace for any index $\bm{k}$, so Corollary \ref{cor:interlacing} applies for any of such systems. %In particular this applies to many classical systems of MOPRL, such as Laguerre of the first and the second kind, Jacobi--Hermite, Jacobi--Laguerre, Laguerre--Hermite, Angelesco--Jacobi (Jacobi--Jacobi), Jacobi--Pi\~{n}eiro. This an many more related questions will be addressed in a forthcoming publication.
    %This provides an elementary proof of some of the results from~\cite{Mar}
    %We provide some examples in Section~\ref{ss:interlacing}.
\end{rem}
%\begin{rem}
%    This allows us to present an example of multiple orthogonality system of non-positive measures which is perfect and whose type II orthogonal polynomials are real-rooted. Let $\bm{\nu}$ be any Angelesco system with convex hull of $\supp(\mu_j)$ being $\Delta_j$. Note if $z_0$ is in the interior of $\Delta_{j_0}$ for some $j_0$  then the transformed system $\widehat{\bm{\nu}}$ is no longer a system of positive measures. Theorem~\ref{thm:ChrMOPRLII}  shows that $\widehat{\bm{\nu}}$ is nevertheless perfect (assuming $z_0$ is not a zero of any $P_{\bm{k}}$) and Corollary~\ref{cor:interlacing} says that all the zeros of $\widehat{P}_{\bm{k}}$ are real, $\Delta_{j_0}$ contains at least $n_{j_0}-1$ of them, while each $\Delta_j$, $j\ne j_0$, contains at least $n_j$ of them.
%\end{rem}

\begin{cor}\label{cor:interlacingGer}
    %In the setting of Theorem~\ref{thm:ChrMOPRLII}, 
    Suppose $\bm{\nu}$ and $\widehat{\bm{\nu}} = (x-z_0)\bm{\nu}$ are perfect for some $z_0\in\bbR$.
    %if ${A}^{(j)}_{\bm{k}-\bm{e}_j}(x)\sim {A}^{(j)}_{\bm{k}}(x)$, then $\widecheck{A}^{(j)}_{\bm{k}}(x)\sim A^{(j)}_{\bm{k}}(x)$ and of $\widecheck{A}^{(j)}_{\bm{k}}(x)\sim A^{(j)}_{{\bm{k}-\bm{e}_j}}(x)$.
    If $\widehat{A}^{(j)}_{\bm{k}-\bm{e}_j}(x)\sim \widehat{A}^{(j)}_{\bm{k}}(x)$, then ${A}^{(j)}_{\bm{k}}(x)$ is real-rooted. Moreover, 
    \begin{align}
    \label{eq:newInt3}
        & {A}^{(j)}_{\bm{k}}(x)\sim \widehat{A}^{(j)}_{\bm{k}}(x) , \\
        \label{eq:newInt4}
        & {A}^{(j)}_{\bm{k}}(x)\sim \widehat{A}^{(j)}_{{\bm{k}-\bm{e}_j}}(x).
    \end{align}
\end{cor}
\begin{proof}
    Use~\eqref{eq:nnr cor type 1} and the leftmost equality in~\eqref{eq:type 1 determinantal formula} to get
\begin{equation}\label{eq:typeICC}
    \bm{A}_{\bm{k} } - \widehat{\bm{A}}_{\bm{k} - \bm{e}_j} 
    = \delta \widehat{\bm{A}}_{\bm{k}},
   %= (b_{\bm{k}-\bm{e}_j-\bm{e}_k,k} - b_{\bm{k}-\bm{e}_j-\bm{e}_k,j})A_{\bm{k},m}
\end{equation}
where $\delta = \delta_{(\bm{k}-\bm{e}_j,0),m,j}$ is non-zero by Lemma~\ref{lem:normality vs recurrence} {\it b)}, and~\eqref{eq:CC1}. If $z_0\in\bbR$ then $\delta\in\bbR$, and Lemma~\ref{lem:Hermite} completes the proof.
\end{proof}
\begin{rem}
    %OPTION1:
    If $\bm\nu$ is Angelesco and $z_0$ is outside of the convex hull of each $\mu_j$, then $\widehat{\bm{\nu}}$ is also Angelesco. For these systems, the zeros of $\widehat{A}^{(j)}_{\bm{k}-\bm{e}_j}(x)$ and $\widehat{A}^{(j)}_{\bm{k}}(x)$ interlace for any index $\bm{k}$ and any $j$ (see \cite{DenYat,FMM}), so Corollary \ref{cor:interlacingGer} applies. Similarly, if $\bm{\nu}$ is Nikishin, then so is $\widehat{\bm{\nu}}$. $\widehat{A}^{(j)}_{\bm{k}-\bm{e}_j}(x)$ and $\widehat{A}^{(j)}_{\bm{k}}(x)$ do interlace but not for all the indices $\bm{n}$ (see~\cite{KVInterlacing} for more details), so Corollary \ref{cor:interlacingGer} applies for such indices.
    %
    %OPTION2: Any Angelesco system $\widehat{\bm{\nu}}$ satisfies the property that zeros of $\widehat{A}^{(j)}_{\bm{k}-\bm{e}_j}(x)$ and $\widehat{A}^{(j)}_{\bm{k}}(x)$ interlace for any index $\bm{k}$ and any $j$ (see \cite{DenYat,FMM}), so Corollary \ref{cor:interlacingGer} applies, assuming $\bm{\nu}$ is perfect (in particular, if $z_0$ is outside of the convex hull of each measures of $\widehat{\bm{\nu}}$ then $\bm{\nu}$ is itself Angelesco, and therefore perfect).
\end{rem}

%We intend to extend these results further using in a forthcoming publication.

%when applied to the MOPRL generalizations of the classical orthogonal polynomials, such as Laguerre of the first and the second kind, Jacobi--Hermite, Jacobi--Laguerre, Laguerre--Hermite, Angelesco--Jacobi (Jacobi--Jacobi), Jacobi--Pi\~{n}eiro. We will illustrate this with the Angelesco--Jacobi system only, the remaining cases can be treated similarly. 

\subsection{Zero interlacing: continuous examples}\label{ss:interlacingContinuous}
\hfill\\

While Corollaries~\ref{cor:interlacing} and~\ref{cor:interlacingGer} might seem tame and almost obvious, they allow to obtain interesting interlacing results when applied to the multiple analogues of the classical orthogonal polynomials. We collect them in Sections~\ref{ss:interlacingContinuous} and~\ref{ss:interlacingDiscrete}.

We remark that recurrences~\eqref{eq:ChrP1} and~\eqref{eq:typeICC} (as well as the NNRR for type II and type I polynomials) provide not only interlacing properties but also various relations that the classical MOPRL (as well as classical OPRL, as a special case $r=1$) must satisfy. It is highly likely that these have appeared in the literature in various forms. We focus on interlacing properties only.

\subsubsection{The multiple Laguerre polynomials of the first kind} Let $P_{\bm{n}}^{\bm{\alpha}}$ be the type II multiple orthogonal polynomials corresponding to the system $\bm{\mu}^{\bm{\alpha}}\in\mathcal{M}^r$ given by
$$
d\mu_j^{\bm{\alpha}} = x^{\alpha_j} e^{-x} \chi_{[0,\infty)}(x) dx,
$$
where $\alpha_j>-1$ and $\alpha_j-\alpha_k \notin\bbZ$ for $j\ne k$. Here and everywhere below $\chi_S(x)$ is the characteristic function of a set $S$. %Let $\bm{A}_{\bm{n}}^{\alpha_1,\ldots,\alpha_r}$ be the corresponding type I polynomials.

Since $\bm{\mu}^{\bm{\alpha}}$ is an AT system and $x \bm{\mu}^{\bm{\alpha}} = \bm{\mu}^{\bm{\alpha}+\bm{1}}$ (where $\bm{1}$ is the vector of $1$'s), Corollary~\ref{cor:interlacing} gives
$$
P_{\bm{n}}^{\bm{\alpha}+\bm{1}}(x)  \sim P_{\bm{n}}^{\bm{\alpha}}(x) \quad \mbox{and} \quad P_{\bm{n}}^{\bm{\alpha}+\bm{1}}(x)\sim P_{\bm{n}+\bm{e}_j}^{\bm{\alpha}}(x) 
$$
for any $\bm{n}\in \bbN^r$ and any $1\le j\le r$.

%\begin{thm}
%       For any $\bm{n}\in \bbN^r$ and any $1\le j\le r$: 
%\begin{itemize}
%    \item[i)] $P_{\bm{n}}^{\alpha_1+1,\ldots,\alpha_r+1}(x)  \sim P_{\bm{n}}^{\alpha_1,\ldots,\alpha_r}(x)  $ and $P_{\bm{n}}^{\alpha_1+1,\ldots,\alpha_r+1}(x)\sim P_{\bm{n}+\bm{e}_j}^{\alpha_1,\ldots,\alpha_r}(x) $
%\end{itemize}
%\end{thm}

\subsubsection{The multiple Laguerre polynomials of the second kind} Let $P_{\bm{n}}^{\bm{c};\alpha}$ be the type II multiple orthogonal polynomials corresponding to the system $\bm{\mu}^{\bm{c};\alpha}\in\mathcal{M}^r$ given by
$$
d\mu_j^{\bm{c};\alpha} = x^{\alpha} e^{-c_j x} \chi_{[0,\infty)} (x)dx,
$$
where $\alpha>-1$, $c_j>0$, and $c_j\ne c_k$ for $j\ne k$. %Let $\bm{A}_{\bm{n}}^{\alpha_1,\ldots,\alpha_r}$ be the corresponding type I polynomials.

Since $\bm{\mu}^{\bm{c};\alpha}$ is an AT system and $x \bm{\mu}^{\bm{c};\alpha} = \bm{\mu}^{\bm{c};\alpha+1}$, Corollary~\ref{cor:interlacing} gives
$$
P_{\bm{n}}^{\bm{c};\alpha+1}(x)  \sim P_{\bm{n}}^{\bm{c};\alpha}(x) \quad \mbox{and} \quad P_{\bm{n}}^{\bm{c};\alpha+1}(x)\sim P_{\bm{n}+\bm{e}_j}^{\bm{c};\alpha}(x) 
$$
for any $\bm{n}\in \bbN^r$ and any $1\le j\le r$.

\subsubsection{The Jacobi--Pi\~{n}eiro polynomials} Let $P_{\bm{n}}^{\bm{\alpha};\beta}$ be the type II multiple orthogonal polynomials corresponding to the system $\bm{\mu}^{\bm{\alpha};\beta}\in\mathcal{M}^r$ given by
$$
d\mu_j^{\bm{\alpha};\beta} = x^{\alpha_j} (1-x)^\beta \chi_{[0,1]}(x) dx,
$$
where $\beta>-1$, $\alpha_j>-1$, and $\alpha_j-\alpha_k \notin\bbZ$ for $j\ne k$. %Let $\bm{A}_{\bm{n}}^{\alpha_1,\ldots,\alpha_r}$ be the corresponding type I polynomials.

Since $\bm{\mu}^{\bm{\alpha};\beta}$ is an AT system and $x \bm{\mu}^{\bm{\alpha};\beta} = \bm{\mu}^{\bm{\alpha}+\bm{1};\beta}$ and
$(1-x) \bm{\mu}^{\bm{\alpha};\beta} = \bm{\mu}^{\bm{\alpha};\beta+1}$ 
Corollary~\ref{cor:interlacing} gives
\begin{itemize}
    \item[\it{i)}] $P_{\bm{n}}^{\bm{\alpha}+\bm{1};\beta}(x)  \sim P_{\bm{n}}^{\bm{\alpha};\beta}(x)$ and $P_{\bm{n}}^{\bm{\alpha}+\bm{1};\beta}(x)\sim P_{\bm{n}+\bm{e}_j}^{\bm{\alpha};\beta}(x)$;
    \item[\it{ii)}] $P_{\bm{n}}^{\bm{\alpha};\beta+1}(x)  \sim P_{\bm{n}}^{\bm{\alpha};\beta}(x)$ and $P_{\bm{n}}^{\bm{\alpha};\beta+1}(x)\sim P_{\bm{n}+\bm{e}_j}^{\bm{\alpha};\beta}(x)$,
\end{itemize}
% \begin{align*}
% & P_{\bm{n}}^{\bm{\alpha}+\bm{1};\beta}(x)  \sim P_{\bm{n}}^{\bm{\alpha};\beta}(x)  &\mbox{and }  \qquad
% & P_{\bm{n}}^{\bm{\alpha}+\bm{1};\beta}(x)\sim P_{\bm{n}+\bm{e}_j}^{\bm{\alpha};\beta}(x),
% \\
% & P_{\bm{n}}^{\bm{\alpha};\beta+1}(x)  \sim P_{\bm{n}}^{\bm{\alpha};\beta}(x)  &\mbox{and }  \qquad 
% & P_{\bm{n}}^{\bm{\alpha};\beta+1}(x)\sim P_{\bm{n}+\bm{e}_j}^{\bm{\alpha};\beta}(x)
% \end{align*}
for any $\bm{n}\in \bbN^r$ and any $1\le j\le r$.

\subsubsection{The Angelesco--Jacobi polynomials}
These are the type II multiple orthogonal polynomials $P_{\bm{n}}^{\alpha,\beta,\gamma}$  corresponding to the system $\bm{\mu}^{\alpha,\beta,\gamma}\in\mathcal{M}^2$ 
%$(\mu^{\alpha,\beta,\gamma}_1,\mu^{\alpha,\beta,\gamma}_2)\in \mathcal{M}^2$  
where 
\begin{align}
    d\mu^{\alpha,\beta,\gamma}_1 & = %\tfrac{1}{Z_1}
    (1-x)^\alpha (x-a)^\beta |x|^\gamma \chi_{[a,0]}(x) \, dx, \\
    d\mu^{\alpha,\beta,\gamma}_2 & = %\tfrac{1}{Z_2}
    (1-x)^\alpha (x-a)^\beta |x|^\gamma \chi_{[0,1]}(x) \, dx,
\end{align}
where  $\alpha,\beta,\gamma>-1$, $a<0$. %, and $Z_1,Z_2$ are the respective normalization constants. 
Denote $\bm{A}_{\bm{n}}^{\alpha,\beta,\gamma}= ({A}_{\bm{n},1}^{\alpha,\beta,\gamma},{A}_{\bm{n},2}^{\alpha,\beta,\gamma})$ to be the corresponding type I polynomials.

Since $\bm{\mu}^{\alpha,\beta,\gamma}\in\mathcal{M}^2$  is an Angelesco system, Corollary~\ref{cor:interlacing} 
and Corollary~\ref{cor:interlacingGer} can be applied with any choice of $z_0$ away from the zeros of type II multiple orthogonal polynomials which are known to all belong to $(a,0)\cup (0,1)$. Note that in our notation, $(x-1)\bm{\mu}^{\alpha,\beta,\gamma} = \bm{\mu}^{\alpha+1,\beta,\gamma}$, $(x-a)\bm{\mu}^{\alpha,\beta,\gamma} = \bm{\mu}^{\alpha,\beta+1,\gamma}$, and $x\bm{\mu}^{\alpha,\beta,\gamma} = \bm{\mu}^{\alpha,\beta,\gamma+1}$ up to a trivial multiplicative normalization. %[Mention Toda structure that the corresponding NNR coefficients form?]
%In particular Corollary~\ref{cor:interlacing} and Corollary~\ref{cor:interlacingGer} imply the following interlacing properties for any $\bm{n}\in \bbN^2$ and any $j=1,2$: 
%\begin{thm}\label{thm:interlacing1}
Then
\begin{itemize}
    \item[\it{i)}] $P_{\bm{n}}^{\alpha+1,\beta,\gamma}(x)  \sim P_{\bm{n}}^{\alpha,\beta,\gamma} (x)  $ and $P_{\bm{n}}^{\alpha+1,\beta,\gamma} (x) \sim P^{\alpha,\beta,\gamma}_{{\bm{n}}+\bm{e}_j}(x)$;
    \item[\it{ii)}] $P_{\bm{n}}^{\alpha,\beta+1,\gamma} (x) \sim P_{\bm{n}}^{\alpha,\beta,\gamma} (x)  $ and $P_{\bm{n}}^{\alpha,\beta+1,\gamma} (x) \sim P^{\alpha,\beta,\gamma}_{{\bm{n}}+\bm{e}_j}(x)$;
    \item[\it{iii)}] $x P_{\bm{n}}^{\alpha,\beta,\gamma+1}(x) \sim P_{\bm{n}}^{\alpha,\beta,\gamma} (x)  $ and $x P_{\bm{n}}^{\alpha,\beta,\gamma+1}(x)  \sim P^{\alpha,\beta,\gamma}_{{\bm{n}}+\bm{e}_j}(x)$;
    \item[\it{iv)}] $A_{\bm{n},j}^{\alpha+1,\beta,\gamma}(x)  \sim A_{\bm{n},j}^{\alpha,\beta,\gamma} (x)  $ and $A_{\bm{n},j}^{\alpha+1,\beta,\gamma} (x) \sim A^{\alpha,\beta,\gamma}_{{\bm{n}}+\bm{e}_j,j}(x)$;
    \item[\it{v)}] $A_{\bm{n},j}^{\alpha,\beta+1,\gamma} (x) \sim A_{\bm{n},j}^{\alpha,\beta,\gamma} (x)  $ and $A_{\bm{n},j}^{\alpha,\beta+1,\gamma} (x) \sim A^{\alpha,\beta,\gamma}_{{\bm{n}}+\bm{e}_j,j}(x)$;
    \item[\it{vi)}] $A_{\bm{n},j}^{\alpha,\beta,\gamma+1}(x) \sim A_{\bm{n},j}^{\alpha,\beta,\gamma} (x)  $ and $A_{\bm{n},j}^{\alpha,\beta,\gamma+1}(x)  \sim A^{\alpha,\beta,\gamma}_{{\bm{n}}+\bm{e}_j,j}(x)$,
\end{itemize}
for any $\bm{n}\in \bbN^2$ and any $j=1,2$.

Note that  {i)}--{iii)} follows from Corollary~\ref{cor:interlacing}. For {i)} and {ii)} the extra zero at $a$ or at $1$ does not matter since all the zeros of type II polynomials belong to $(a,1)$.  {iv)}--{vi)} is immediate from Corollary~\ref{cor:interlacingGer}.

\subsubsection{The Jacobi--Laguerre polynomials}
These are the type II multiple orthogonal polynomials $P_{\bm{n}}^{\beta,\gamma}$ corresponding to the system $\bm{\mu}^{\beta,\gamma}\in\mathcal{M}^2$ 
%$(\mu^{\alpha,\beta,\gamma}_1,\mu^{\alpha,\beta,\gamma}_2)\in \mathcal{M}^2$  
where 
\begin{align}
    d\mu^{\beta,\gamma}_1 & = %\tfrac{1}{Z_1}
    (x-a)^\beta |x|^\gamma e^{-x}\chi_{[a,0]}(x) \, dx, \\
    d\mu^{\beta,\gamma}_2 & = %\tfrac{1}{Z_2}
    (x-a)^\beta |x|^\gamma e^{-x}\chi_{[0,\infty)}(x) \, dx,
\end{align}
where  $\beta,\gamma>-1$, $a<0$. %, and $Z_1,Z_2$ are the respective normalization constants. 
Denote $\bm{A}_{\bm{n}}^{\beta,\gamma}= ({A}_{\bm{n},1}^{\beta,\gamma},{A}_{\bm{n},2}^{\beta,\gamma})$ to be the corresponding type I polynomials.

Since $\bm{\mu}^{\beta,\gamma}\in\mathcal{M}^2$  is an Angelesco system,  and $(x-a)\bm{\mu}^{\beta,\gamma} = \bm{\mu}^{\beta+1,\gamma}$, $x\bm{\mu}^{\beta,\gamma} = \bm{\mu}^{\beta,\gamma+1}$ up to a trivial multiplicative normalization, we get as in the previous section: 
\begin{itemize}
    \item[\it{i)}] $P_{\bm{n}}^{\beta+1,\gamma} (x) \sim P_{\bm{n}}^{\beta,\gamma} (x)  $ and $P_{\bm{n}}^{\beta+1,\gamma} (x) \sim P^{\beta,\gamma}_{{\bm{n}}+\bm{e}_j}(x)$;
    \item[\it{ii)}] $x P_{\bm{n}}^{\beta,\gamma+1}(x) \sim P_{\bm{n}}^{\beta,\gamma} (x)  $ and $x P_{\bm{n}}^{\beta,\gamma+1}(x)  \sim P^{\beta,\gamma}_{{\bm{n}}+\bm{e}_j}(x)$;
    \item[\it{iii)}] $A_{\bm{n},j}^{\beta+1,\gamma} (x) \sim A_{\bm{n},j}^{\beta,\gamma} (x)$ and $A_{\bm{n},j}^{\beta+1,\gamma} (x) \sim A^{\beta,\gamma}_{{\bm{n}}+\bm{e}_j,j}(x)$;
    \item[\it{iv)}] $A_{\bm{n},j}^{\beta,\gamma+1}(x) \sim A_{\bm{n},j}^{\beta,\gamma}(x)$ and $A_{\bm{n},j}^{\beta,\gamma+1}(x)  \sim A^{\beta,\gamma}_{{\bm{n}}+\bm{e}_j,j}(x)$.
\end{itemize}
for any $\bm{n}\in \bbN^2$ and any $j=1,2$.

\subsubsection{The Jacobi--Hermite polynomials}
These are the type II multiple orthogonal polynomials $P_{\bm{n}}^{\gamma}$ corresponding to the system $\bm{\mu}^{\gamma}\in\mathcal{M}^2$ 
%$(\mu^{\alpha,\beta,\gamma}_1,\mu^{\alpha,\beta,\gamma}_2)\in \mathcal{M}^2$  
given by
%$(\mu^{\alpha,\beta,\gamma}_1,\mu^{\alpha,\beta,\gamma}_2)\in \mathcal{M}^2$  
\begin{align}
    d\mu^{\gamma}_1 & = %\tfrac{1}{Z_1}
    |x|^\gamma e^{-x^2/2}\chi_{(-\infty,0]}(x) \, dx, \\
    d\mu^{\gamma}_2 & = %\tfrac{1}{Z_2}
    x^\gamma e^{-x^2/2}\chi_{[0,\infty)}(x) \, dx,
\end{align}
where  $\gamma>-1$. Since $\bm{\mu}^{\gamma}\in\mathcal{M}^2$  is an Angelesco system,  and $x\bm{\mu}^{\gamma} = \bm{\mu}^{\gamma+1}$ up to a trivial multiplicative normalization, we get: 
\begin{itemize}
    \item[\it{i)}] $x P_{\bm{n}}^{\gamma+1}(x) \sim P_{\bm{n}}^{\gamma} (x)$ and $x P_{\bm{n}}^{\gamma+1}(x)  \sim P^{\gamma}_{{\bm{n}}+\bm{e}_j}(x)$;
    \item[\it{ii)}] $A_{\bm{n},j}^{\gamma+1}(x) \sim A_{\bm{n},j}^{\gamma}(x)$ and $A_{\bm{n},j}^{\gamma+1}(x)  \sim A^{\gamma}_{{\bm{n}}+\bm{e}_j,j}(x)$.
\end{itemize}
for any $\bm{n}\in \bbN^2$ and any $j=1,2$.

\smallskip

Zero interlacing  for type II Angelesco--Jacobi, Jacobi--Laguerre, and Jacobi--Hermite polynomials were proved by Mart\'{i}nez-Finkelshtein--Morales in their recent~\cite[Thm~2.2]{MarMor24}   with more involved arguments and for  $\bm{n}$ along the stepline (see also~\cite[Lem~2.1]{dosSan}). % [Check more references]

\subsection{Zero interlacing: discrete examples}\label{ss:interlacingDiscrete}

\subsubsection{The multiple Charlier polynomials}\label{ss:Charlier}
These are the type II multiple orthogonal polynomials $P_{\bm{n}}^{\bm{a}}$ corresponding to the discrete system $\bm{\mu}^{\bm{a}}\in\mathcal{M}^r$ given by
$$
\mu_j^{\bm{a}} = \sum_{k=0}^\infty \frac{a_j^k}{k!} \delta_{k},
$$
where $a_j>0$, and $\alpha_j \ne \alpha_s$ for $j\ne s$. 

Recall that $\bm{\mu}^{\bm{\alpha}}$ is an AT system. Now note that $x \bm{\mu}^{\bm{a}}$ is supported on $\{k\in\bbZ: k\ge 1\}$ and its weight at $\{k\}$ is $\frac{a_j^k}{k!} k = a_j \frac{a_j^{k-1}}{(k-1)!}$. This means that  up to an inconsequential multiplicative normalization, $x \bm{\mu}^{\bm{a}}$ shifted by $1$ to the left coincides with $\bm{\mu}^{\bm{a}}$ itself. Corollary~\ref{cor:interlacing} then gives 
\begin{equation}\label{eq:self-interlacing}
P_{\bm{n}}^{\bm{a}}(x-1)  \sim P_{\bm{n}}^{\bm{a}}(x) \quad \mbox{and} \quad xP_{\bm{n}}^{\bm{a}}(x-1)\sim P_{\bm{n}+\bm{e}_j}^{\bm{a}}(x) 
\end{equation}
for any $\bm{n}\in \bbN^r$ and any $1\le j\le r$.

The self-interlacing property $xP_{\bm{n}}^{\bm{a}}(x-1)  \sim P_{\bm{n}}^{\bm{a}}(x)$ implies a number of interesting properties, see~\cite[Sect~8.7]{Fisk}. In particular, it is easy to see that~\eqref{eq:self-interlacing} implies that the distance between any two roots of $P_{\bm{n}}^{\bm{a}}$ is at least 1. %(which is well-known for the scalar Charlier polynomials). 
In Theorem~\ref{thm:minDistance} below we provide a more general statement.

\subsubsection{The multiple Meixner polynomials of the first kind}
These are the type II multiple orthogonal polynomials $P_{\bm{n}}^{\bm{c};\beta}$ corresponding to the discrete system $\bm{\mu}^{\bm{c};\beta}\in\mathcal{M}^r$ given by
$$
\mu_j^{\bm{c};\beta} = \sum_{k=0}^\infty \frac{(\beta)_k c_j^k}{k!} \delta_{k},
$$
where $\beta>0$, $0<c_j<1$, and $c_j \ne c_s$ for $j\ne s$. Here $(x)_0=1$, $(x)_k =x(x+1)\ldots(x+k-1)$ is the Pochhammer symbol.

Again, $\bm{\mu}^{\bm{c};\beta}$ is an AT system, and $x \bm{\mu}^{\bm{c};\beta}$ is supported on $\{k\in\bbZ: k\ge 1\}$ with the weight at $\{k\}$ ($k\ge 1$) being $\frac{(\beta)_k c_j^k}{k!} k = \beta c_j \frac{(\beta+1)_{k-1} c_j^{k-1}}{(k-1)!}$. This means that  up to an inconsequential multiplicative normalization, $x \bm{\mu}^{\bm{c};\beta}$ shifted by $1$ to the left coincides with $\bm{\mu}^{\bm{c};\beta+1}$. Corollary~\ref{cor:interlacing} then gives 
$$
P_{\bm{n}}^{\bm{c};\beta+1}(x-1)  \sim P_{\bm{n}}^{\bm{c};\beta}(x) \quad \mbox{and} \quad P_{\bm{n}}^{\bm{c};\beta+1}(x-1)\sim P_{\bm{n}+\bm{e}_j}^{\bm{c};\beta}(x) 
$$
for any $\bm{n}\in \bbN^r$ and any $1\le j\le r$.

\subsubsection{The multiple Meixner polynomials of the second kind}
These are the type II multiple orthogonal polynomials $P_{\bm{n}}^{c;\bm{\beta}}$ corresponding to the discrete system $\bm{\mu}^{c;\bm{\beta}}\in\mathcal{M}^r$ given by
$$
\mu_j^{c;\bm{\beta}} = \sum_{k=0}^\infty \frac{(\beta_j)_k c^k}{k!} \delta_{k},
$$
where $\beta_j>0$, $0<c<1$, and $\beta_j-\beta_s \notin\bbZ$ for $j\ne s$. 

$\bm{\mu}^{c;\bm{\beta}}$ is an AT system, and up to an inconsequential multiplicative normalization, $x \bm{\mu}^{c;\bm{\beta}}$ shifted by $1$ to the left coincides with $\bm{\mu}^{c;\bm{\beta}+\bm{1}}$. Corollary~\ref{cor:interlacing} then gives 
$$
P_{\bm{n}}^{c;\bm{\beta}+\bm{1}}(x-1)  \sim P_{\bm{n}}^{c;\bm{\beta}}(x) \quad \mbox{and} \quad P_{\bm{n}}^{c;\bm{\beta}+\bm{1}}(x-1)\sim P_{\bm{n}+\bm{e}_j}^{c;\bm{\beta}}(x) 
$$
for any $\bm{n}\in \bbN^r$ and any $1\le j\le r$.

\subsubsection{The multiple Krawtchouk polynomials}
These are the type II multiple orthogonal polynomials $P_{\bm{n}}^{N;\bm{p}}$ corresponding to the discrete finite system $\bm{\mu}^{N;\bm{p}}\in\mathcal{M}^r$ given by
$$
\mu_j^{N;\bm{p}} = \sum_{k=0}^N \binom{N}{k} p_j^k (1-p_j)^{N-k} \delta_{k},
$$
where $N\in\nat$, $0<p_j<1$, and $p_j-p_s \ne 0$ for $j\ne s$. 

$\bm{\mu}^{N;\bm{p}}$ is a discrete AT system (see~\cite{Discrete moprl}), and up to an inconsequential multiplicative normalization, $x \bm{\mu}^{N;\bm{p}}$ shifted by $1$ to the left coincides with $\bm{\mu}^{N-1;\bm{p}}$. Corollary~\ref{cor:interlacing} then gives 
$$
P_{\bm{n}}^{N-1;\bm{p}}(x-1)  \sim P_{\bm{n}}^{N;\bm{p}}(x) \quad \mbox{and} \quad P_{\bm{n}}^{N-1;\bm{p}}(x-1)\sim P_{\bm{n}+\bm{e}_j}^{N;\bm{p}}(x) 
$$
for any $\bm{n}\in \bbN^r$ with $|\bm{n}|<N$ and any $1\le j\le r$.

Similarly, up to an inconsequential multiplicative normalization, $(N-x) \bm{\mu}^{N;\bm{p}}$ coincides with $\bm{\mu}^{N-1;\bm{p}}$. Corollary~\ref{cor:interlacing} then gives 
$$
P_{\bm{n}}^{N-1;\bm{p}}(x)  \sim P_{\bm{n}}^{N;\bm{p}}(x) \quad \mbox{and} \quad P_{\bm{n}}^{N-1;\bm{p}}(x)\sim P_{\bm{n}+\bm{e}_j}^{N;\bm{p}}(x) 
$$
for any $\bm{n}\in \bbN^r$ with $|\bm{n}|<N$ and any $1\le j\le r$.

\subsubsection{The multiple Hahn polynomials}
These are the type II multiple orthogonal polynomials $P_{\bm{n}}^{\bm{\alpha};\beta;N}$ corresponding to the discrete finite system $\bm{\mu}^{\bm{\alpha};\beta;N}\in\mathcal{M}^r$ given by
$$
\mu_j^{\bm{\alpha};\beta;N} = \sum_{k=0}^N \frac{(\alpha_j+1)_k (\beta+1)_{N-k}}{k! (N-k)!} \delta_{k}
$$
where $N\in\nat$, $\beta>-1$, $\alpha_j>-1$, and $\alpha_j-\alpha_s \ne 0$ for $j\ne s$. 

Again, $\bm{\mu}^{\bm{\alpha};\beta;N}$ is a discrete AT system (see~\cite{Discrete moprl}), and up to an inconsequential multiplicative normalization, $x \bm{\mu}^{\bm{\alpha};\beta;N}$ shifted by $1$ to the left coincides with $\bm{\mu}^{\bm{\alpha};\beta;N-1}$, and $(N-x) \bm{\mu}^{\bm{\alpha};\beta;N}$ coincides with $\bm{\mu}^{\bm{\alpha};\beta;N-1}$. Corollary~\ref{cor:interlacing} then gives 
\begin{equation}\label{eq:Hahn1}
P_{\bm{n}}^{\bm{\alpha};\beta;N-1}(x-1)  \sim P_{\bm{n}}^{\bm{\alpha};\beta;N}(x) \quad \mbox{and} \quad P_{\bm{n}}^{\bm{\alpha};\beta;N-1}(x-1)\sim P_{\bm{n}+\bm{e}_j}^{\bm{\alpha};\beta;N}(x) 
\end{equation}
and
\begin{equation}\label{eq:Hahn2}
P_{\bm{n}}^{\bm{\alpha};\beta;N-1}(x)  \sim P_{\bm{n}}^{\bm{\alpha};\beta;N}(x) \quad \mbox{and} \quad P_{\bm{n}}^{\bm{\alpha};\beta;N-1}(x)\sim P_{\bm{n}+\bm{e}_j}^{\bm{\alpha};\beta;N}(x) 
\end{equation}
for any $\bm{n}\in \bbN^r$ with $|\bm{n}|<N$ and any $1\le j\le r$.

\subsection{Minimal distance between roots}\label{ss:Minimal}
\hfill\\

In Section~\ref{ss:Charlier} we observed that the minimal distance between consecutive roots (sometimes called the {\it mesh}) of any type II multiple Charlier polynomial is larger than 1. In the next theorem we generalize this to all the other discrete multiple orthogonal polynomials discussed here. This also serves as a new elementary proof for the case of one measure $r=1$, for which this property has been well known, see~\cite{ChiSta,Lev67,KraZar}.

\begin{thm}\label{thm:minDistance}
    The minimal distance between roots of any type II multiple Hahn (for $|\bm{n}|<N$), multiple Krawtchouk (for $|\bm{n}|<N$), or multiple Charlier polynomial is larger than 1. The minimal distance between roots of any type II multiple Meixner polynomial of the first or second kind is not smaller than 1. 
\end{thm}
\begin{proof}
    For the Charlier case this is clear from~\eqref{eq:self-interlacing}. 
    
    Consider the multiple Hahn case. 
    Since all the zeros of type II polynomials belong to the interval $[0,N]$ (AT systems), then interlacing in~\eqref{eq:Hahn1} and ~\eqref{eq:Hahn2} can be strengthened to $xP_{\bm{n}}^{\bm{\alpha};\beta;N-1}(x-1) 
    \sim
    %\interl 
    P_{\bm{n}+\bm{e}_j}^{\bm{\alpha};\beta;N}(x)$     and 
    $(N-x) P_{\bm{n}}^{\bm{\alpha};\beta;N-1}(x-1) 
    \sim
    %\interr 
    P_{\bm{n}+\bm{e}_j}^{\bm{\alpha};\beta;N}(x)$.
    %Using interlacing, we get ${P^{\bm{\alpha;\beta;N}}_{\bm{n}+\bm{e}_j}(0)}/{P^{\bm{\alpha;\beta;N}}_{\bm{n}}(0)}<0$  and ${P^{\bm{\alpha;\beta;N}}_{\bm{n}+\bm{e}_j}(N)}/{P^{\bm{\alpha;\beta;N}}_{\bm{n}}(N)}<0$ for any $\bm{n}$. Then \eqref{eq:ChrP1} together with     Lemma~\ref{lem:Hermite} allows to improve ~\eqref{eq:Hahn1} and \eqref{eq:Hahn2} to a stronger statement $xP_{\bm{n}}^{\bm{\alpha};\beta;N-1}(x-1) \interl P_{\bm{n}+\bm{e}_j}^{\bm{\alpha};\beta;N}(x)$     and $(N-x) P_{\bm{n}}^{\bm{\alpha};\beta;N-1}(x-1) \interr P_{\bm{n}+\bm{e}_j}^{\bm{\alpha};\beta;N}(x)$.

    This implies that below the first zero of $P_{\bm{n}+\bm{e}_j}^{\bm{\alpha};\beta;N}(x)$ there is no zero of either $P_{\bm{n}}^{\bm{\alpha};\beta;N-1}(x-1)$ nor $P_{\bm{n}}^{\bm{\alpha};\beta;N-1}(x)$. Then between the first and the second zero of $P_{\bm{n}+\bm{e}_j}^{\bm{\alpha};\beta;N}(x)$ there is exactly one zero of $P_{\bm{n}}^{\bm{\alpha};\beta;N-1}(x-1)$ and of $P_{\bm{n}}^{\bm{\alpha};\beta;N-1}(x)$. This proves that the distance between the first two zeros of $P_{\bm{n}+\bm{e}_j}^{\bm{\alpha};\beta;N}(x)$ is larger than one. An easy inductive argument can be used to complete the proof.

    The proof for the multiple Krawtchouk polynomials is identical.

    For the Meixner of the first and second kind, recall that they can be obtained from the Hahn polynomials by taking a certain limit with respect to the coefficients $\bm{\alpha}, \beta, N$, see~\cite{BCVA}. Since the minimal distance for the Hahn polynomials is $>1$, we get $\ge 1$ in the limit. 
\end{proof}

\bibsection

\vspace{0.25cm}

\end{document}